\documentclass[12pt,a4paper]{amsart}
\usepackage[left=2.5cm, right=2.5cm, top=2.8cm, bottom=2.8cm]{geometry}
\usepackage{tikz}
\usepackage{hyperref}
\usepackage{mathrsfs}
\numberwithin{equation}{section}
\usepackage{lmodern}			
\usepackage[T1]{fontenc}		
\usepackage[utf8]{inputenc}		
\usepackage{indentfirst}		
\usepackage{color,xcolor}		
\usepackage{graphicx}			
\usepackage{microtype} 			
 \usepackage{amsbsy}			
 \usepackage{amssymb}			
\usepackage{amsthm}			
\usepackage{cancel, bbding}
\usepackage{tikz}
\usepackage{xcolor}
\usepackage{color}
\usepackage{float}
\usetikzlibrary{arrows}
\usepackage{ragged2e}
\usepackage{calculator}

\usepackage{lineno}

\newtheorem{theorem}{Theorem}[section]
\newtheorem{corollary}[theorem]{Corollary}
\newtheorem{lemma}[theorem]{Lemma}
\newtheorem{proposition}[theorem]{Proposition}
\newtheorem{remark}[theorem]{Remark}
\newtheorem{definition}[theorem]{Definition}
\newtheorem{example}[theorem]{Example}
\newcommand{\T}{{\mathcal T}}

\newcommand{\U}{\mathcal{U}}

\title[Shadowing for Infinite Dimensional Dynamical Systems] 
{Shadowing for Infinite Dimensional Dynamical Systems}

\author[J. M. Arrieta]{José. M. Arrieta}
\address[J. M. Arrieta]{Dept. An\'alisis Matem\'atico y Matem\'atica Aplicada, Universidad Complutense de Madrid 28040 Madrid, Spain   and  Instituto de Ciencias Matem\'aticas ICMAT-CSIC-UAM-UC3M-UCM, Madrid, Spain.}
 \email{jarrieta@ucm.es}

\author[A. N. Carvalho]{Alexandre N. Carvalho}
\address[A. N. Carvalho]{Universidade de S\~{a}o Paulo\\ Instituto de Ci\^encias Matem\'aticas e de Computa\c{c}\~ao-ICMC\\ 13566-590 - S\~{a}o Carlos - S\~ao Paulo - Brazil.}
\email{andcarva@icmc.usp.br}

\author[C. R. Takaessu Jr]{Carlos R. Takaessu Jr}
\address[C. R. Takaessu Jr]{Universidade de S\~{a}o Paulo\\ Instituto de Ci\^encias Matem\'aticas e de Computa\c{c}\~ao-ICMC\\ 13566-590 - S\~{a}o Carlos - S\~ao Paulo - Brazil and Dept. An\'alisis Matem\'atico y Matem\'atica Aplicada, Universidad Complutense de Madrid 28040- Madrid - Madrid - Spain.}
 \email{carlostakaessujr@gmail.com}

\date{\today}

\begin{document}
\allowdisplaybreaks 
\maketitle

\begin{abstract} In this paper we extend to an infinite dimensional setting some results on the Shadowing property that are known on finite dimensional compact manifolds without border and in $\mathbb{R}^n$. In fact, we show that if $\{\T(t):t\ge 0\}$ is a Morse-Smale semigroup defined in a Hilbert space,   with global attractor $\mathcal{A}$ and non-wandering set given by its equilibria, then $\T(1)|_{\mathcal{A}}:\mathcal{A}\to \mathcal{A} $ admits the Lipschitz Shadowing property. Moreover, for any positively invariant bounded neighborhood $\U\supset\mathcal{A}$ of the global attractor, the map $\T(1)|_{\U}:\U\to \U$ has the H\"{o}lder-Shadowing property. We obtain  results related to the structural stability of Morse-Smale semigroups, that were only known on finite dimension,  and  continuity of global attractors.



\end{abstract}

\maketitle

\section{Introduction and statement of the results} 

The theory of dynamical systems (or semigroups) has been constantly developed and has proven to be a powerful tool for studying the asymptotic behavior of solutions of differential equations, both in finite \cite{smale1963stable,PEIXOTO1967214,palis2012geometric,viana2016foundations,shub2013global} and infinite dimensional spaces \cite{hale2006dynamics,da2010dirichlet,carvalho2011exponential,Bortolan-Carvalho-Langa-book,carvalho1998boundary,bortolan2014structure,carvalho2008regularity,carvalho2009local,temam2012infinite}. Very often the phase space of the dynamical system is a Hilbert space \cite{arrieta1992damped,rozendaal2019optimal,arrieta2000abstract,carvalho2012attractors,Bortolan-Carvalho-Langa-book,carvalho2008regularity,hale2006dynamics,tan1993asymptotic,arrieta2017distance,ArrietaEsperanza}. 

One of the most important concepts related to
  infinite dimensional dynamics is the global attractor $\mathcal{A}$ \cite{hale2006dynamics,hale2010asymptotic,henry2006geometric,Bortolan-Carvalho-Langa-book,carvalho2008regularity,aragao2013non,carvalho2011exponential,raugel2002global,rosa1998global,hoang2015continuity}, since this special set dictates the whole asymptotic dynamics of the semigroup. Therefore we are very interested in analyzing the behavior of the orbits inside $\mathcal{A}$ and on its neighborhood. In fact, a large part of the studies related to infinite dimensional dynamical systems are based on studying the existence of the global attractor, to understand its structure and to analyze the dynamics inside $\mathcal{A}$.

The behavior of the attractor  $\mathcal{A}$ under small perturbations is also an important topic in the theory of dynamical systems \cite{bortolan2014structure,arrieta2017distance,Arrieta-Esperanza-2014,carvalho2012attractors,hale1989lower,pilyugin2010lipschitz,raugel2002global, ArrietaEsperanza}. For example, the structural stability of dynamical systems and continuity of global attractors are important topics of the theory. One of the  tools that allows us to study the relation between perturbed and non-perturbed attractors is the Shadowing property \cite{pilyugin2010lipschitz,arrieta2017distance}.

The Shadowing property (see Definition \ref{defshadoww}) is a well known tool in the theory of dynamical systems and has been intensively studied over the years \cite{Pilyugin,pilyugin2010lipschitz,lee2021diffeomorphisms,palmer2012lipschitz,pilyugin2017shadowing,henry1994exponential,gan2002generalized,chow1989shadowing,santamaria2014distance}. Shadowing has a lot of important consequences on finite dimensional dynamical systems and also on applied sciences, such as numerical analysis \cite{Pilyugin}. Moreover, Shadowing has also been studied on semigroups defined in infinite dimensional phase spaces  \cite{Pilyugin,henry1994exponential,chow1989shadowing,li2003chaos}.

It is known that if $\{\mathcal{T}(t):t\ge 0\}$ is a  Morse-Smale semigroup, see Definition \ref{definitionms} below and \cite{kupka1963contributiona,smale1963stable,PEIXOTO1967214,palis1969morse,hale2006dynamics,bortolan2020lipschitz,Bortolan-Carvalho-Langa-book,bortolan2022nonautonomous}, defined in a finite dimensional smooth compact manifold  (without border), then $\mathcal{T}(1)$ has the Lipschitz Shadowing property \cite{Pilyugin}. By compactification, the same holds if the phase space is $\mathbb{R}^n$ \cite{santamaria2014distance} and in this case we obtain Lipschitz Shadowing in a neighborhood of the global attractor $\mathcal{A}\subset\mathbb{R}^n$. Moreover, if a Morse-Smale semigroup (possibly defined in an infinite dimensional space) has an inertial manifold \cite{sell2013dynamics,robinson2002computing,foias1988inertial,jolly1989explicit,robinson2001infinite}, which is a finite dimensional invariant manifold that attracts each point of the phase space exponentially, then it is also possible to obtain Lipschitz Shadowing in a neighborhood of its attractor. In fact, the existence of  an inertial manifold $\mathcal{M}$ allow us to reduce the dynamics on a finite dimensional manifold  so we can apply the known results of Shadowing (in manifolds) to obtain Lipschitz Shadowing in a neighborhood of $\mathcal{A}\subset\mathcal{M}$ \cite{Pilyugin}.

In  \cite{santamaria2014distance,arrieta2017distance} the authors use the existence of an inertial manifold and the Lipschitz Shadowing property in a neighborhood of the global attractor to  obtain results about  continuity of global attractors. To be more specific, the authors studied how the inertial manifolds of a perturbed problem behave in relation to the inertial manifold of the limit problem. With this, they were able to use the Lipschitz Shadowing property in a neighborhood of global attractor to obtain an optimal rate of convergence of the global attractors. 

Unfortunately, inertial manifolds are not easily found in applications. In fact, the existence of an inertial manifold usually  requires a strong gap condition on the eigenvalues of the linear operator associated to the equation \cite{sell2013dynamics,robinson2002computing}. There are several examples of equations whose dynamics can not be reduced to a finite dimensional manifold.  For example, the linear operator associated to the damped wave equation \cite{arrieta1992damped,brunovsky2003genericity,carvalho2012attractors} does not satisfy the gap condition. This implies that the dynamics generated by the nonlinear damped wave equation is ``purely infinite dimensional'' in general. Consequently, it is not possible to apply the known results about Shadowing to guarantee that the semigroup generated by the damped wave equation has the Shadowing property in a neighborhood of its global attractor.

Therefore, a natural question arises: is it possible to obtain Shadowing on semigroups defined in infinite dimensional vector spaces such that its dynamics can not be reduced to a finite dimensional manifold? 
As a matter of fact, we want to prove that Morse-Smale semigroups (Definition \ref{definitionms}), defined in infinite dimensional vector spaces, have the Shadowing property in a neighborhood of its global attractor $\mathcal{A}$, even without the existence of an inertial  manifold.  The main result of this manuscript states that we can remove the hypothesis related to the inertial manifold $\mathcal{M}$ and still obtain  Lipschitz Shadowing inside the global attractor $\mathcal{A}$. Moreover, the property of H\"{o}lder Shadowing \cite{tikhomirov2015holder} holds in a neighborhood of the global attractor. That is:

\begin{theorem}\label{teointroducao}
     Let $\mathcal{T}:=\{\mathcal{T}(t):t\ge0\}\subset\mathcal{C}^1(X)$ be a Morse-Smale semigroup in a Hilbert space $X$ with global attractor $\mathcal{A}$, as in Definition \ref{definitionms} below. Assume that  the Fréchet derivative $D_x\mathcal{T}(t)$ of $\mathcal{T}(t)$ at $x$  satisfies:
     \begin{enumerate}
         \item[(H1)] There exists $C_0,C_1>0$ such that
         \begin{equation*}
            \parallel D_x\mathcal{T}(t)\parallel_{\mathcal{L}(X)}\le C_1e^{C_0t},\ \forall t\ge0,\ \forall x\in\mathcal{A};
         \end{equation*}
        \item[(H2)] For each $x\in\mathcal{A}$ and $v\in X$ the map $[0,+\infty)\ni t\mapsto (D_x\mathcal{T}(t))v\in X$ is continuous.
     \end{enumerate}

      Then,
     \begin{enumerate}
         \item\label{item1daintro..}  $\mathcal{T}(1)|_{\mathcal{A}}:\mathcal{A}\to\mathcal{A}$ has the Lipschitz Shadowing property;
         \item\label{item2daintro..}  For any positively invariant bounded neighborhood $\mathcal{U}\supset\mathcal{A}$, the map $\mathcal{T}(1)|_{\mathcal{U}}:\mathcal{U}\to\mathcal{U}$ has the $\alpha$-H\"{o}lder-Shadowing property for some $0<\alpha<1$.
     \end{enumerate}
\end{theorem}

 As far as we know, Theorem \ref{teointroducao} is the first result of Shadowing in a neighborhood of the global attractor $\mathcal{A}$ that does not appeal to the result of Lipschitz Shadowing in finite dimension.
Moreover, the proof of Lipschitz Shadowing for Morse-Smale semigroups defined in finite dimensional smooth compact manifolds (with no border) \cite{Pilyugin} also  requires that items (H1) and  (H2)  hold. Hence, item (1) of Theorem \ref{teointroducao} generalizes the result of Lipschitz Shadowing in finite dimensional compact manifolds, in the particular case where the non-wandering set is the set of equilibrium points (see Definition \ref{definitionms}). We highlight that items (H1)  and (H2) are reasonable assumptions and are often satisfied in the infinite dimensional context \cite[Theorem 6.33]{carvalho2012attractors}, especially when the semigroup $\mathcal{T}$ satisfies the variation of constants formula. For example, the damped wave equation generates a Morse-Smale semigroup that satisfies items (H1) and (H2), as we will see in Example \ref{exemploms}.

As we said before, the semigroup related to the damped wave operator (see \ref{dwequation})  does not satisfy the conditions required to obtain an inertial manifold  and consequently it is not possible to obtain shadowing in a neighborhood of its global attractor by the known results. Despite that, since the semigroup related to the damped wave equation satisfies all assumptions from Theorem \ref{teointroducao}, we will guarantee that the H\"{o}lder-Shadowing property  holds in a neighborhood of its attractor.

Let us briefly discuss the idea to prove Theorem \ref{teointroducao}. We will prove items (1) and (2) of Theorem \ref{teointroducao} separately. First we prove item (1), which is the most eloborated part.  In fact, after proving item (1) we will show item (2) just by  estimating the distance between bounded pseudo-orbits and the global attractor (see Lemma \ref{lemma-distancia}). To put more details, we will show that if $\{x_n\}_{n\in\mathbb{Z}}\subset\mathcal{U}$ is a  $\delta$-pseudo orbit  of $\mathcal{T}(1)$, where $\mathcal{U}$ is a bounded neighborhood of $\mathcal{A}$, then the smaller the $\delta$ the close $x_n$ is to $\mathcal{A}$, so that if $\delta\to 0$ then $d(x_n,\mathcal{A})\to 0$.

 The strategy to prove item (1) of Theorem \ref{teointroducao} is to adapt the (very beautiful) proof of  Lipschitz Shadowing for Morse-Smale systems in compact smooth manifolds $M$ (with no border), that is given in \cite[Chapter 2]{Pilyugin}. We will now provide a sketch of the proof in the finite dimensional case and then highlight the difficulties to adapting this proof to our case, where the phase space is infinite dimensional. The proof in finite dimension is  divided in two parts: the construction of compatible subbundles, originally introduced by \cite{robinson1974structural}, and the application of the Banach Fixed Point Theorem to obtain an orbit close to a pseudo-orbit. To put more details about this two parts, let us first focus on the application of the Banach Fixed Point Theorem. 

Let $\{\mathcal{T}(t):t\ge0\}$ be a Morse-Smale semigroup defined in a finite dimensional space $(X,\parallel\cdot\parallel)$ (it could be also a compact manifold with no border) with global attractor $\mathcal{A}$. Fix  a bounded $\delta-$pseudo orbit  $\{x_n\}_{n\in\mathbb{Z}}\subset\mathcal{A}$  of $\mathcal{T}(1)$ (see Definition \ref{defshadoww}). To obtain an orbit of $\mathcal{T}(1)$ close to $\{x_n\}_{n\in\mathbb{Z}}$ we have to find a family of vectors $\{v_n\}_{n\in\mathbb{Z}}\subset X$, where $\parallel v_n\parallel$ is small for all $n\in\mathbb{Z}$, such that $\mathcal{T}(1)(x_n+v_n)=x_{n+1}+v_{n+1}$ for all $n\in\mathbb{Z}$. The ``true orbit'' is given by $\{x_n+v_n\}_{n\in\mathbb{Z}}$. Hence, we  want to obtain a fixed point of the map \begin{align*}
    \phi:&Y\longrightarrow Y\\
    &\{v_n\}_{n\in\mathbb{Z}}\mapsto \{\mathcal{T}(1)(x_{n-1}+v_{n-1})-x_{n}\}_{n\in\mathbb{Z}},
\end{align*}
where $Y=X^{\mathbb{Z}}$ is the Banach space of bounded sequences of $X$, with $\parallel v_n\parallel_Y=\sup\limits_{n\in\mathbb{Z}}\parallel v_n\parallel$. Note that $\parallel \phi(\{0\}_{n\in\mathbb{Z}})\parallel_Y<\delta$, which indicates that the fixed point of $\phi$ has small norm, as we desire. Therefore, for each $n\in\mathbb{Z}$ we have to study the map $\phi_n:X\to X$ given by $\phi_n(v)=\mathcal{T}(1)(x_{n-1}+v)-x_n$ for all $v\in X$. Since $\phi_n=D_{x_n}\mathcal{T}(1)+(\phi_n-D_{x_n}\mathcal{T}(1))$ and the map $\phi_n-D_{x_n}\mathcal{T}(1)$ is a contraction (in a ball of $Y$ centered in $0$ and with small radius),  we need to extract properties from the linear map $D_{x_n}\mathcal{T}(1)$ so it becomes a contraction and we can apply the Banach Fixed Point Theorem. To be more precise, we will have to find subspaces $S(x_n)$ and $U(x_n)$ such that   $S(x_n)\oplus U(x_n)=X$, and the restrictions $(D_{x_n}\mathcal{T}(1))|_{S(x_n)}$ and $(D_{x_n}\mathcal{T}(-1))|_{U(x_n)}$ have small norms, where we will need to give a meaning to $\mathcal{T}(-1)$ and $D_x\mathcal{T}(-1)$ since in general $\mathcal{T}$ is only defined for $t\ge0$. This subspaces will behave similar to the stable and unstable subspaces of a hyperbolic fixed point. This is where the geometrical construction of the compatible subbundles is essential.

Therefore,  to apply the Banach Fixed Point Theorem we have to construct two families of subspaces $\{S(x)\subset X\}_{x\in\mathcal{A}}$ and $\{U(x)\subset X\}_{x\in\mathcal{A}}$, that must satisfy several properties. Summarizing,  the most important properties are the positively invariance in relation to the derivative
\begin{equation*}
     (D_x\mathcal{T}(t))S(x)\subset S(\mathcal{T}(t)x)\ \text{and}\ (D_x\mathcal{T}(-t))U(x)\subset U(\mathcal{T}(-t)x),\ \forall x\in\mathcal{A},\ \forall t\ge0
\end{equation*}
and the exponential decay
\begin{align*}
    &\parallel (D_x\mathcal{T}(t))v\parallel\le C\lambda^t\parallel v\parallel,\ \forall v\in S(x),\ \forall t\ge0,\forall x\in\mathcal{A},\\
    &\parallel (D_x\mathcal{T}(-t))v\parallel\le C\lambda^t\parallel v\parallel,\ \forall v\in U(x),\ \forall t\ge0,\forall x\in\mathcal{A}.
\end{align*}
for some $C>0$ and $\lambda\in(0,1)$.
This is where the main difficulties lie in extending  the proof of Lipschitz Shadowing in finite dimensional to the infinite dimensional setting. In fact, the application of the Banach Fixed Point Theorem in the infinite dimensional context follows very similar to the finite dimensional case but the construction of the subbundles $S$ and $U$ need several adjustments in the infinite dimensional setting. Therefore, if we are able to construct the compatible subbundles, we are practically done with the proof of item (1) of Theorem \ref{teointroducao}. Let us highlight why the construction of the compatible subbundles in the infinite dimensional setting does not follow immediately from the proof in finite dimension. To obtain the subbundles $S$ and $U$ in a compact manifold $M$, the author in \cite{Pilyugin} uses the following facts:

\begin{enumerate}
    \item[(P1)] The derivative $D_x\mathcal{T}(t)$ of $\mathcal{T}(t)$ at $x$ is an isomorphism for all $x\in M$ and $t\ge0$;
    \item[(P2)] $M$ is a smooth (finite dimensional) manifold and consequently, $W^s(x^*)$ and $W^u(x^*)$ are $\mathcal{C}^1$-manifolds for any hyperbolic equilibrium $x^*$;
\end{enumerate}

The surjectivity of the derivative  $D_x\T(t)$ (that is just an isomorphism onto its range in the infinite dimensional setting), is essential to construct the compatible subbundles in finite dimensional. For example,  the compatible subbundles are constructed  by mathematical induction (see \cite[Lemma 2.2.9]{Pilyugin}) and the surjectivity of the derivative is used to prove the induction hypothesis.  Moreover, if the derivative is an isomorphism, then the construction of the subbundle $S$ is analogous to the construction of the subbundles $U$, which is not our case. Furthermore, $W^s(x^*)$ is not necessarily a manifold if the phase space has infinite dimension and therefore the tangent spaces $T_xW^s(x^*)$, which are also essential for the construction of the compatible subbundles,  are not well defined for each $x\in W^s(x^*)$. These are  some of the obstacles that we will have to overcome, which clearly states that the proof in infinite dimension does not follow straightforward from the finite dimensional case.

 As applications of Theorem \ref{teointroducao}, we will present results related to the structural stability of Morse-Smale semigroups defined in Hilbert spaces and to the continuity of global attractors. It is known that small perturbations of Morse-Smale semigroups are still Morse-Smale and there exists a  phase diagram isomorphism between the limit and the perturbed global attractors \cite{Bortolan-Carvalho-Langa-book}. However, it is not known if the global orbits inside the attractors from the perturbed problems remain close to the global orbits inside the attractor from the limit problem, as it happens in the finite dimensional case \cite{Pilyugin}. With item (2) of Theorem \ref{teointroducao} we will guarantee this property (Theorem \ref{robustezams}). In fact, we can guarantee that bounded global orbits from perturbed Morse-Smale systems remains close to a bounded global orbit from the limit problem even for non-autonomous perturbations of an autonomous Morse-Smale semigroup (see Theorem  \ref{gssna}). 

 Regarding the continuity of global attractors, we will show a new method to estimate the distance between bounded pseudo-orbits and the global attractor (see Lemma \ref{lemma-distancia}). With this method, we will provide a new way to estimate the distance between global attractors, that allow us to obtain an $\alpha$-H\"{o}lder relation, for any $\alpha\in(0,1)$, between the distance of the global attractors and the distance of the semigroups (see Theorem \ref{contfinal}).

The present paper is divided as follows: in Section \ref{secbasic} we define the main concepts of this work, such as Morse-Smale semigroups and Shadowing. In Section \ref{sectionnoatrator} we adapt the construction of compatible subbundles \cite{Pilyugin,robinson1974structural} to prove item (1) of Theorem \ref{teointroducao}. In Section \ref{sectionneighbo} we estimate the distance between pseudo-orbits and the global attractor and prove item (2) of Theorem \ref{teointroducao}. In Section \ref{seccont} we mix  results from Sections \ref{sectionnoatrator} and \ref{sectionneighbo} to obtain applications related to the stability of Morse-Smale semigroups (in infinite dimensional Hilbert spaces) and continuity of global attractors. Finally, we have an Appendix in Section \ref{appendix}.

\section{Basic concepts and known facts}\label{secbasic}
In this section we announce some of the fundamental concepts and results from  the theory of dynamical systems, such as Shadowing, global attractors and Morse-Smale semigroups. These concepts will be present through the whole text and are essential for the understanding of this work.

We start this section by introducing the notion of pseudo-orbits and  Shadowing.
\begin{definition}\label{defshadoww}
	Let $(M,d)$ be a metric space,
	$\mathcal{T}:M\to M$ be a map. 
	\begin{enumerate}
        \item A sequence $\{x_n\}_{n\in \mathbb{Z}}$ in $M$ is an \textbf{orbit} of $\mathcal{T}$ if $\mathcal{T}x_n=x_{n+1}$ for all $n\in\mathbb{Z}$.
	\item A sequence $\{x_n\}_{n\in \mathbb{Z}}$ in $M$ 
		is said to be a  \textbf{$\delta$-pseudo orbit}
		of $\mathcal{T}$ for some $\delta>0$ if
	\begin{equation*}
		d(\mathcal{T}x_n,x_{n+1})\leq \delta,  \ \forall\  n\in \mathbb{Z}.
	\end{equation*}
        
		\item Let $\{x_n\}_{n\in \mathbb{Z}}$ and $\{z_n\}_{n\in \mathbb{Z}}$ be sequences on $M$.
		We say that $\{z_n\}_{n\in \mathbb{Z}}$ \textbf{$\epsilon$-shadows}
		$\{x_n\}_{n\in \mathbb{Z}}$, for some $\epsilon>0$ if 
		\begin{equation*}
		d(x_n,z_n)\leq \epsilon, \ \forall\ n\in\mathbb{Z}.
		\end{equation*}

	\item We say that $\mathcal{T}$ admits the  \textbf{$\alpha$-H\"{o}lder Shadowing} property
		if there exists $d_0,L>0$ and $\alpha\in(0,1]$ such that  any  
		$d$-pseudo orbit of $\mathcal{T}$, with $d\in [0,d_0]$, is $Ld^{\alpha}$-shadowed by an orbit of $\mathcal{T}$. If $\alpha=1$ we say that $\mathcal{T}$ has the \textbf{Lipschitz Shadowing} property.
  \item We say that $\mathcal{T}$ admits the  \textbf{Logarithm Shadowing} property if there exists $d_0,L>0$ such that  any  
		$d$-pseudo orbit of $\mathcal{T}$, with $d\in [0,d_0]$, is $Ld|\ln{d}|$-shadowed by an orbit of $\mathcal{T}$.
	\end{enumerate} 
\end{definition}

    Note that if $\alpha\in(0,1)$ then
    \begin{equation*}
        \text{Lipschitz Shadowing}\Rightarrow\text{Logarithm Shadowing}\Rightarrow\alpha-\text{H\"{o}lder Shadowing}.
    \end{equation*}

As we said before, the main result of this work is to show the properties of Lipschitz and H\"{o}lder Shadowing stated in Theorem \ref{teointroducao}. We introduce the notion of Logarithm Shadowing in Definition  \ref{defshadoww} because it will also appear in Section \ref{sectionneighbo}, where we show that it is possible to obtain Logarithm Shadowing in a especial case, depending on the Lipschitz constant of the map (Theorem \ref{teodasconstantesdelip}).

We now introduce some standard definitions from the semigroup/dynamical systems theory.

\begin{definition}
Let $(M,d)$ be a metric space and $\mathcal{C}(M)$ be the space of continuous functions from $M$ into $M$. We say that the family $\mathcal{T}=\{\mathcal{T}(t):t\ge 0\}\subset \mathcal{C}(M)$ is a \textbf{semigroup} (dynamical system) if satisfies:
\begin{enumerate}
    \item $\mathcal{T}(0)x=x,\ \forall x\in M;$
    \item $\mathcal{T}(t)\mathcal{T}(s)=\mathcal{T}(t+s),\ \forall t,s\ge0;$
    \item The map $[0,+\infty)\times M\ni(t,x)\mapsto \mathcal{T}(t)x$ is continuous.
\end{enumerate}
\end{definition}

\begin{definition}
    Let $\mathcal{T}=\{\mathcal{T}(t):t\ge0\}\subset\mathcal{C}(M)$ be a semigroup in a metric space $(M,d)$ and $S\subset M$. We say that
    \begin{enumerate}
        \item $S$ is \textbf{positively invariant} if
        $\mathcal{T}(t)S\subset S$ for all $t\ge0$;
        \item $S$ is  \textbf{invariant} if  $\mathcal{T}(t)S= S$ for all $t\ge0$. In particular, if $S=\{x^*\}$ is an invariant unitary set, we say that $x^*$ is an \textbf{equilibrium point} of $\mathcal{T}$.
    \end{enumerate}
\end{definition}

Let us define the concept of global attractor, which is a crucial subject from the theory of infinite dimensional dynamical systems. In fact, the global attractor dictates the whole dynamics of the system.
\begin{definition}\label{defatt}
    Let $\mathcal{T}=\{\mathcal{T}(t):t\ge 0\}$ be a semigroup in a metric space $(M,d)$. We say that a compact set $\mathcal{A}\subset M$ is a \textbf{global attractor} if  it is invariant, i.e. $\mathcal{T}(t)\mathcal{A}=\mathcal{A}$ for all $t\ge 0$,
    and 
    attracts bounded subsets of $M$, that is, for each bounded subset $B$ of $M$ 
        \begin{equation}
            \lim_{t\to +\infty} dist_H(\mathcal{T}(t)B,\mathcal{A})=0,
        \end{equation}
        where $dist_H(A_1,A_2)=\sup\limits_{a_1\in A_1}\inf\limits_{a_2\in A_2}d(a_1,a_2)$ is the Hausdorff semi-distance between the subsets $A_1,A_2\subset M$.
\end{definition}

 It follows from Definition \ref{defatt} that the global attractor is unique. We now define the concept of global solution, which is important to characterize the global attractor.
\begin{definition}[Global Solution/Orbit]\label{defglobalsol}
    Let $\mathcal{T}=\{\mathcal{T}(t):t\ge0\}\subset\mathcal{C}(M)$ be a semigroup in a metric space $M$. We say that a continuous function $\phi:\mathbb{R}\to M$ is a \textbf{global solution} (or global orbit) of $\mathcal{T}$ if satisfies $\mathcal{T}(s)\phi(t)=\phi(t+s)$ for all $s\ge0$ and $t\in\mathbb{R}$. If $\phi(0)=x$ for some $x\in M$ we say that $\phi$ is a global solution through $x$.
\end{definition}

The following result is a characterization of the global attractor. We refer to   \cite{Bortolan-Carvalho-Langa-book} for more details.

\begin{proposition}\label{characatt}
    Let $\mathcal{T}=\{\mathcal{T}(t):t\ge0\}$ be a  semigroup in a metric space $(M,d)$ with global attractor $\mathcal{A}$. Then the global attractor is characterized by  
    \begin{equation*}
        \mathcal{A}=\{x\in M:\text{there exists a bounded global solution}\ \phi:\mathbb{R}\to M\ \text{through}\ x\}.
    \end{equation*}
    Consequently, any bounded global solution $\phi:\mathbb{R}\to M$ is contained in the global attractor $\mathcal{A}$.
\end{proposition}

Let us introduce some notations that will be used in this paper. Consider a semigroup $\mathcal{T}=\{\mathcal{T}(t):t\ge0\}\subset\mathcal{C}(M)$ in a metric space $(M,d)$. We  denote 
    \begin{enumerate}
    \item The \textbf{positive orbit} of $S\subset M$ by $$\gamma^+(S):=\bigcup\limits_{t\ge0}\{\mathcal{T}(t)x:x\in S\}.$$
   \item The negative orbit of $x\in M$ by $$\gamma^-(x)=\{\phi(t):\phi \ \text{is a global solution through}\ x, t\le 0\}.$$    Similarly,  we denote the \textbf{negative orbit} of $S\subset M$ by
   \begin{equation*}
       \gamma^-(S):=\bigcup\limits_{x\in S}\gamma^-(x).
   \end{equation*}
  In particular, if $\mathcal{T}(t):M\to M$ is bijective for each $t\ge0$, with inverse denoted by $\mathcal{T}(-t)$, then $\gamma^-(S)=\bigcup\limits_{t\ge0}\{\mathcal{T}(-t)x:x\in S\}$.
    \item The \textbf{orbit} of $S\subset M$  by
    \begin{equation*}
        \gamma(S)=\gamma^+(S)\cup\gamma^-(S).
    \end{equation*}
    \end{enumerate}

We will also stablish the following notations through this paper: if $X,Y$ are normed vector spaces, we will denote the space of continuous linear operators from $X$ into $Y$ by $\mathcal{L}(X,Y)$. We also define $\mathcal{L}(X):=\mathcal{L}(X,X)$.
If $F:X\to X$ is a differentiable map in a Banach space $X$, then we will denote the Fréchet derivative of $F$ at $x\in X$ by $D_xF\in\mathcal{L}(X).$
If $\mathcal{T}=\{\mathcal{T}(t):t\ge0\}\subset \mathcal{C}^1(X)$ is a  semigroup in a Banach space $X$ and $x^*$ is an equilibrium point of $\mathcal{T}$,  we  denote the inverse image of $E\subset X$ by the derivative of $\mathcal{T}(t)$ at $x^*$ by
\begin{equation*}
    (D_{x^*}\mathcal{T}(-t))E:=(D_{x^*}\mathcal{T}(t))^{-1}E=\{z\in X: (D_{x^*}\mathcal{T}(t))z\in E\}
\end{equation*}
If $F:X\to Y$ is a map and $C\subset X$ is a subset, then we will denote the range of $F$ by $R(F)$ and the restriction of $F$ in $C$ by $F|_C:C\to Y$.

Now we introduce the notions of hyperbolicity and non-wandering points, that are required to define Morse-Smale semigroups. Let us start with the notion of hyperbolic fixed points in Banach spaces \cite{henry1994exponential,Bortolan-Carvalho-Langa-book}. 

\begin{definition}
 \label{hyperbolicset}
Let $\mathcal{T}=\{\mathcal{T}(t):t\ge 0\}\subset\mathcal{C}^1(X)$ be a semigroup in a Banach space $X$ and $x^*$ be an equilibrium point of $\mathcal{T}$. Suppose that for each $t\ge0$ the linear operator $D_{x^*}\mathcal{T}(t)$ is injective. We say that $x^*$ is hyperbolic if it satisfies:
\begin{enumerate}
    \item There exists a continuous projection $P(x^*):X\to X$ such that   $$S_0(x^*)\oplus U_0(x^*)=X,$$ where $U_0(x^*)=R(P(x^*))$ and $S_0(x^*)=R(I-P(x^*))$;
    \item $dim\ U_0(x^*)<+\infty$;
    \item  $(D_{x^*}\mathcal{T}(t))P(x^*)=P(x^*)(D_{x^*}\mathcal{T}(t))$ for all $t\ge0$ and the operator $$(D_{x^*}\mathcal{T}(t))|_{U_0(x^*)}:U_0(x^*)\to U_0(x^*)$$ is an isomorphism. Consequently,
    \begin{equation*}
        (D_{x^*}\mathcal{T}(t))U_0(x^*) = U_0(x^*)\ \forall t\in\mathbb{R} \ \ \text{and}\ \ \ (D_{x^*}\mathcal{T}(t))S_0(x^*)\subset S_0(x^*),\ \forall t\ge 0; 
    \end{equation*}
         \item There exists $C>0, \lambda\in (0,1)$  such that
        \begin{align*}
            \parallel (D_{x^*}\mathcal{T}(t))v^s\parallel&\le C\lambda^t\parallel v^s\parallel,\ \forall v^s\in S_0(x^*),\ \forall t\ge 0,\\
           \parallel (D_{x^*}\mathcal{T}(-t))v^u\parallel&\le C\lambda^t\parallel v^u\parallel,\ \forall v^u\in U_0(x^*),\ \forall t\ge 0.
        \end{align*}
        The constants $C,\lambda$ are called constants of hyperbolicity of $x^*$.
    \end{enumerate}
\end{definition}

\begin{definition}
	Let $(M,d)$ be a metric space and $\mathcal{T}=\{\mathcal{T}(t):t\ge 0\}$ be a semigroup in $M$.
	We say that $x\in M$ is a \textbf{non-wandering} point of $\mathcal{T}$ if for any $t_0\ge 0$ and any neighborhood $V_x$ of $x$ 
	there exists $t > t_0$ such that
	$\mathcal{T}(t)V_x\cap V_x\neq \emptyset$. We will denote the set of non-wandering points of $\mathcal{T}$ by $\Omega$.
\end{definition} 
Note that if $x_0\in M$ is such that $\mathcal{T}(t^*)x_0=x_0$ for some $t^*>0$ , then $x_0\in \Omega$. In particular, equilibrium points and periodic orbits are  inside the non-wandering set $\Omega$.

Now we can finally define Morse-Smale semigroups, that is  the class of dynamical systems that we will work throughout this paper.
\begin{definition}[Morse-Smale]\label{definitionms}
	Let $X$ be a Banach space and  
	$\mathcal{T}=\{\mathcal{T}(t):t\geq 0\}\subset \mathcal{C}^1(X)$ be a semigroup with global attractor 	$\mathcal{A}$.
	We say that $\mathcal{T}$ is a \textbf{Morse-Smale} semigroup if it satisfies the following conditions:
	\begin{enumerate}
	   \item\label{tinjetora} $\mathcal{T}(t)|_\mathcal{A}$  is injective for all $t\ge 0$.
	\item\label{derivadainjetora} The Fréchet derivative $D_z\mathcal{T}(t)\in\mathcal{L}(X)$ of $\mathcal{T}(t)$ at $z$ is an 
	isomorphism onto its image for all $z\in \mathcal{A}$ and $t\ge 0$.
	\item The non-wandering set of $\mathcal{T}$ is given just by the set of equilibrium points $\Omega=\{x^*_1,\cdots,x^*_p\}$, which is a finite set.  Moreover, each equilibrium point $x^*_i$ is  hyperbolic. In particular, $\mathcal{T}$ does not have periodic orbits.
	\item\label{dimensaovariedade}  $\dim W^u_{loc}(x^*)< \infty$ for every $x^*\in\Omega$, where $dim$ is the dimension related to the differentiable manifold.
	\item  $W^u(x^*_i)$ and
	$W^s_{loc}(x^*_j)$ are transverse for all $i\neq j$, i.e.,
	if $z\in W^u(x^*_i)\cap W^s_{loc}(x^*_j)$, then
	\begin{equation*}
	T_zW^u(x^*_i)+T_zW^s_{loc}(x^*_j)=X,
	\end{equation*}
	where $T_zW^u(x^*_i)$ and $T_zW^s_{loc}(x^*_j)$ are the tangent spaces of  
	$W^u(x^*_i)$ and $ W^s_{loc}(x^*_j)$ at $z$, respectively.
	\end{enumerate}
\end{definition}

The following Remark explains why items (4) and (5) from Definition \ref{definitionms} are well defined.
\begin{remark}\label{bebosim}
    Let $X$ is a Banach space, $\mathcal{T}=\{\mathcal{T}(t):t\ge0\}\subset \mathcal{C}^1(X)$ be a semigroup that satisfies items (1) and (2) from Definition \ref{definitionms} and $x^*$ be a hyperbolic equilibrium point of $\mathcal{T}$. Consider the subspaces $S_0(x^*)$ and $U_0(x^*)$, defined in Definition \ref{hyperbolicset}, related to the hyperbolicity of $x^*$. Then the local unstable manifold $W_{loc}^u(x^*)$ and local stable manifold $W^s_{loc}(x^*)$ are  graphs of  $\mathcal{C}^1$-maps $\Psi_u:U_0(x^*)\to S_0(x^*) \ \ \text{and}\ \ \Psi_s:S_0(x^*)\to U_0(x^*)$, respectively \cite[Chapter 4]{Bortolan-Carvalho-Langa-book}. Consequently $W^u_{loc}(x^*)$ and  $W^s_{loc}(x^*)$  are $\mathcal{C}^1$- manifolds. Moreover, the unstable manifold  $W^u(x^*)$ is a $\mathcal{C}^1$-immersed manifold for Morse-Smale semigroups, since the dimension of the manifold $W^u_{loc}(x^*)$ is finite \cite[Chapter 4]{Bortolan-Carvalho-Langa-book}. Note that it is not clear if $W^s(x^*)$ is a manifold. 
    
    Therefore,  the dimensions of the manifolds $W^u_{loc}(x^*)$ and $\dim W^u_{loc}(x^*)$, in item (4) of Definition \ref{definitionms}, are well defined. The same reasoning shows that the tangents spaces $T_zW^u(x_i^*)$ and $T_zW^s_{loc}(x_j^*)$, in item (5), are well defined. Moreover,
    \begin{equation*}
    T_{x^*}W^u(x^*)=U_0(x^*)\,\,\, \text{and}\,\,\, T_{x^*}W^s_{loc}(x^*)=S_0(x^*)
\end{equation*}
for each hyperbolic equilibrium point $x^*$ of $\mathcal{T}$  \cite[Chapter 4]{Bortolan-Carvalho-Langa-book}.
\end{remark}
 
   We highlight that the usual definition of infinite dimensional Morse-Smale semigroups allows periodic orbits in the non-wandering set $\Omega$ and just demand that the derivative $D_x\mathcal{T}(t)$  is a bounded injective operator, for all $t>0$ and $x\in\mathcal{A}$ (see e.g. \cite{Bortolan-Carvalho-Langa-book}).  Since we will assume that $\Omega$ consists in the equilibrium points of $\mathcal{T}$ and that $D_x\mathcal{T}(t)$ is an isomorphism onto its range throughout the whole manuscript, we have redefined  Morse-Smale semigroups as in Definition \ref{definitionms} in order to simplify the text. This definition still generalizes the concept of Morse-Smale diffeomorphisms in finite dimension, in the particular case where $\Omega$ just has equilibrium points.

Let us provide an example of a Morse-Smale semigroup that satisfies all assumptions from Theorem \ref{teointroducao}.

 \begin{example}\label{exemploms}
       Let $\Omega\subset\mathbb{R}^3$ be  a smooth bounded domain and consider the damped wave equation 
       \begin{equation}\label{dwequation}
    \begin{cases}
      u_{tt}+\gamma u_t-\Delta u=f(u),\ \ \  u|_{\partial\Omega=0}\\
      u(x,0)=u_0\in H^1_0(\Omega),\\
      u_t(x,0)=v_0\in L^2(\Omega).
    \end{cases}\,.
\end{equation}
     
Assume that $f:\mathbb{R}\to\mathbb{R}$ is a $C^2$ function satisfying 
       \begin{equation}\label{cond1}
\limsup_{|s|\to+\infty}\dfrac{f(s)}{s}\le0
       \end{equation}
       and
       \begin{equation}\label{cond2}
           |f'(s)|\le k(1+|s|^p),\ \forall s\in\mathbb{R}
       \end{equation}
       for some $k>0$ and $p\in[0,2)$.
       In \cite[Chapter 15]{carvalho2012attractors} the authors show that, under this conditions of $f$, equation \eqref{dwequation} is well posed and its solutions define a semigroup $T_f=\{T_f(t):t\ge0\}\subset\mathcal{C}^1(X)$, with global attractor, in the phase space $X=H^1_0(\Omega)\times L^2(\Omega)$. The semigroup $T_f$ satisfies conditions (H1) and (H2). Moreover, for any neighborhood $\mathcal{O}\ni f$  in the $C^2$-strong Whitney topology (see Chapter 3.4 of \cite{Pilyugin}) there exists $g\in \mathcal{O}$ such that $T_g$ is a
        Morse-Smale semigroup.

In \cite{carvalho2012attractors} the authors prove the well-posedness of the damped wave equation in $X=H^1_0(\Omega)\times L^2(\Omega)$ and that the semigroup $\{T_f(t):t\ge0\}$ generated by its solutions has global attractor, is gradient and $T_f(t)\in \mathcal{C}^1(X)$ for each $t\ge0$. Properties (H1) and (H2) follow from \cite[Theorem 6.33]{carvalho2012attractors} and the Gronwall-inequality. Property (H2) is an immediately consequence of \cite[Theorem 6.33]{carvalho2012attractors}.

The proof that the Morse-Smale property is generic in the Whitney topology can be found in \cite{brunovsky2003genericity}.
 \end{example}

    We recall that the eigenvalues of the linear operator associated to the damped wave equation do not satisfy the gap condition demanded to obtain an inertial manifold (see \cite{robinson2002computing,sell2013dynamics}). Therefore, it is not possible to guarantee the Shadowing property in a neighborhood of its global attractor by the previous  results, since the existence of an inertial manifold was required (see \cite[Section 3.4]{Pilyugin}). Despite that, since the semigroup related to the damped wave equation is Morse-Smale and satisfies (H1) and (H2), we will use the results from Section \ref{sectionneighbo} to show that the Shadowing property still holds in a  neighborhood of its global attractor.

Before we finish this section, we now provide some properties of Morse-Smale semigroups.
\begin{lemma}\label{lemma-Lyapunov-function}
	Let $X$ be a Banach space and $\mathcal{T}=\{\mathcal{T}(t):t\ge0\}$ be a Morse-Smale semigroup in $X$ with non-wandering set $\Omega=\{x^*_1,\dots,x^*_p\}$. Then
\begin{enumerate}
    \item There exists a continuous function $V:X\to [0,+\infty)$, called Lyapunov function, such that
	\begin{itemize}
		\item $V(x^*_i)=i$, for all $1\leq i\leq p$.
		\item If $x\notin \Omega$, then $[0,+\infty)\ni t\mapsto V(\mathcal{T}(t)x)$ is strictly decreasing. 
	\end{itemize}
\item For each $x\in\mathcal{A}$ there exist exactly two equilibrium points $x_i^*,x_j^*\in\Omega$ and an unique global solution $\phi$ through $x$ such that
\begin{equation*}
    \lim\limits_{t\to-\infty}d(\phi(t),x_i^*)=0\ \text{and}\ \lim\limits_{t\to+\infty}d(\phi(t),x_j^*)=0.
\end{equation*}
In particular,
\begin{equation}\label{caraatrator}
    \mathcal{A}=\bigcup_{i=1}^p W^u(x^*_i).
\end{equation}
\item There is no homoclinic structure in $X$ (see Definition \ref{homoclinic}).
\end{enumerate}    
\end{lemma}
\begin{proof}
    It is known that  Morse-Smale semigroups are Dynamically Gradient \cite{Bortolan-Carvalho-Langa-book}, that is, satisfy items (2) and (3).  In \cite{aragao2013non} the authors prove that Dynamically Gradient semigroups are Gradient (it has a Lyapunov function). In fact, the concepts of Gradient and Dynamically Gradient are equivalent \cite{Bortolan-Carvalho-Langa-book}.
\end{proof}

\section{Lipschitz Shadowing in the Global Attractor}\label{sectionnoatrator}
In this section we will prove item (1) of Theorem \ref{teointroducao}, that is, we will show that the Lipschitz Shadowing property holds for $\mathcal{T}(1)|_{\mathcal{A}}:\mathcal{A}\to \mathcal{A}$, where $\mathcal{A}$ is the global attractor of a Morse-Smale semigroup $\mathcal{T}$. Therefore, in this section we will only consider the semigroup $\mathcal{T}=\{\mathcal{T}(t):t\ge0\}$ restricted to the global attractor $\mathcal{A}$.  We recall that item (1) of Theorem \ref{teointroducao} extends the results of Lipschitz Shadowing known for Morse-Smale diffeomorphisms defined in compact smooth manifolds \cite[Chapter 2]{Pilyugin} and $\mathbb{R}^n$ \cite{santamaria2014distance}. Moreover, we are not assuming the existence of an inertial manifold and therefore the results from this section are not a consequence of \cite[Chapter 3]{Pilyugin}.

To prove item (1) of Theorem \ref{teointroducao} we have to construct a family of compatible subbbundles, as in the proof of the finite dimensional case  \cite{Pilyugin}. This construction is very technical and need several adjustments in the infinite dimensional setting. We have divided this section in three subsections: in Subsection \ref{subseccontinuidade} we will define the notion of continuity of subbundles and prove several results related to this concept. In Subsection \ref{subseccompatible} we use the results about continuity of subbundles, given in  Subsection \ref{subseccontinuidade}, to construct the compatible subbundles.  Finally, in Subsection \ref{subsec33}, we exploit the geometric properties of the compatible subbundles to obtain sufficient conditions for applying the Banach fixed point theorem and conclude the proof of Lipschitz Shadowing on $\mathcal{A}$.

Since some of the concepts and proofs of this section can be found in the proof of the finite dimensional case \cite[Chapter 2]{Pilyugin}, in this section we will only  prove  the results whose  proofs do not follow from the finite dimensional case.  In any case, we will enunciate every essential result.  From now on  $\mathcal{T}=\{\mathcal{T}(t):t\ge0\}\subset\mathcal{C}^1(X)$ will be a Morse-Smale semigroup in a Hilbert space $X$, with global attractor $\mathcal{A}$ and  non-wandering set $\Omega=\{x_1^*,\dots,x_p^*\}$, given by a finite amount of hyperbolic equilibrium points of $\mathcal{T}$. Through this section we will  assume that $\mathcal{T}$ satisfies items (H1) and (H2) from Theorem \ref{teointroducao}.

 Since $\mathcal{T}(t)|_{\mathcal{A}}:\mathcal{A}\to\mathcal{A}$ is a homeomorphism for each $t\ge0$ (from item 1 of Definition \ref{definitionms} and the invariance of the global attractor), from now on we will denote $\mathcal{T}(-t):=(\mathcal{T}(t)|_{\mathcal{A}})^{-1}$ for each $t\ge0 $. If $E\subset X$, $x\in\mathcal{A}$ and $t\ge0$  we will denote
\begin{equation*}
    (D_x\mathcal{T}(-t))E:=(D_{\mathcal{T}(-t)x}\mathcal{T}(t))^{-1}E=\{z\in X: (D_{\mathcal{T}(-t)x}\mathcal{T}(t))z\in E\}.
\end{equation*}
Moreover, since the linear operator  $D_{\mathcal{T}(-t)x}\mathcal{T}(t)\in\mathcal{L}(X)$ is  an isomorphism onto its range for each $x\in\mathcal{A}$, then we   denote 
\begin{equation*}
    D_x\mathcal{T}(-t):=(D_{\mathcal{T}(-t)x}\mathcal{T}(t))^{-1}\in\mathcal{L}(R(D_{\mathcal{T}(-t)x}\mathcal{T}(t)), X),\ \forall x\in\mathcal{A},\ \forall t\ge0.
\end{equation*}

Finally, we will denote by $V$ a Lyapunov function for $\T$ satisfying the properties from Lemma \ref{lemma-Lyapunov-function}.

The following results are technical lemmas, whose proofs follow exactly as in the finite dimensional case \cite[Chapter 2]{Pilyugin}.
\begin{lemma}\label{lemma2.2.5}
	Fix $i\in\{1,\dots,p\}$ and $d\in(0,1)$ and define the set 
	\begin{equation*}
	D:=V^{-1}(i+d)\cap W^s(x^*_i)\cap\mathcal{A}.
	\end{equation*}
	For any neighborhood $Q$  of $D$ in $V^{-1}(i+d)\cap\mathcal{A}$ the set 
	$\gamma^+(Q)\cup W^u(x^*_i)$ contains a neighborhood (in $\mathcal{A}$) of $x^*_i$ .
\end{lemma}

\begin{lemma}\label{lemma2.2.6}
        Let $N$ be a neigborhood (in $\mathcal{A}$) of $x^*_i$ , for some $i\in\{1,\dots,p\}$. Then, there exists $d'\in(0,1)$ such that
        $$V^{-1}(i+d)\cap W^s(x^*_i)\cap\mathcal{A}\subset N, \forall\ d\in[0,d'].$$
\end{lemma}
 Since $\mathcal{T}(t)|_{\mathcal{A}}:\mathcal{A}\to \mathcal{A}$ is a homeomorphism, Lemmas \ref{lemma2.2.5} and \ref{lemma2.2.6} also holds for the unstable manifold. For example, regarding Lemma \ref{lemma2.2.6}, given a neigborhood $N$ (in $\mathcal{A}$) of $x^*_i$ then there exists $d'\in(0,1)$ such that
        $$V^{-1}(i-d)\cap W^u(x^*_i)\cap\mathcal{A}\subset N, \forall\ d\in[0,d'].$$
Analogously,  Lemma \ref{lemma2.2.5} also holds for the unstable manifold of $x_i^*$.
\subsection{Continuity of Subbundles}\label{subseccontinuidade}

In this subsection we introduce the notion of continuity for a family of vector subspaces and prove several properties related to this concept. The results of this subsection will be required later, in Subsection \ref{subseccompatible}, to construct the compatible subbundles.

Let us introduce two notions of continuity of subbundles.
\begin{definition}\label{projecaoz}
    Let $X$ be a Hilbert space and $(M,d)$ be a metric space. For each $m\in M$ fix a closed subspace $E(m)$ of $X$. 
    We say that the family $\{E(m):m\in M\}$ is \textbf{continuous} if there exists a family of projections $\{P(m):m\in M\}\subset \mathcal{L}(X)$ such that
    $m\mapsto P(m)$ is continuous in $\mathcal{L}(X)$ and $R(P(m))=E(m)$ for all $m\in M$. We will often refer a continuous family $E$ as a continuous subbundle on $M$.
\end{definition}
\begin{definition}\label{projecao}
    Let $X$ be a Hilbert space and $(M,d)$ a metric space. For each $m\in M$ fix a closed subspace $E(m)$ of $X$ and take $P(m)\in\mathcal{L}(X)$ as the orthogonal projection onto $E(m)$. We say that the family $\{E(m):x\in M\}$ is {\bf orthogonally continuous} if the family of projections $\{P(m):m\in M\}$ is continuous in $\mathcal{L}(X)$.
\end{definition}

 The following notation will be use through this whole section.
\begin{remark}\label{defprosepa}
    Let $X$ be a Banach space. If we have the direct sum $V\oplus W=X$, with both subspaces $V,W$ closed, we will denote by $P_{VW}$ the projection $x_v+x_w\mapsto x_v$. Note that $P_{VW}$ is continuous, since $V,W$ are closed, and $P_{VW}+P_{WV}=I$, where $I$ is the identity. In this notation, if $X$ is a Hilbert space, the projection $P_{VV^{\perp}}$ is the orthogonal projection onto the closed subspace $V$ and we simplify this notation by $P_V:=P_{VV^{\perp}}$.
\end{remark}

\begin{remark}\label{permanenciadim}
From Lemma \ref{projfechada} it follows  that if $X$ is a Banach space and $\{E(m):m\in M\}$ is a continuous family of subspaces of $X$, then for each $m_0\in M$ there exists a neighborhood $V_{m_0}\ni m_0$ such that $\dim E(m)=\dim E(m_0)$ for all $m\in V_{m_0}$.
\end{remark}
The next result shows that Definitions \ref{projecaoz} and \ref{projecao} are equivalent in Hilbert spaces.
\begin{lemma}\label{equivdefproj}
Let $X$ be a Hilbert space and $Q_n,Q\in\mathcal{L}(X)$ be continuous projections such that $Q_n\to Q$ in $\mathcal{L}(X)$. 
Denoting $S_n:=R(Q_n)$ and $S=R(Q)$, define the orthogonal projections 
$$P_n:=P_{S_n}\ \text{and}\ P:=P_{S}.$$
Then $P_n\to P$ in $\mathcal{L}(X)$.
\end{lemma}
\begin{proof}
It is sufficient to show that $P(I-P_n)\to 0\ \text{and}\ (I-P)P_n\to 0$ in $\mathcal{L}(X)$.
We will prove that $(P-I)P_n\to 0$ in $\mathcal{L}(X)$ and the other case will follow 
analogously by Lemma  \ref{gj}.
Given $v\in X$, $\parallel v\parallel=1$, it follows from the fact that $P_n, P$ are orthogonal projections that 
\begin{align*}
    \parallel (P-I)P_nv\parallel&=\parallel PP_nv-P_nv\parallel\\
    &\le\parallel QP_nv-P_nv\parallel\\
    &=\parallel QQ_nP_nv-Q_nP_nv\parallel\\
    &\le \parallel (Q-I)Q_n\parallel_{\mathcal{L}(X)}\to 0.
\end{align*}
\end{proof}

        From now on, for each $x\in W^s_{loc}(x^*_i)\cap W^u(x^*_j)$ we will denote by $\mathcal{S}(x)$ and $\mathcal{U}(x)$ the subspaces $T_xW^s_{loc}(x^*_i)$ and $T_xW^u(x^*_j)$ respectively, that are well defined by Remark \ref{bebosim}. It follows from \eqref{caraatrator} that $\mathcal{U}(x)$ is defined for any $x\in\mathcal{A}$  but $\mathcal{S}(x)$ is well defined only if $x$ is sufficiently close to an equilibrium point $x^*$. Now we will extend the family $\mathcal{S}$ to the whole global attractor in a way that still holds its original properties (see \eqref{miser} and \eqref{originalpro}).

\begin{proposition}\label{extensaos}
The family $\mathcal{S}$ can be extended (not necessarily continuously) to the whole attractor $\mathcal{A}$ in a way that still satisfies
\begin{equation}\label{miser}
    \mathcal{S}(x)+\mathcal{U}(x)=X,\ \forall x\in\mathcal{A}
\end{equation}
and
\begin{equation}\label{originalpro}
     (D_x\mathcal{T}(t))\mathcal{S}(x)\subset \mathcal{S}(\T(t)x),\ \forall x\in \mathcal{A},\ \forall t\ge 0.
\end{equation}
\end{proposition}
\begin{proof}
Fix $r>0$ such that $W^s_{loc}(x^*_i)\cap B[x^*_i,r]$ is part of the graph of a $\mathcal{C}^1$ map $$\Psi:\mathcal{S}(x^*_i)\to \mathcal{U}(x^*_i)$$
and $B_i:=B[x^*_i,r]$ (the closed ball in $X$ of center $x^*_i$ and radius $r$) satisfies
\begin{equation*}
    x\in W^s(x^*_i)\cap B_i\Rightarrow \T(t)x\in  B_i,\ \forall t\ge 0.
\end{equation*}

Then $W^s(x^*_i)\cap B_i=W^s_{loc}(x^*_i)\cap B_i$ is a $\mathcal{C}^1$ manifold and  $\mathcal{S}(x)=T_xW^s_{loc}(x^*_i)$ is well defined for each $x\in W^s(x^*_i)\cap B_i$.  Given $x\in (\mathcal{A}\cap W^s(x^*_i))\setminus B_i$ fix $t_0=t_0(x)>0$  as
\begin{equation}\label{free}
   t_0=\min\{t>0: \T(t)x\in B_i\}.
\end{equation}
 Denoting $x_0=\T(t_0)x\in B_i$ we define $\mathcal{S}(x)$ as below
\begin{equation*}
    \mathcal{S}(x)=(D_x\T(t_0))^{-1}\mathcal{S}(x_0)=(D_{x_0}\T(-t_0))S(x_0).
\end{equation*}
From a straightforward computation we conclude that the extension $\{\mathcal{S}(x)\}_{x\in\mathcal{A}}$ satisfies \eqref{originalpro}. Equality \eqref{miser} follows from item (5) of Definition \ref{definitionms}, Lemma \ref{iminv} and the fact that $$U(x)=(D_{\mathcal{T}(-t)x}\mathcal{T}(t))U(\mathcal{T}(-t)x)\subset R(D_{\mathcal{T}(-t)x}\mathcal{T}(t))$$ for all $x\in\mathcal{A}$ and $t\ge0$.
\end{proof}

Now we will construct a special neighborhood (in $\mathcal{A}$) $V_i\ni x^*_i$ for each $x^*_i\in\Omega$. In fact, for each $x^*_i\in\Omega$ we want to find a neighborhood $V_i\ni x^*_i$  and continuous subbundles $S'_i, U_i'$ on $V_i$ such that 
\begin{equation}\label{remarkmudancacoordenadas}
    \begin{aligned}
    U_i'(x)=\mathcal{U}(x),\ \ \text{for all}\ \ x\in W^u(x^*_i)\cap V_i,\\
    S_i'(x)=\mathcal{S}(x), \ \ \text{for all}\ \ x\in W^s(x^*_i)\cap V_i
    \end{aligned}
\end{equation}
and
\begin{equation}\label{mudancacoordenadas}
 S_i'(x)\oplus U_i'(x) = X, \forall x\in V_i.
\end{equation}
These subbundles will be  very helpful later, since they can be used later as a ``change of coordinates''  in   Theorem \ref{lemma2.2.9}, which is the essential key to prove Lipschitz Shadowing. Lemmas \ref{dugtheorem} and \ref{lemaproj2} below are dedicated to the proof of properties \eqref{remarkmudancacoordenadas} and \eqref{mudancacoordenadas}, respectively.

\begin{lemma}\label{dugtheorem}
Let $x^*\in\Omega$ and $\delta>0$ such that $W^u_{\delta}(x^*)$ and  $W^s_{\delta}(x^*)$ are graphs (see Remark \ref{bebosim}) of  $\mathcal{C}^1-$maps 
\begin{equation*}
    \Psi_u:U_0(x^*)\to S_0(x^*) \ \ \text{and}\ \ \Psi_s:S_0(x^*)\to U_0(x^*)
\end{equation*}
where $S_0(x^*)\oplus U_0(x^*)=X$ and $U_0(x^*)$ is finite dimensional. Consider the continuous family of projections $\{P_s(x):x\in W^s_{\delta}(x^*)\}$ and $\{P_u(x):x\in W^u_{\delta}(x^*)\}$, where 
\begin{equation*}
    R(P_s(x))=T_xW^s_{\delta}(x^*)\ \ \text{and}\ \ R(P_u(x))=T_xW^u_{\delta}(x^*). 
\end{equation*}
 Then, there exist continuous extensions of these families in the open ball $B(x^*,\delta)$.
\end{lemma}
\begin{proof}
Given $x\in B(x^*,\delta)$ we know from Definition \ref{hyperbolicset} that $x=x_s+x_u$, where $x_s\in S_0(x^*)$ and $x_u\in U_0(x^*)$. Define $P_s(x)=P_s(x_s+\Psi_s(x))$ and $P_u(x)=P_u(x_u+\Psi_u(x))$ for each $x\in B(x^*,\delta)$. Then $\{U_i'(x):=R(P_u(x))\}_{x\in B(x^*,\delta)}$ 
 and $\{S_i'(x):=R(P_u(x))\}_{x\in B(x^*,\delta)}$ are continuous extensions of $\mathcal{S}$ and $\mathcal{U}$ in $B(x^*,\delta)$.
\end{proof}

Now we prove property \eqref{mudancacoordenadas}. In fact, we will show that small (continuous) perturbations of direct sums are still direct sums.
\begin{lemma}\label{lemaproj2}
     Let $X$ be a Hilbert space and $(M,d)$ be a metric space. Assume that the families of closed subspaces $\{S(m):m\in M\}$ and $\{U(m):m\in M\}$ are continuous, $dim\ U(m)=n<+\infty$ for all $m\in M$ and $$S(m_0)\oplus U(m_0)=X$$ for some $m_0\in M$. Then there exists $\epsilon>0$ such that $S(m)\oplus U(m)=X$, for all $m\in B(m_0,\epsilon)$.
\end{lemma}
	\begin{proof}
 Let $\{P(m):m\in M\}$ and $\{Q(m):m\in M\}$ be the orthogonal
 projections related to $\{S(m):m\in M\}$ and $\{U(m):m\in M\}$ respectively. 
 
 We will first prove that if $S(m_0)\oplus U(m_0)=X$ for some $m_0\in M$, then there exists $\epsilon>0$ such that $S(m)\cap U(m)=\{0\}$, for all $m\in B(m_0,\epsilon)$. Note that $S(m_0)\cap U(m_0)=\{0\}$ if and only if $P(m_0)-Q(m_0)$ is injective. Since the families of projections are continuous and the set of injective linear maps are open in $\mathcal{L}(X)$, it follows that there exists $\epsilon>0$ such that $S(m)\cap U(m)=\{0\}$, for all $m\in B(m_0,\epsilon)$.
 
 Now, we have to prove that if $S(m_0)+ U(m_0)=X$ for some $m_0\in M$, then there exists $\epsilon>0$ such that $S(m)+ U(m)=X$, for all $m\in B(m_0,\epsilon)$. 
First, we prove that 
$S(m_0)+U(m_0)=X$ if and only if 
$P(m_0)+Q(m_0)$ is surjective and the conclusion will follow from the continuity of the projections and the fact that the set of surjective maps is open in $\mathcal{L}(X)$. 
Note that it is obvious that if $P(m_0)+Q(m_0)$ is surjective then  $S(m_0)+U(m_0)=X$, thus it remains to show the other way. Suppose that $S(m_0)+U(m_0)=X$ and take $v\in X$ arbitrary. 
Then, there exists $v_s\in S(m_0),v_u\in U(m_0)$  such that
		$v=v_s+v_u$. Since $S(m_0)^{\perp}+U(m_0)^{\perp}=X$, because $S(m_0)\cap U(m_0)=\{0\}$ and $U(m_0)$ is finite dimensional, we can take $w_s\in S(m_0)^{\perp}$ and $w_u\in U(m_0)^{\perp}$ such that
		$$v_s-v_u=w_s+w_u.$$
		Define $\tilde{v}=v_s-w_s=v_u+w_u$. Then $(P+Q)\tilde{v}=v$.
	\end{proof}
    
This concludes  the existence of a neighborhood $V_i\ni x_i^*$ and  continuous subbundles $S_i',U'_i$ in $V_i$ satisfying \eqref{remarkmudancacoordenadas} and \eqref{mudancacoordenadas}. 

We now start the construction of the compatible subbundles. Fixing $i\in\{1,\dots,p\}$ and $d\in (0,1)$, consider the set 
\begin{equation}\label{definicaoD}
D=V^{-1}(i+d)\cap W^s(x^*_i)\cap\mathcal{A},
\end{equation}
and let $Q$ be a compact neighborhood of $D$ in $V^{-1}(i+d)\cap\mathcal{A}$ with 
$Q\cap W^u(x^*_i)=\emptyset.$
Assume that $\mathcal{V}$ is a continuous subbundle
defined in $Q$ such that
\begin{equation}\label{coordenadasnu}
\mathcal{S}(x)\oplus \mathcal{V}(x)=X, \hbox{ for } x\in D.
\end{equation}
Now, we  extend the subbundle $\mathcal{V}$  to the set 
$\gamma(Q)^+\cup W^u(x^*_i)$ as below:
\begin{equation*}
    \mathcal{V}(x):=\mathcal{U}(x),\ \forall x\in W^u(x_i^*),
\end{equation*}
	\begin{equation*}
    \mathcal{V}(\mathcal{T}(t)x):=(D_x\mathcal{T}(t))\mathcal{V}(x),\ \forall x\in Q,\ \forall t\ge 0
\end{equation*}
Note that each element of the extension of $\mathcal{V}$ is a closed subspace of $X$ (the derivative is an isomorphism onto its range). Therefore, we can analyze the continuity of the family $\mathcal{V}$. Our next goal is to show that we can reduce $Q$ in a way that the subbundle $\mathcal{V}$ becomes continuous in $\gamma^+(Q)\cup W^u(x^*_i)$.

Firstly, we will focus on results that guarantee the continuity in $\gamma^+(Q)$ (see Lemma \ref{lemaajudacoro}). Let us start with a technical lemma.

\begin{lemma}\label{decaimentoperturbado}
Let $X,Y$ be normed vector spaces and $T_n,T\in \mathcal{L}(X,Y)$ be injective bounded linear maps such that $T_n^{-1}:R(T_n)\to X$ and $T^{-1}:R(T)\to X$ are bounded and $T_n\to T$ in $\mathcal{L}(X,Y)$. Then
\begin{equation*}
   \sup\limits_{n\in\mathbb{N}} \parallel T_n^{-1}\parallel_{\mathcal{L}(R(T_n),X)}<+\infty.
\end{equation*}
\end{lemma}
\begin{proof}
Assume that there exists a sequence $v_n\in R(T_n)$, with $\parallel v_n\parallel_Y=1$ for all $n$, such that
\begin{equation*}
    \parallel T_n^{-1}v_n\parallel_X\xrightarrow{n\to+\infty}+\infty.
\end{equation*}
Defining $w_n=T_n^{-1}v_n\in X$, we have
\begin{equation*}
    T_n\left(\dfrac{w_n}{\parallel w_n\parallel_X}\right)=\dfrac{v_n}{\parallel w_n\parallel_X}\xrightarrow{n\to+\infty} 0,
\end{equation*}
which contradicts the fact that $T$ is an isomorphism onto its image. 
\end{proof}

The following result will guarantee that if $X$ is a Hilbert space,  $\mathcal{V}=\{\mathcal{V}(x)\subset X:x\in M\}$ is a continuous subbundle in the metric space $M$ and $\{T(t):t\in [a,b]\}\subset\mathcal{L}(X)$ is a family of maps such that $T(t)$  is an isomorphism onto its range for all $t\in[a,b]$ and the map $[a,b]\ni t\mapsto T(t)$ is continuous in $\mathcal{L}(X)$, then $\{T(t)\mathcal{V}(x):x\in M,t\in [a,b]\}$ is a continuous subbundle.

\begin{lemma}\label{lemaajudacoro}
Let $X,Y$ be Hilbert spaces, $T_n,T\in \mathcal{L}(X,Y)$ be isomorphisms onto its images and $Q_n,Q\in\mathcal{L}(X)$ be continuous projections such that 
\begin{equation*}
    T_n\to T\ \text{in}\ \mathcal{L}(X,Y) \ \text{and}\ \ Q_n\to Q \ \ \text{in}\ \ \mathcal{L}(X).
\end{equation*}
Define $P_n$ the orthogonal projection onto $R(T_nQ_n)$ and $P$ the orthogonal projection onto $R(TQ)$. Then $P_n\to P$ in $\mathcal{L}(Y)$. In the particular case $Q_n=Q=I$ we have $R(T_n)\to R(T)$, in the sense of Definition \ref{projecao}.
\end{lemma}
\begin{proof}
It is sufficient to show that $(I-P)P_n\to 0$ and $P(P_n-I)\to 0$ in $\mathcal{L}(Y)$. We will only prove that $(I-P)P_n\to 0$ in $\mathcal{L}(Y)$ and the other case follows analogously by Lemma \ref{gj}. Given $v\in X$, $\parallel v\parallel_Y\le 1$, we have
\begin{equation*}
    \parallel P_nv-PP_nv\parallel_Y=d_Y(P_nv,R(P))=d_Y(P_nv,R(TQ)).
\end{equation*}
Since $P_nv=T_nQ_nv_n$ for some $v_n\in X$ with $\parallel v_n\parallel_X\le \parallel T_n^{-1}\parallel_{\mathcal{L}(R(T_n),X)}$ (we can consider $v_n=Q_nv_n$), we have
\begin{equation*}
    \parallel P_nv-PP_nv\parallel_Y\le \parallel T_nQ_nv_n-TQv_n\parallel_Y.
\end{equation*}
From Lemma \ref{decaimentoperturbado} there exists $M>0$ that bounds $\parallel T_n^{-1}\parallel_{\mathcal{L}(R(T_n),X)}$ for all $n$ and therefore
\begin{equation*}
    \sup\limits_{\parallel v\parallel=1}\parallel P_nv-PP_nv\parallel_Y\le M\parallel T_nQ_n-TQ\parallel_{\mathcal{L}(X,Y)}\to 0.
\end{equation*}
\end{proof}

With this result we are ready to prove the continuity of the subbundle $\mathcal{V}$  in $\gamma^+(Q)$ (Lemma \ref{lemma-2.2.7}). Despite that, we now proceed with a technical lemma that will help us to prove of the continuity of the subbundle $\mathcal{V}$ in $\gamma(Q)^+\cup W^u(x^*_i)$. This lemma will connect the notion of inclination, introduced in \cite{palis1969morse}, and continuity of subbundles (Definition \ref{projecao}).

\begin{lemma}\label{mudandolemma227}
     Let $(M,d)$ be a metric space, $X$ be a Hilbert space and  $\{S'(x)\subset X:x\in M\}$, $\{U'(x)\subset X:x\in M\}$ and $\{\mathcal{V}(x)\subset X:x\in M\}$ be continuous subbundles such that 
    \begin{align*}
        &S'(x)\oplus U'(x)=X,\ \forall x\in M,\\
        &S'(x)\cap \mathcal{V}(x)=0,\ \forall x\in M.
    \end{align*}
Given $v\in\mathcal{V}(x)$ non nulle, we define the inclination of $v$ in relation to the decomposition $S'(x)\oplus U'(x)=X$ as $\alpha(v):=\dfrac{|v^s|}{|v^u|}$. Fix $x_0\in M$ and a sequence $x_n\in M$ such that $x_n\to x_0$ as $n\to+\infty$. If for any sequence $v_n\in\mathcal{V}(x_n)$ (non null) we have $\alpha(v_n)\to 0$ as $n\to+\infty$, then $\mathcal{V}(x_n)$ converges to $U'(x_0)$ as $n\to+\infty$, in the sense of Definition \ref{projecao}.

\begin{proof}
Let $P_{U'(x)}$ and $P_{\mathcal{V}(x)}$ be the orthogonal projections onto $U'(x)$ and $\mathcal{V}(x)$ respectively, and $P_{U'(x)^{\perp}}=I-P_{U'(x)}$, $P_{\mathcal{V}(x)^{\perp}}=I-P_{\mathcal{V}(x)}.$
 
 If $\mathcal{V}(x_n)$ does not converges to $U'(x_0)$, then  we have two possible cases: 
 \begin{equation*}
     P_{U'(x_0)^{\perp}}P_{\mathcal{V}(x_n)}\nrightarrow 0\ \text{in}\ \mathcal{L}(X)\ \ \text{or}\ \ P_{U'(x_0)}P_{\mathcal{V}(x_n)^{\perp}}\nrightarrow 0\ \text{in}\ \in\mathcal{L}(X)
 \end{equation*}
 We will prove by contradiction that the first case does not hold case  and the second case will follow analogously. Assume that there exists $\epsilon\in (0,1)$ and a  subsequence of $\{x_n\}_{n\in\mathbb{N}}$,  still denoted by $\{x_n\}_{n\in\mathbb{N}}$, such that $x_n\to x_0$ as $n\to+\infty$ and  
	\begin{equation}\label{contradic}
	    \parallel P_{U'(x_0)^{\perp}}P_{\mathcal{V}(x_n)}w_n\parallel\ge\epsilon,\ \forall n\in\mathbb{N},
	\end{equation}
	 for some sequence $\{w_n\}_{n\in\mathbb{N}}\subset X$, with $\parallel w_n\parallel\le 1$. 
     
	Define $v_n=P_{\mathcal{V}(x_n)}w_n$. Take $\epsilon'>0$ such that
	\begin{equation*}
	    \dfrac{\epsilon'}{1-\epsilon'}<\epsilon
	\end{equation*}
and fix	large $N\in\mathbb{N}$ such that
	\begin{equation}\label{fff}
	     \dfrac{|v_n^s|}{|v_n^u|}<\epsilon',\ \forall n>N,
	\end{equation}
that is
\begin{equation*}
    -\epsilon'|v_n^u|<-|v^s_n|,\ \forall n>N.
\end{equation*}

Hence
\begin{equation*}
    (1-\epsilon')\parallel v_n^u\parallel\le\parallel v_n^u \parallel-\parallel v^s_n\parallel \le \parallel v_n^s+v_n^u\parallel=\parallel v_n\parallel,\ \forall n>N,
\end{equation*}
which implies
\begin{equation}\label{riguii}
    \parallel v^u_n\parallel\le \dfrac{\parallel v_n\parallel}{1-\epsilon'},\ \forall n>N.
\end{equation}

Putting \eqref{fff} together with \eqref{riguii} we obtain
\begin{equation*}
    \parallel v_n-v^u_n\parallel=\parallel v_n^s\parallel\le \epsilon'\parallel v_n^u\parallel\le\dfrac{\epsilon'}{1-\epsilon'}\parallel v_n\parallel\le \dfrac{\epsilon'}{1-\epsilon'}<\epsilon,\ \forall n>N.
\end{equation*}

Since $P_{U'(x_0)}$ is the orthogonal projection onto $U'(x_0)$, then 
\begin{equation*}
 \parallel P_{U(x_0)^{\perp}}P_{\mathcal{V}(x_n)}w_n \parallel =\parallel(I-P_{U(x_0)})v_n \parallel=\parallel v_n- P_{U(x_0)}v_n\parallel\le \parallel v_n-v^u_n\parallel
<\epsilon,\ \forall n>N,
\end{equation*}
which contradicts \eqref{contradic} and finishes the proof.
\end{proof}
\end{lemma}

Now we finally prove the continuity of the subbundle $\mathcal{V}$ in $\gamma(Q)^+\cup W^u(x^*_i)$.

\begin{lemma}\label{lemma-2.2.7}
	Let $V_i$ be a neighborhood of $x^*_i$ satisfying properties \eqref{remarkmudancacoordenadas} and \eqref{mudancacoordenadas}. Fix $d'>0$  such that the properties stated in Lemma \ref{lemma2.2.6} holds for the neighborhood $V_i$. Take $d\in (0,d')$ and define the set $D$ as in \eqref{definicaoD}. Assume that $Q$ is a neighborhood (in $V^{-1}(i+d)\cap\mathcal{A}$) of $D$ and $\mathcal{V}$ is a continuous  subbundle on $Q$ satisfying \eqref{coordenadasnu}. Then, we can reduce $Q$ such that $\mathcal{V}$ is continuous in 
	$(\gamma(Q)^+\cup W^u(x^*_i))\cap V_i$.
\end{lemma}
\begin{proof}
   First we prove the continuity of $\mathcal{V}$ in $\gamma^+(Q)$.  We just have to show that the subbundles $\mathcal{V}_{t_0}:=\{(D_{x}\mathcal{T}(t_0))\mathcal{V}(x)\}_{x\in Q}$ and $\mathcal{V}_{x_0}:=\{\{(D_{x_0}\mathcal{T}(t))\mathcal{V}(x_0)\}_{t\ge0}$ are continuous for each $t_0>0$ and $x_0\in Q$. We start by proving the continuity of $\mathcal{V}_{t_0}$.

Fix $t_0>0$ and assume that $x,x_n\in Q$  and $\mathcal{T}(t_0)x_n\to \mathcal{T}(t_0)x$ as $n\to+\infty$.  Then $ x_n\to x$ ($\mathcal{T}|_{\mathcal{A}}$ is a group) and $\mathcal{V}(x_n)\to\mathcal{V}(x)$, in the sense of Definition \ref{projecaoz}. Since $\mathcal{T}(t_0)\in\mathcal{C}^1(X)$, we can use  Lemma \ref{lemaajudacoro} with $T_n=D_{x_n}\mathcal{T}(t_0)$, $T=D_{x}\mathcal{T}(t_0)$ and $Q_n=P_{\mathcal{V}(x_n)},Q=P_{\mathcal{V}(x)}$ to guarantee that
\begin{equation*}
    \mathcal{V}(\mathcal{T}(t_0)x_n)=(D_{x_n}\mathcal{T}(t_0)) \mathcal{V}(x_n)\to (D_{x}\mathcal{T}(t_0))\mathcal{V}(x)=\mathcal{V}(\mathcal{T}(t_0)x),
\end{equation*}
where the convergence is in the sense of projections (Definitions \ref{projecaoz} and \ref{projecao}). Hence the subbundle $\mathcal{V}_{t_0}$ is continuous for each $t_0\ge0$. 

Now we show the continuity of $\mathcal{V}_{x_0}$, for $x_0\in D$ fixed. From hypothesis (H2) (from Theorem \ref{teointroducao}) it follows that for each $t\ge 0$ it holds
\begin{equation}\label{solouchuva}
  (0,+\infty)\ni t_n\to t\Longrightarrow (D_{x_0}\mathcal{T}(t_n))v\to  (D_{x_0}\mathcal{T}(t))v,\ \text{for all}\ v\in X.
\end{equation}

 Since $\mathcal{V}(x_0)$ is finite dimensional (because the unstable manifolds are finite dimensional), it follows from \eqref{solouchuva} that
\begin{equation*}
  t_n\to t \Longrightarrow (D_{x_0}\mathcal{T}(t_n))|_{\mathcal{V}(x_0)}\to  (D_{x_0}\mathcal{T}(t))|_{\mathcal{V}(x_0)}\ \text{in}\ \mathcal{L}(\mathcal{V}(x_0),X).
\end{equation*}

Therefore we can use Lemma \ref{lemaajudacoro} once more to guarantee the continuity of the subbundle $\mathcal{V}_{x_0}$, for each $x_0\in Q$. This finishes the proof of the continuity of $\mathcal{V}$ in $\gamma^+(Q)$.

Now we focus on the continuity of $\mathcal{V}$ in $\gamma^+(Q)\cup W^u(x^*_i)$. In \cite[Lemma 2.2.7]{Pilyugin} the author proves, using the ideas introduced by Palis  \cite{palis1969morse}, that we can reduce $Q$ such that the inclination of the vectors $v(x_n)\in \mathcal{V}(x_n)$ related to the decomposition $S_i'(x_n)\oplus U_i'(x_n)=X$ converges to $0$ as $x_n\in\gamma ^+(Q) $ approaches the unstable manifold $W^u(x_i^*)$. In other words, if $\gamma^+(Q)\ni x_n\to x\in W^u(x_i^*)$ and $v_n\in \mathcal{V}(x_n)$, then $\alpha(v_n)\to 0$ as $n\to+\infty$. Hence, it follows from Lemma \ref{mudandolemma227} that we can reduce $Q$ such that $\mathcal{V}$ is continuous, in the sense of Definitions \ref{projecaoz} and \ref{projecao}, in  $\gamma^+(Q)\cup W^u(x^*_i)$.    
	\end{proof}

Now that we have proved the continuity of $\mathcal{V}$ in $\gamma^+(Q)\cup W^u(x_i^*)$, we want to find conditions for the extension of 
$\mathcal{V}$  in $\gamma^-(Q)$
\begin{equation*}
    \mathcal{V}(\mathcal{T}(-t)x):=(D_x\mathcal{T}(-t))\mathcal{V}(x),\ \forall\ x\in Q,\ \forall t\ge0
\end{equation*}
to be continuous in $\gamma^-(Q)$. This continuity does not follow from Lemma \ref{lemaajudacoro}, since the Fréchet derivative $D_x\mathcal{T}(t)$ is not an isomorphism for $x\in\mathcal{A}$. Obtaining this ``backwards continuity'' is be essential to construct one of the compatible subbundles in Theorem \ref{lemma2.2.9}, as we will see further on.

The following result  is the first step to obtain the backwards continuity of $\mathcal{V}$.
\begin{lemma}\label{contiparatras}
    Let $X$ be a Hilbert space, $Q_n,Q\in\mathcal{L}(X)$ be projections and $T_n,T\in\mathcal{L}(X)$ be isomorphisms onto its images such that $Q_n\to Q$  in $\mathcal{L}(X)$ and $T_n\to T$  in $\mathcal{L}(X)$. If $R(Q_n)\subset R(T_n)$ for all $n$, and $R(Q)\subset R(T)$, then $T_n^{-1}Q_n\to T^{-1}Q$ in $\mathcal{L}(X)$. Moreover, if $P_n$ and $P$ are the orthogonal projections onto $R(T_n^{-1}Q_n)$ and $R(T^{-1}Q)$ respectively, then  $P_n\to P$ in $\mathcal{L}(X)$.
\end{lemma}
\begin{proof}
    We know from Lemma \ref{decaimentoperturbado} that $\sup\limits_{n\in\mathbb{N}}\parallel T_n^{-1}\parallel_{\mathcal{L}(R(T_n),X)}$ is bounded. Hence
    \begin{equation*}
        \parallel T^{-1}Q-T_n^{-1}Q_n\parallel_{\mathcal{L}(X)}\le \parallel T_n^{-1}\parallel_{\mathcal{L}(R(T_n),X)}\parallel T_nT^{-1}Q-Q_n\parallel_{\mathcal{L}(X)}\to 0.
    \end{equation*}
    In order to prove that $P_n\to P$ in $\mathcal{L}(X)$, we will show that $(I-P)P_n\to 0$ and $P(P_n-I)\to 0$ in $\mathcal{L}(X)$. We will prove only that $(I-P)P_n\to 0$ and the other case will follow analogously by Lemma \ref{gj}. Given $v\in X$ with $\parallel v\parallel=1$ we have
    \begin{equation*}
        \parallel PP_nv-P_nv\parallel= d(P_nv,R(P))=d(P_nv,R(T^{-1}Q)).
    \end{equation*}
    Since $P_nv=T_n^{-1}Q_nv_n$, with $\parallel v_n\parallel=\parallel Q_nv_n\parallel\le M:=\sup\limits_{n\in\mathbb{N}}\parallel T_n\parallel<\infty$ for all $n$, we have
    \begin{align*}
         \parallel PP_nv-P_nv\parallel&\le \parallel T_n^{-1}Q_nv_n-T^{-1}Qv_n\parallel\le M\parallel T_n^{-1}Q_n- T^{-1}Q\parallel_{\mathcal{L}(X)}\\
         &\Rightarrow\sup\limits_{\parallel v\parallel=1}\parallel PP_nv-P_nv\parallel\le M\parallel T_n^{-1}Q_n- T^{-1}Q\parallel_{\mathcal{L}(X)}\to 0.
    \end{align*}
\end{proof}

\begin{corollary}\label{normainversaconv}
     Let $X$ be a Hilbert space and $T_n,T\in\mathcal{L}(X)$ be isomorphisms onto its ranges such that $T_n\to T$ in $\mathcal{L}(X)$. Then $$\parallel T_n^{-1}\parallel_{\mathcal{L}(R(T_n),X)}\to \parallel T^{-1}\parallel_{\mathcal{L}(R(T),X)}.$$
\end{corollary}
\begin{proof}
    Just apply Lemma \ref{contiparatras} with $Q_n$ and $Q$ being the orthogonal projections onto $R(T_n)$ and $R(T)$ respectively (they are continuous from Lemma \ref{lemaajudacoro}) and note that 
    \begin{equation*}
        \parallel T_n^{-1}P_n\parallel_{\mathcal{L}(X)}=\parallel T_n^{-1}\parallel_{\mathcal{L}(R(T_n),X)}\ \ \text{and}\ \ \parallel T^{-1}P\parallel_{\mathcal{L}(X)}=\parallel T^{-1}\parallel_{\mathcal{L}(R(T),X)}
    \end{equation*}
\end{proof}

 Note that if we  have a continuous subbundle $\{G(x)\}_{x\in M}$ in a metric space $M$ and we want the continuity of the subbundle $\{T^{-1}G(x)\}_{x\in M}$, where $T$ is an isomorphism onto its image, then we must have $G(x)\subset R(T)$ in order to apply  Lemma \ref{contiparatras}. Since
\begin{equation*}
    T^{-1}G(x)=T^{-1}(G(x)\cap R(T)),\ \forall x\in M
\end{equation*}
and $(G(x)\cap R(T))\subset R(T)$, then we just have to guarantee  the continuity of the subbundle $\{G(x)\cap R(T)\}_{x\in M}$ to apply Lemma \ref{contiparatras}.
Unfortunately, the continuity of two  subbundles  $\{E(x)\}_{x\in M}$ and $\{F(x)\}_{x\in M}$  does not implies that  $\{E(x)\cap F(x)\}_{x\in M}$ is a continuous subbundle. For example, if we consider the Hilbert space $X=\mathbb{R}^2$, the metric space $M=[0,+\infty)$ and fix the continuous subbundles $\{F_m=\langle (1,1)\rangle\}_{m\ge0}$ and $\{E_m=\langle (1,m)\rangle\}_{m\ge0}$, where $\langle v\rangle$ denotes the subspace generated by $v\in \mathbb{R}^2$, then $F_m\cap E_m=\{0\}$ if $m\neq 1$ and $F_1=\langle (1,1)\rangle$. Therefore, the subbundle $\{F_m\cap E_m\}_{m\ge0}$ is not continuous at $m=1$ (see Remark \ref{permanenciadim}).

\begin{figure}[H]
    \centering
    \begin{tikzpicture}
        \draw[line width=0.2mm,->](0,-2)--(0,3);
        \draw[line width=0.2mm,->](-2,0)--(3,0);
        \draw[line width=0.4mm,blue,-](-1.54,-1.54)--(2.5,2.5);
        \node at (4,2.5)[above] { $F_m=<(1,1)>$};
         \draw[line width=0.2mm,red,-](-1.8,-0.9)--(3.2,1.6);
         \node at (3.2,1.6)[right] { $E_m$};
         \draw[line width=0.2mm,red,-](-1.7,-1.19)--(2.9,2.03);
         \draw[line width=0.2mm,red,-](-0.5,-0.4)--(2.8,2.24);
\draw[line width=0.2mm,red,-](-1.6,-1.44)--(2.6,2.34);  
\draw[line width=0.2mm,red,-](-1.35,-1.755)--(2.1,2.73);   
\draw[line width=0.2mm,red,-](-1.2,-1.8)--(1.9,2.85);   
\draw[line width=0.2mm,red,-](-1.5,-1.65)--(2.35,2.585);   
\draw[line width=0.2mm,red,-](-0.95,-1.9)--(1.5,3);   
\draw[line width=0.2mm,red,-](-0.65,-1.95)--(1.05,3.15);   
\end{tikzpicture}
    \caption{Discontinuity of the intersection of subbundles}
    \label{quemsabe}
\end{figure}
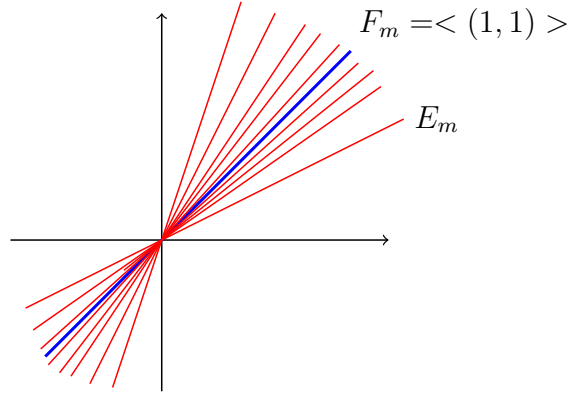

Therefore we need to find conditions that guarantees the continuity of the intersection of two continuous subbundles.

\begin{proposition}\label{propinterseccont}
    Let $(M,d)$ be a  metric space, $X$ be a Hilbert space and $\{E_m\subset X\}_{m\in M}$, $\{F_m\subset X\}_{m\in M}$ be  continuous subbundles. If 
    \begin{equation*}
        E_m+F_m=X\ \ \text{and}\ \ \dim (E_m)^{\perp}<+\infty,\ \text{for all}\ m\in M,
    \end{equation*}
     then the family $\{E_m\cap F_m\}_{m\in M}$ is continuous.
\end{proposition}
\begin{proof}
    From $E_m+F_m=X$ we have $E_m^{\perp}\cap F_m^{\perp}=\{0\}$ and from $\dim (E_m)^{\perp}<+\infty$ we know that $E_m^{\perp}\oplus F_m^{\perp}$ is closed in $X$. Since  $\{E_m\cap F_m\}_{m\in M}$ is continuous if, and only if $\{E_m^{\perp}\oplus F_m^{\perp}\}_{m\in M}$ is continuous, we  just have to show that the last subbundle is continuous. Take $m_n,m\in M$   such that $m_n\to m$ and define $P_n,P$ the orthogonal projections onto $E_{m_n}^{\perp}\oplus F_{m_n}^{\perp}$ and $E_m^{\perp}\oplus F_m^{\perp}$ respectively. We will show that $P_n\to P$ in $\mathcal{L}(X)$. Note that that $E_{m_n}$ is isomorphic to $E_m$ and $F_{m_n}$ is isomorphic to $F_m$ for large $n$ (see Lemma \ref{projfechada}). For each $n\in\mathbb{N}$ define the  linear map
    \begin{align*}
        \psi_n:E_{m_n}&\oplus F_{m_n}\to E_m\oplus F_m\\
        &e+f\mapsto (P_{E_m})e+(P_{F_m})f,
    \end{align*}
    where $P_{E_m}$ and $P_{F_m}$ are the orthogonal projections onto $E_m$ and $F_m$, respectively. Note that
    \begin{equation*}
        \parallel\psi_n-I|_{E_{m_n}\oplus F_{m_n}}\parallel_{\mathcal{L}(E_{m_n}\oplus F_{m_n},X)}\to 0.
    \end{equation*}
   In order to prove that $P_n-P\to 0$ in $\mathcal{L}(X)$, we will prove that 
   $(I-P)P_n\to 0$ and $P(I-P_n)\to 0$ in $\mathcal{L}(X)$. It is sufficient to prove that $(I-P)P_n\to 0$ in $\mathcal{L}(X)$ (the other case follows analogously by Lemma \ref{gj}).
    Given $v\in X$, $\parallel v\parallel\le 1$, we have
    \begin{equation*}
        \parallel (I-P)P_nv\parallel=d(P_nv,E_m\oplus F_m)\le \parallel P_nv-\psi_n(P_nv)\parallel\le\parallel \psi_n-I|_{E_{m_n}\oplus F_{m_n}}\parallel_{\mathcal{L}(E_{m_n}\oplus F_{m_n},X)}\to 0.
    \end{equation*}
\end{proof}

We finish the section defining the notion of lower semicontinuity and showing this property for the subbundles $\mathcal{S}$ and $\mathcal{U}$.

\begin{definition}[Lower Semicontinuous]
	    Let $(M,d)$ be a metric space and $\{E(x):x\in M\}$ be a  family of closed linear subspaces of a Hilbert $X$. The subbundle $\{E(x):x\in M\}$ is  \textbf{lower semicontinuous} at $x_0\in M$ if for any closed subspace $L\subset E(x_0)$ there exists a neighborhood $U\ni x_0$ and a continuous subbundle $\{F(x)\}_{x\in U}$ in $U$  such that $F(x_0)=L$ and $F(x)\subset E(x)$ for each $x\in U$.
	    The subbundle $\{E(x):x\in M\}$ is  lower semicontinuous if it is lower semicontinuous at every $x\in M$.
	\end{definition}

\begin{lemma}\label{lemma-2.2.8}
		The families $\{\mathcal{U}(x)\}_{x\in\mathcal{A}}$ and $\{\mathcal{S}(x)\}_{x\in\mathcal{A}}$ are lower semicontinuous.
	\end{lemma}
\begin{proof}
   We can replicate the proof in \cite[Lemma 2.2.8]{Pilyugin} to prove the lower semicontinuity for the family $\mathcal{U}$. In our case, where  the derivative $D_x\mathcal{T}(t)$ is not surjective, the proof of the lower semicontinuity of  the family $\mathcal{S}$ does not follow as in the finite dimensional case. Despite that, it is still possible to prove this property for $\mathcal{S}$. Since the strategy of the proof is  similar to the construction of the subbundles $S_i$   in Theorem \ref{lemma2.2.9}, we will not replicate it here. In fact,  the proof of Theorem \ref{lemma2.2.9} is even more complex. Therefore we choose to prove Theorem \ref{lemma2.2.9} rigorously instead of proving the lower semicontinuity of $\mathcal{S}$.
\end{proof}

\begin{remark}\label{remarklsc}
	Let $M$ be a metric space, $X$ be a Hilbert space and $\{E(x)\subset X\}_{x\in M}$ and $\{L(x)\subset X\}_{x\in M}$ be lower semicontinuous subbundles in $M$ such that
	$$E(x_0)+L(x_0)=X, $$
	for some $x_0\in M$. Then it follows from Lemma \ref{lemaproj2} that there exists a neighborhood  $V\ni x_0$  such that
	$$E(x)+L(x)=X,\ \forall x\in V.$$
	\end{remark}

\subsection{Compatible Subbundles}\label{subseccompatible}
In this subsection we construct the compatible subbundles \cite{robinson1974structural, Pilyugin} on the global attractor  $\mathcal{A}$. The construction of compatible subbundles was first announced in 
\cite{robinson1974structural}, where the phase space is a compact smooth manifold (without border). The reader can also consult the proof in \cite[Lemma 2.2.9]{Pilyugin}. We recall that the global attractor $\mathcal{A}$ is not necessarily a manifold in our case.

As we said before, the construction of the compatible subbundles will allow us later, with tools from functional analysis, to apply the Banach fixed point theorem to obtain the Lipschitz Shadowing property in the global attractor $\mathcal{A}$. To apply the Banach fixed point theorem we can follow similarly to the proof of Lipschitz Shadowing in finite dimension \cite{Pilyugin,santamaria2014distance}. In fact, the infinite dimension of $X$ is not a problem in this part of the proof. Therefore, we  just have to  the construct the compatible subbundle  to guarantee the Lipscthitz Shadowing property in $\mathcal{A}$.

The construction of the compatible subbundles in finite dimensional 
compact manifolds is  via  mathematical induction and  uses the surjectivity of the derivative $D_x\mathcal{T}(t)$, that does not hold in the infinite dimensional setting, to prove the induction hypothesis. To overcome this problem, we will introduce Lemmas \ref{feliz} and \ref{felizz}, related to the dynamics of the attractor $\mathcal{A}$.

Moreover, to construct one the subbundles we  need to  ``evolve backwards continuously''  with the derivative, which is not straightforwad if the derivative is not an isomorphism. In fact, we will use Lemmas \ref{contiparatras} and \ref{propinterseccont} to overcome this obstacle. 

Finally, if we have a diffeomorphism defined in  a compact (finite dimensional) manifold, we can use the Lyapunov norm \cite[Lemma 1.2.1]{Pilyugin} to assume without lost of generality that the constant of hiperbolicity $C$ (see Definition \ref{hyperbolicset}) is equal to $1$. This is not the case in the infinite dimensional setting. To solve this, we have Remark \ref{precisadisso}.

Let us now construct the compatible subbundles $S_i$ and $U_i$ for the infinite dimensional case.

\begin{theorem}\label{lemma2.2.9}
Let $X$ be a Hilbert space, $\{\mathcal{T}(t): t\geq 0\}$ be a Morse-Smale  semigroup in $X$ and\ \ $\Omega=\{x^*_1,\cdots,x^*_p\}$ be its non-wandering set.  Then, for each $i\in\{1,\dots, p\}$ there exist $\tilde{C}>0$, $\lambda_1\in(0,1)$, a neighborhood $\mathcal{O}_i\subset \mathcal{A}$ of $x^*_i$  and  subbundles $S_i,U_i$ defined in $\overline{\mathcal{O}_i}\cup\gamma(\mathcal{O}_i)$ and continuous in $\overline{\mathcal{O}_i}$
such that:

\begin{enumerate}
\item\label{item1} For all $x\in\gamma(\mathcal{O}_i)$  it holds
\begin{align*}
(D_x\mathcal{T}(t))S_i(x)\subset S_i(\mathcal{T}(t)x),\forall\ t\ge 0\ \ \text{and}\ \ (D_x\mathcal{T}(t))U_i(x)= U_i(\mathcal{T}(t)x),\forall\ t\in\mathbb{R}.
\end{align*}
\item\label{item2} $S_i(x_i^*)=\mathcal{S}(x_i^*)$ and $U_i(x_i^*)=\mathcal{U}(x_i^*)$ for all $i\in\{1,\dots,p\}$.
\item\label{item3} $S_j(x)\subset S_k(x)$, $U_k(x)\subset U_j(x)$, for every $x\in \gamma^+(\mathcal{O}_j)\cap \gamma^-(\mathcal{O}_k)$.
\item\label{item4} $S_i(x)\oplus U_i(x)=X$, for all $x\in \gamma^-(\mathcal{O}_i)$.
\item\label{item5} For every $x\in \mathcal{O}_i$ it holds
\begin{align}\label{hipfinita}
     &\|(D_x\mathcal{T}(t))v^s\|\leq \tilde{C}\lambda_1^t \|v^s\|, \forall v^s\in S_i(x),\ \forall t\in[0,1],\\
 &\label{hipinfinita2}\|(D_x\mathcal{T}(-t))v^u\|\leq \tilde{C}\lambda_1^t \|v^u\|, \forall v^u\in U_i(x),\ \forall t\in[0,1].
\end{align}
\end{enumerate}
\end{theorem}

\begin{proof}
First of all, note that items (\ref{item4}) and (\ref{item5}) are consequences of items (\ref{item1}), (\ref{item2}), the hyperbolicity of $x^*_i$ and the continuity of the subbundles (see Lemma \ref{lemaproj2}). Therefore we will prove only items (\ref{item1}), (\ref{item2}) and (\ref{item3}). We will first prove the existence of the subbundles $U_i$. The construction of $U_i$ actually follows as in the finite dimensional case (the subbundles $U_i$ will be finite dimensional), except in the last step, where we prove the induction hypothesis \ref{hipinducao}. In fact, to show \ref{hipinducao} in finite dimension the surjectivity of the derivative is required. We will reproduce the whole construction of $U_i$ here because its understanding will be important to construct the subbundles $S_i$.

If the derivative $D_x\mathcal{T}(t)$ is an isomorphism for all $x\in\mathcal{A}$ and $t\ge0$ (like the finite dimensional case), then the construction of $S_i$ follows analogously as the construction of $U_i$. In our case the construction of the family $S_i$ is  more complex since the derivative is not surjective and we still need to  ``evolve backwards continuously''  with the derivative. Hence, after we finish the construction of $U_i$, we will also provide a proof for the existence of the families $S_i$.

Let us start the construction of $U_i$. As in the finite dimensional case, we will do the proof by induction, assuming that for each $j\in\{i+1,\dots,p\}$ we have  neighborhoods $\mathcal{O}_j$ and  subbundles $U_j$ on $\gamma(\mathcal{O}_j)\cup \overline{\mathcal{O}_j}$ satisfying items (\ref{item1})-(\ref{item5}) and  the extra property below 
\begin{equation}\label{hipinducao}
    \mathcal{S}(x)+U_j(x)=X,\ \forall x\in\gamma^+(\mathcal{O}_j).
\end{equation}

In the finite dimensional case, property \eqref{hipinducao} is justified because the derivative $D_x\mathcal{T}(t)$ is surjective for all $t\ge0$ and $x\in\mathcal{A}$. In fact, once we have proved items (2) and (4), it follows from Lemma \ref{lemma-2.2.8} and Remark \ref{remarklsc} that $\mathcal{S}(x)+U_j(x)=X$ close to $x_j^*$. If the derivative is surjective, we can carry this transversality through the positive orbit of $\mathcal{O}_j$. In our case, we have to show the  induction hypothesis \eqref{hipinducao} using another methods. To maintain the organization of this proof we will assume that \eqref{hipinducao}  holds and justify  later (Lemmas \ref{feliz} and \ref{felizz}) why we can assume this condition without loss of generality.

Note that for $j=p$ the neighborhood $\mathcal{O}_p:=W^u(x^*_p)$ of $x_p^*$ and the continuous subbundle 
\begin{equation*}
    U_p(x)=\mathcal{U}(x),\ \forall x\in \gamma(\mathcal{O}_p)=\mathcal{O}_p
\end{equation*}
satisfy items (1)-(5),  where item (2) follows from Remark \ref{bebosim}. Moreover, if  follows from Lemma \ref{extensaos} that condition  \eqref{hipinducao} holds. This proves our induction basis.  Therefore we can assume that our induction hypothesis holds for $j\in\{i+1,\dots,p\}$. We will assume, without loss of generality, that the neighborhood $\mathcal{O}_j$ satisfies
\begin{equation}\label{ajudaeq}
    \mathcal{O}_j\subset V^{-1}\left(j-\dfrac{1}{4},j+\dfrac{1}{4}\right).
\end{equation}

We can also assume that $\mathcal{O}_j$ is a closed (compact) neighborhood of $x_j^*$. Fix the neighborhood $V_i$ of $x_i^*$ in $\mathcal{A}$ such that 
$$V_i\subset V^{-1}\left(i-\dfrac{1}{4},i+\dfrac{1}{4}\right)$$
and we have  continuous subbundles $S_i',U_i'$ on $V_i$ satisfying  \eqref{remarkmudancacoordenadas} and \eqref{mudancacoordenadas}.  Take $d'>0$ as in Lemma \ref{lemma2.2.6} such that
\begin{equation}
    W^s(x^*_i)\cap V^{-1}(i+d)\cap\mathcal{A}\subset V_i, \forall d\in[0,d'].
\end{equation}

Fix $d\in(0,d']$ and define once more the set $D= W^s(x^*_i)\cap V^{-1}(i+d)\cap\mathcal{A}$. Note that
\begin{equation*}
    D\subset\bigcup_{j=i+1}^{p}W^u(x_j^*)\subset \bigcup_{j=i+1}^{p} \gamma^+(\text{int}\mathcal{O}_j)
\end{equation*}
and since $D$ is compact there exists $q>0$ such that
\begin{equation*}
    D\subset \bigcup_{j=i+1}^{p} \gamma_{[0,q]}(\text{int}\mathcal{O}_j).
\end{equation*}

For each $x\in D$ define $j=j(x)$ such that 
\begin{equation*}
    x\in \gamma_{[0,q]}(\mathcal{O}_j) \ \ \text{and}\ \  x\notin\gamma_{[0,q]}(\mathcal{O}_k),\ \forall i<k<j.
\end{equation*}
This means that the neighborhood $\mathcal{O}_j$ is the last one that $x$ visited (until $q$). Since the neighborhoods $\mathcal{O}_k$ are compact, there exists a neighborhood $Z(x)\subset V_i$ of $x$ in $\mathcal{A}\cap V^{-1}(i+d)$ such that
\begin{equation*}
    Z(x)\subset \gamma_{[0,q]}\mathcal{O}_j\ \ \text{and}\ \ Z(x)\cap \gamma_{[0,q]}\mathcal{O}_k=\emptyset,\ \forall i<k<j.
\end{equation*}
\begin{itemize}
\item[\underline{Claim:}] $Z(x)\cap\gamma^+(\mathcal{O}_k)=\emptyset, \ \forall i<k<j$.

Indeed, suppose that there exists $y\in Z(x)\subset\gamma_{[0,q]}\mathcal{O}_j$ such that
$$y\in\gamma^+(\mathcal{O}_k)\setminus\gamma_{[0,q]}(\mathcal{O}_k), $$
for some $i<k<j$. Then $y=\mathcal{T}(s)y'$ for some $s>q$ and $y'\in \mathcal{O}_k$ satisfying
\begin{equation}
   V(y')<k+\dfrac{1}{4}<j-\dfrac{1}{4}. 
\end{equation}
The last inequality follows from (\ref{ajudaeq}). From the decreasing property of the Lyapunov function we have
\begin{equation*}
    V(\mathcal{T}(t)y')<j-\dfrac{1}{4},\ \forall \ 0\le t\le s.
\end{equation*}
Hence $y\notin\gamma_{[0,q]}(\mathcal{O}_j)$, which is a contradiction. This proves the claim.
\end{itemize}

From the continuity of $U_j$ in $\gamma^{+} (\mathcal{O}_j)$, the fact that $S_i'(x)=\mathcal{S}(x)$ and from  (\ref{hipinducao}) there exists continuous subbundles $G^{(x)}$ in ${Z(x)}$ such that
\begin{equation}\label{mudancacoordenadas3}
    G^{(x)}(y)\subset U_j(y)\ \ \text{and}\ \ S_i'(y)\oplus G^{(x)}(y)=X,\ \forall y\in Z(x).
\end{equation}

It follows from (\ref{mudancacoordenadas}) and (\ref{mudancacoordenadas3}) that $G^{(x)}(y)$ can be written as the graph of a continuous linear map $g^{(x)}_y:U'_i(y)\to S'_i(y)$ for all $y\in Z(x), x\in D$. Let $Q$ be a closed neighborhood of $D$ in $V^{-1}(i+d)\cap\mathcal{A}$ and $x_1,\dots,x_n\in D$ such that $Q\subset\bigcup\limits_{k=1}^nZ(x_k)$ . 
Take a smooth partition of unity $\{\beta_k:Q\to\mathbb{R}\}_{k=1}^n$ subordinated to the covering $\{Z(x_k)\}_{k=1}^n$. Define  the continuous  map $g^Q:Q\to\mathcal{L}(X)$ as
\begin{equation}
g^Q(y)=\sum_{k=1}^{n}\beta_k(y)g^{(x_k)}_yP_{U'_i(y)},\ \forall y\in Q.
\end{equation}

 Denoting by $G^Q(y)$ the graph of $g^Q(y)|_{U'_i(y)}$ for all $y\in Q$, we obtain that $G^Q:=\{G^Q(y)\}_{y\in Q}$ is a continuous subbundle in $Q$.
 We want to show that
\begin{equation}\label{equacaoqualquer}
    G^Q(y)\subset U_{k'}(y), \ \forall y\in Q\cap\gamma^+(\mathcal{O}_k'),\ \forall i<k'\le p.
\end{equation}
 Given $v\in U'_i(y)$ we have
 \begin{equation*}
v+g^Q(y)v=\sum_{k=1}^n\beta_k(y)(v+g^{(x_k)}_yv)\in  U_j(y).
 \end{equation*}
We know from our claim that $j\le k'$ and therefore it follows from (\ref{mudancacoordenadas3}) and our induction hypothesis that
\begin{equation*}
   v+g^Q(y)v\in U_j(y)\subset U_{k'}(y),
\end{equation*}
which proves (\ref{equacaoqualquer}).

Now we extend our subbundle $G^Q$ in $\gamma(Q)\cup W^u(x^*_i)$ as
\begin{align*}
    G^Q(\mathcal{T}(t)x)&:=(D_x\mathcal{T}(t))G^Q(x),\ \forall x\in Q,\ \forall t\in\mathbb{R},\\
    G^Q(x)&:=\mathcal{U}(x),\ \forall x\in W^u(x^*_i).
\end{align*}

Note that $G^Q$ already satisfies items (\ref{item1}), (\ref{item2}) and (\ref{item3}). If we reduce $Q$ we obtain the continuity of the subbundle by Lemma \ref{lemma-2.2.7}. From Lemma \ref{lemma2.2.5} the set $\gamma^+(Q)\cup W^u(x^*_i)$ contains a neighborhood $\mathcal{O}_i$ of $x^*_i$ that, together with  the subbundle $U_i(x):=G_i(x)$ for all $x\in \gamma(\mathcal{O}_i)\cup\overline{\mathcal{O}_i}$, satisfy every condition from our thesis.  This finishes the proof for the families $U_i$.

Now we proceed to construct the family $S_i$. Since the proof will follow the same reasoning as before, we will provide only the essential parts. Once more we will prove by induction. Note that the neighborhood $\mathcal{O}_1=W^s(x_1^*)$ and the family $S_1(x)=X$ for all $x\in W^s(x_1^*)$ satisfy the properties required. Lets construct $S_i$ and $\mathcal{O}_i$ assuming that  $S_j$ and $\mathcal{O}_j$ are constructed for $j\in\{1,\dots,i-1\}$, where $S_j$ is continuous in $\overline{\mathcal{O}_j}$. We can proceed as before until the change of coordinates \eqref{mudancacoordenadas3}. Indeed, the  continuous subbundles $G^{(x)}$ in \eqref{mudancacoordenadas3} exists  because the subbundle $U_{j(x)}$ is continuous in $\gamma^+(\mathcal{O}_{j(x)})$. The continuity of $S_{j(x)}$  in $\gamma^-(\mathcal{O}_{j(x)})$ is not straightforward and we have to justify it.

Arguing as in the proof of Proposition \ref{extensaos}, we recall that  for each $x\in\mathcal{A}$ and $t>0$ we have $$\mathcal{U}(x)=(D_{\mathcal{T}(-t)x}\mathcal{T}(t))\mathcal{U}(\mathcal{T}(-t)x),$$ that is,
\begin{equation}\label{maoehhi}
    \mathcal{U}(x)\subset R(D_{\mathcal{T}(-t)x}\mathcal{T}(t)),\ \forall t\ge0,\ \forall x\in\mathcal{A}.
\end{equation}

Let us remember that in the construction of $S_i$ we have
\begin{equation*}
 D=V^{-1}(i-d)\cap W^u(x^*_i)
\end{equation*}
and $U_i'(x)=\mathcal{U}(x)$ for each $x\in V_i\cap W^u(x^*_i)$. 
Fix $x\in D$ and $t_0>0$ such that $\mathcal{T}(t_0)x\in \mathcal{O}_{j(x)}$. Since the family $S_{j(x)}$ is continuous in $\mathcal{O}_{j(x)}$ and $\mathcal{U}$ is lower semicontinuous, we can use \eqref{hipinducao} to fix a neighborhood $Z(\mathcal{T}(t_0)x)\subset \mathcal{O}_{j(x)}$ of $\mathcal{T}(t_0)x$ such that
\begin{equation*}
S_{j(x)}(y)+\mathcal{U}(y)=X,\ \forall y\in Z(\mathcal{T}(t_0)x).
\end{equation*}

Note that
\begin{equation*}
       \dim (S_{j(x)}(y))^{\perp}=m<+\infty,\ \ \forall y\in Z(\mathcal{T}(t_0)x).
   \end{equation*}
   
Hence, from \eqref{maoehhi} we have
\begin{equation*}
    S_{j(x)}(y)+R(D_{\mathcal{T}(-t_0)y}\mathcal{T}(t_0))=X,\ \ \forall y\in Z(\mathcal{T}(t_0)x)
\end{equation*}

 Therefore, we can use Proposition \ref{propinterseccont} to guarantee  that the subbundle $$\{L(y):=S_{j(x)}(y)\cap R(D_{\mathcal{T}(-t_0)y}\mathcal{T}(t_0))\}_{y\in Z(\mathcal{T}(t_0)x)}$$ is continuous.
Now we can use Lemma \ref{contiparatras} to define a continuous subbundle $G'_x$ in $Z(x):=\mathcal{T}(-t_0)Z(\mathcal{T}(t_0)x)$ as 
\begin{equation*}
    G'_x(\mathcal{T}(-t_0)y):=(D_{y}\mathcal{T}(-t_0))L(y)=(D_{\mathcal{T}(-t_0)y}\mathcal{T}(t_0))^{-1}L(y),\ \forall y\in Z(\mathcal{T}(t_0)x).
\end{equation*}

Since $\mathcal{U}(y)\subset R(D_{\mathcal{T}(-t_0)y}\mathcal{T}(t_0))$ we have
\begin{equation*}
    L(y)+\mathcal{U}(y)=(S_j(y)+\mathcal{U}(y))\cap R(D_{\mathcal{T}(-t_0)y}\mathcal{T}(t_0)) =R(D_{\mathcal{T}(-t_0)y}\mathcal{T}(t_0)),\ \forall y\in Z(\mathcal{T}(t_0)x)
\end{equation*}
and from Lemma \ref{iminv}
\begin{equation}
G'_x(y)+\mathcal{U}(y)=X,\ \forall y\in Z(x).
\end{equation}

Therefore, there exists a continuous subbundle $\{G^{(x)}(y)\}_{y\in Z(x)}$ in $Z(x)$ (possibly reducing this neighborhood) satisfying  
\begin{equation*}
    G^{(x)}(y)\subset S_{j(x)}(y)\ \ \text{and}\ \ G^{(x)}(y)\oplus U'_i(y)=X,\ \ \forall y\in Z(x).
\end{equation*}

Now we can proceed as before to construct the continuous subbundle $G^Q$ on $Q$ and extend this subbundle to $\gamma^-(Q)$ as usual
\begin{equation*}
    G^Q(\mathcal{T}(-t)x)
:=(D_{x}\mathcal{T}(-t))G^Q(x),\ \forall x\in Q.
\end{equation*}

We know that $G^Q$ is continuous in $Q$ and $G^Q(y)+ \mathcal{U}(y)=X$ for all $y\in Q$ (remember that $\mathcal{U}$ is lower semicontinuous). Hence $G^Q(y)+R(D_{\mathcal{T}(-t)y}\mathcal{T}(t))=X$ for $y\in Q$ and we can proceed as before to conclude that $G^Q$ is continuous in  $\gamma^-(Q)$. Now we can proceed as in the construction of the subbundles $U_i$ to conclude the proof.
\end{proof}

We remember that the proof of Theorem \ref{lemma2.2.9} is not finished yet. In fact, it remains to justify condition \eqref{hipinducao}. The following lemmas are devoted to show that we can assume the induction hypothesis \eqref{hipinducao} without lost of generality, even without the surjectivity of the derivative $D_x\mathcal{T}(t)$ for $x\in\mathcal{A}$ and $t\ge0$.

\begin{lemma}\label{feliz}
    Let $(M,d)$ be a metric space and $\mathcal{T}=\{\mathcal{T}(t):t\ge0\}$ be a  semigroup in $M$ with finite set of equilibria $\mathcal{E}=\{x_1^*,\dots,x_p^*\}$. Assume that  $\mathcal{T}$ is Gradient, that is, $\mathcal{T}$ has a Lyapunov function $V$ (see Lemma \ref{lemma-Lyapunov-function}). Given $i,k\in\{1,\dots,p\}$, $k<i$, $d\in (0,\frac12)$ and $\epsilon,\epsilon'>0$ there exists $\delta>0$ such that if $x\in B(x^*_i,\delta)\cap\mathcal{A}$ satisfies 
    \begin{enumerate}
        \item there exists $t>0$ such that $\mathcal{T}(t)x\in V^{-1}(k+d)\cap W^s(x_k^*)$,
        \item $d(\mathcal{T}(t)x,x_j^*)\ge \epsilon',\ \forall t\in\mathbb{R},\ \forall k<j<i$,
    \end{enumerate}
    then, there exists a bounded global solution $\xi$, with $\xi(0)\in W^u(x_i^*)$, such that
    \begin{equation*}
        d(\mathcal{T}(t)x,\xi(0))<\epsilon.
    \end{equation*}
    Moreover, $d(\xi(t),x_j^*)\ge \epsilon'$ for all $t\in\mathbb{R}$ and $k<j<i.$
\end{lemma}

\begin{proof}
We will prove by contradiction. Suppose that there exist $i,k\in\{1,\dots,p\}$, $k<i$, $d\in(0,\frac12)$, $\epsilon,\epsilon'>0$, $\mathcal{A}\ni x_n\to x^*_i$ and $t_n\to+\infty$ such that 
\begin{equation*}
    \mathcal{T}(t_n)x_n\in V^{-1}(k+d)\cap W^s(x_k^*),
\end{equation*}
 \begin{equation}\label{sinuca}
            d(\mathcal{T}(t)x_n,x_j^*)\ge \epsilon',\ \forall t\in\mathbb{R},\ \forall k<j<i,\ \forall n\in\mathbb{N}
        \end{equation}
        and
        \begin{equation}\label{paodequeijo}
               d(\mathcal{T}(t_n)x_n,W^u(x^*_i))\ge\epsilon.
        \end{equation}
        
        We assume without loss of generality that
        \begin{equation}\label{cafe}
            i-\dfrac12<V(x_n)<i+\dfrac{1}{2}.
        \end{equation}
        
        For each $n$ take $s_n>0$ such that $V(\mathcal{T}(s_n)x_n)=i-\dfrac{1}{2}$. Despite that $s_n\to+\infty$ we will show that $t_n-s_n$ is bounded. Assume that $t_n-s_n\to+\infty$ and define the sequence of global solutions
        \begin{equation*}
            \xi_n(t)=\mathcal{T}\left(t+\dfrac{t_n+s_n}{2}\right)x_n,\ \forall t\in \left[\frac{-t_n-s_n}{2},+\infty\right).
        \end{equation*}

        Note that
        \begin{equation*}
            k+d \le V(\xi_n(t))\le i-\dfrac{1}{2},\ \forall t\in\left[\dfrac{s_n-t_n}{2},\dfrac{t_n-s_n}{2}\right].
        \end{equation*}
        
        Hence, we can use Proposition \ref{subsequenciasolucoes},with $\sigma_k=\frac{t_k+s_k}{2}$, $u_k=x_k$ and $\xi_k(s)=\mathcal{T}(s+\sigma_k)x_k$ for $s\ge-t_k$, to obtain  a bounded global solution $\xi(\cdot)$ such that (up to a subsequence)
        \begin{equation*}
            \xi_n(t)\xrightarrow{n\to+\infty}\xi(t),\ \forall t\in\mathbb{R}
        \end{equation*}
        and $k+d \le V(\xi(t))\le i-\dfrac{1}{2},\ \forall t\in\mathbb{R}$. Since Gradient semigroups  satisfies item (2) from Lemma \ref{lemma-Lyapunov-function} \cite{hale2010asymptotic,aragao2013non,Bortolan-Carvalho-Langa-book}, then $\xi(0)\in W^s(x_j^*)$ for some $k<j<i$. This contradicts (\ref{sinuca}). Thus $t_n-s_n$ is bounded, that is, there exists $C<0$ such that
        \begin{equation}\label{bounded}
            C \le s_n-t_n,\ \forall n\in\mathbb{N}.
        \end{equation}
        
For each $n\in\mathbb{N}$ consider the global solutions
\begin{equation*}
    \xi_n(t)=\mathcal{T}(t+t_n)x_n,\ \forall t\in\mathbb{R}.
\end{equation*}

From Proposition \ref{subsequenciasolucoes}) there exists a bounded global solution  $\xi(\cdot)$  such that
\begin{equation}\label{mocaccino}
    \xi_n(t)\xrightarrow{n\to+\infty}\xi(t), \forall t\in\mathbb{R}.
\end{equation}

It follows from (\ref{cafe}) that
\begin{equation}\label{capuccino}
    i-\dfrac{1}{2}\le V(\xi_n(t))\le i+\dfrac{1}{2},\ \forall t\le [-t_n,s_n-t_n],\ \forall n\in\mathbb{N}.
\end{equation}

Putting (\ref{bounded}) and (\ref{capuccino}) together and assuming without lost of generality that $-t_n<C$ for all $n$ we have 
\begin{equation}\label{pingado}
    i-\dfrac{1}{2}\le V(\xi_n(t))\le i+\dfrac{1}{2},\ \forall t\in [-t_n,C], \forall n.
\end{equation}

From (\ref{mocaccino}) and (\ref{pingado}) we conclude that
\begin{equation*}
    i-\dfrac{1}{2}\le V(\xi(t))\le i+\dfrac{1}{2},\ \forall t\in (-\infty,C).
\end{equation*}

Therefore 
$$\xi_n(0)=\mathcal{T}(t_n)x_n\xrightarrow{n\to+\infty} \xi(0)\in W^u(x^*_i),$$
which contradicts (\ref{paodequeijo}). This concludes the proof.
\end{proof}

\begin{lemma}\label{felizz}
In the conditions and notations of Theorem \ref{lemma2.2.9}, assume that the neighborhoods $\mathcal{O}_{i+1},\dots,\mathcal{O}_p$ and the subbundles $\{U_{i+1},\dots,U_p\}$ are constructed.  Then, we can reduce the neighborhoods $\mathcal{O}_j$, $j\in\{i+1,\dots,p\}$, such that if $x\in D=V^{-1}(i+d)\cap W^s(x_i^*)\cap\mathcal{A}$ and $x\in\gamma^+(\mathcal{O}_j)$, where $j=j(x)$ is defined by
\begin{equation*}
    j(x):=\min\{i< k\le p: x\in \gamma^+(\mathcal{O}_k)\},
\end{equation*}
then $\mathcal{S}(x)+U_j(x)=X$.
\end{lemma}
\begin{proof}
Consider $d\in(0,1)$ and $ D=V^{-1}(i+d)\cap W^s(x^*_i)\cap\mathcal{A}$ as in Lemma \ref{lemma2.2.9}. Since $D$ is compact, we can fix $\epsilon_1>0$ such that
\begin{equation}\label{naosei}
   B_{\mathcal{A}}(D,\epsilon_1):=\{y\in\mathcal{A}:d(y,D)<\epsilon_1\}\subset V^{-1}(i+\frac{d}{2},+\infty). 
\end{equation}

Given $\epsilon'>0$ and $i<j\le p$ define the subset
$$\mathcal{H}_j:=\{x\in W^u(x_j^*): d(\gamma(x),x_k^*)\ge\epsilon', i<k<j\}\subset\mathcal{A}.$$

We claim that the subset
$$\mathcal{J}_j:=\mathcal{H}_j\cap V^{-1}\left[i+\dfrac{d}{2},j-\dfrac{d}{2}\right] $$
is compact. Indeed, let $\mathcal{J}_j\ni x_n\to x\in\mathcal{A}$. Take bounded global solutions $\xi_n$ such that $\xi_n(0)=x_n$. Define the global solutions $\psi_k(\cdot)=\xi_k(\cdot-k)$ for all $k\in\mathbb{N}$. It follows from Proposition \ref{subsequenciasolucoes}, with $\sigma_k=k$, $\xi_k(s)=\psi_k(s+k)$ for all $s\in\mathbb{R}$, that there exists a subsequence of $\xi_n$ and a bounded global solution $\xi$ such that
\begin{equation*}
    \xi_n(t)\xrightarrow{n\to+\infty}\xi(t),\ \forall t\in\mathbb{R}.
\end{equation*}

Note that $V(\xi(t))\le j$ for all $t\in\mathbb{R}$ and $V(x)\in\left[i+\dfrac{d}{2},j-\dfrac{d}{2}\right]$, that is,
\begin{equation*}
    \xi(t)\to x_k^*,\ \text{as}\ t\to-\infty,\ \text{for some}\ i<k\le j.
\end{equation*}

Since 
\begin{equation*}
    d(\xi_n(t),x_k^*)\ge\epsilon',\ \forall i<k<j,\ \forall n\in\mathbb{N},\ \forall t\in\mathbb{R}, 
\end{equation*}
we conclude that $\xi(t)\to x_j^*$ as $t\to-\infty$, and $d(\xi(t),x_k^*)\ge\epsilon'$, that is, $x\in\mathcal{H}_j$. Obviously $x\in V^{-1}\left[i+\dfrac{d}{2},j-\dfrac{d}{2}\right]$. This concludes our claim that $\mathcal{J}_j$ is compact.

Note that $U_j(x)=\mathcal{U}(x)$ for $x\in\mathcal{J}_j$. Hence, it follows from Lemma \ref{extensaos} that 
$$\mathcal{S}(x)+U_j(x)=X,\ \forall x\in\mathcal{J}_j.$$

From the compactness of $\mathcal{J}_j$, the continuity of $U_j$ in $\gamma^+(\mathcal{O}_j)$ and the lower semicontinuity of $\mathcal{S}$, we can Lemma \ref{lemma-2.2.8}  to obtain an $\epsilon_2>0$ such that 
\begin{equation}\label{contido0}
   y\in B_{\mathcal{A}}(\mathcal{J}_j,\epsilon_2):=\{y\in\mathcal{A}:d(y,\mathcal{J}_j)<\epsilon_2\}\Rightarrow \mathcal{S}(y)+U_j(y)=X,\ \forall j\in\{i+1,\dots,p\}.
\end{equation}

Now we are ready to prove our statement by induction. Define $\epsilon:=\min\{\epsilon_1,\epsilon_2\}$. We start our proof for the neighborhood $\mathcal{O}_{i+1}$.   From Lemma \ref{feliz} we can reduce $\mathcal{O}_{i+1}$ such that
\begin{equation*}
    x\in \mathcal{O}_{i+1}\ \text{and}\ \mathcal{T}(t)x\in W^s(x^*_i)\cap V^{-1}(i+d)\Rightarrow d(\mathcal{T}(t)x,W^u(x_{i+1}^*))<\dfrac{\epsilon}{2}.
\end{equation*}

    Hence, if $x\in D$ and $j(x)=i+1$, there exists $\tilde{x}\in W^u(x_{i+1}^*)$ with $d(x,\tilde{x})<\epsilon$. From  (\ref{naosei}) we know that $V(\tilde{x})>i+\dfrac{d}{2}$. Therefore $\tilde{x}\in \mathcal{J}_{i+1}$ and $x\in B(\mathcal{J}_j,\epsilon_2)$. It follows from (\ref{contido0}) that $\mathcal{S}(x)+U_{i+1}(x)=X$. This finishes the proof for $j=i+1$.

    Assume that for $i+1\le r<j$ it is valid that if $x\in\gamma^+(\mathcal{O}_r)\cap D$ and $x\notin\gamma^+(\mathcal{O}_k)$, $i<k<r$, then
    $$\mathcal{S}(x)+U_r(x)=X.$$

    Fix $\epsilon'>0$ such that 
    $$B_{\mathcal{A}}(x_k^*,\epsilon')\subset \mathcal{O}_k, i<k<j.$$
   With that $\epsilon'$, consider the set $\mathcal{J}_j$.
From Lemma \ref{feliz} we can reduce $\mathcal{O}_j$ such that if $x\in \mathcal{O}_j$, $$d(\mathcal{T}(t)x,x_k^*)\ge \epsilon',\ \forall t\in\mathbb{R}, i<k<j$$ and $\mathcal{T}(t)x\in D$, then $$d(\mathcal{T}(t)x,W^u(x_j^*))<\dfrac{\epsilon}{2}.$$

Thus, if $x\in D$ and $j(x)=j$, there exists $\tilde{x}\in W^u(x_j^*)$ such that $d(x,\tilde{x})<\epsilon$. Therefore $x\in B_{\mathcal{A}}(\mathcal{J}_j,\epsilon_2)$ and from \eqref{contido0} we have $\mathcal{S}(x)+U_j(x)=X$. This concludes the proof.
\end{proof}

With  Lemmas \ref{feliz} and \ref{felizz} we have finally finished the proof of Theorem \ref{lemma2.2.9}. To finish this subsection, we will improve the property stated in item \ref{item5} of Theorem \ref{lemma2.2.9}. From item \eqref{item5} of Theorem \ref{lemma2.2.9} we have
\begin{equation*}
    \mathcal{T}(s)x\in\mathcal{O}_i,\ \forall s\in[0,n]\Longrightarrow \parallel (D_x\mathcal{T}(n))v^s\parallel\le \tilde{C}^n\lambda_1^n,\ \forall v^s\in S(x),\ \forall x\in\mathcal{O}_i, \forall n\in\mathbb{N}. 
\end{equation*}

Note that $\tilde{C}^n\lambda_1^n$ does not (necessarily) converges to $0$ as $n\to+\infty$.  This convergence will be important to apply the Banach fixed point theorem and finish the proof of Lipschitz Shadowing in $\mathcal{A}$. In the finite dimensional setting (where the phase space is a finite dimensional compact manifold with no border or a finite dimensional vector space) this convergence is straightforward, since $\mathcal{T}(t)$ is a diffeomorphism and we can use the Lyapunov norm to assume  $\tilde{C}=1$ without lost of generality when (see \cite[Lemma 1.2.1]{Pilyugin}).

Hence, we have to manipulate our semigroup to obtain this convergence. To do that, we will work with the semigroups $\mathcal{T}_N:=\{\mathcal{T}(tN):t\ge0\}$, where $N\in\mathbb{N}$. In order to do that, we need to announce two technical lemmas that will allow us to prove Lipschitz Shadowing for the operator $\mathcal{T}(N)$, where $N\in\mathbb{N}$, instead of the operator $\mathcal{T}(1)$. We will not provide the proof of this results, that can be found in \cite[Lemmas 3.4.1 and 3.4.2]{Pilyugin}.

\begin{lemma}\label{lemaimportante}
Let $(M,d)$ be a metric space and $\mathcal{T}:M\to M$ be a  Lipschitz map with Lipschitz constant $L_1$. Assume that $\{x_n\}_{n\in\mathbb{Z}}\subset M$ is a  $d$-pseudo-orbit of $\mathcal{T}$ for some $d>0$. 
 Then, for any $N \in\mathbb{N}$ we have
 \begin{equation}\label{teste}
     d(\mathcal{T}^N(x_{Nk}),x_{(N+1)k})\le C_1d, \forall k\in\mathbb{Z},
 \end{equation}
 where $C_1=(1+L_1+\dots+L_1^{N-1})$.
 
 In particular, if we define  $z_n:=x_{nN}$ for all $n\in\mathbb{Z}$, then the sequence $\{z_n\}_{n\in\mathbb{Z}}$ is a 
$C_1d$-pseudo-orbit of the map $\phi:=\mathcal{T}^N$.
\end{lemma}

\begin{lemma}\label{auxiliar1}
Let $(M,d)$ be a metric space and $\mathcal{T}:M\to M$ be a  Lipschitz map with Lipschitz constant $L_1$. Assume that $\{x_n\}_{n\in\mathbb{Z}}$ is a $d$-pseudo-orbit of $\mathcal{T}$ for some $d>0$. Given $N\in\mathbb{N}$  define the map $\phi:=\mathcal{T}^N$ and consider the sequence $z_n:=x_{nN}$, for all $n\in\mathbb{Z}$.  Assume that there exists $L>1$ and $x\in M$ such that 
\begin{equation*}
    d(\phi^kx,z_k)\le Ld,\ \forall k\in\mathbb{Z}.
\end{equation*}

Then
\begin{equation*}
   d(\mathcal{T}^kx,x_k)\le L^*d,\ \forall k\in\mathbb{Z},
\end{equation*}
where $L^*=(1+L_1+\dots+L_1^N)L$ is independent of $d$. 
\end{lemma}

To finish this subsection, we now improve item \ref{item5} of Theorem \ref{lemma2.2.9}. 

\begin{remark}\label{precisadisso}
    From item \eqref{item5} of Theorem \ref{lemma2.2.9} we can assume without lost of generality   that there exists $\tilde{C}>0$ and $\lambda_0\in(0,1)$ such that if $\mathcal{T}(s)x\in \mathcal{O}_i$ for all $s\in[0,t]$ then
    \begin{equation}\label{1h30damanha}
        \parallel (D_x\mathcal{T}(t))v^s\parallel\le \tilde{C}\lambda_0^t\parallel v^s\parallel, \ \forall v^s\in S(x).
    \end{equation}
    
    Similarly, if $\mathcal{T}(s)x\in \mathcal{O}_i$ for all $s\in[-t,0]$ then
     \begin{equation}\label{ainda1h30}
        \parallel (D_x\mathcal{T}(-t))v^u\parallel\le \tilde{C}\lambda_0^t\parallel v^u\parallel, \ \forall v^u\in U(x).
    \end{equation}
    
    Indeed, fix $N\in\mathbb{N}$ such that $\lambda_0:=\tilde{C}\lambda_1^N<1$ and reduce the neighborhood $\mathcal{O}_i$ such that  inequalities \eqref{hipfinita} and \eqref{hipinfinita2} hold for $t\in[0,N]$.  Now we will work with the semigroup $\mathcal{T}_N:=\{\mathcal{T}(Nt):t\ge0\}$. Suppose that $x\in\mathcal{O}_i$ and $t>0$ satisfy $\mathcal{T}_N(s)x\in \mathcal{O}_i$ for all $s\in[0,t]$. Take $n\in\mathbb{N}$ and $r\in [0,1)$ such that $t=n+r$. Assuming without lost of generality that $\tilde{C}\ge 1$,  we have
    \begin{align*}
        \parallel (D_x\mathcal{T}_N(t))v^s\parallel&=\parallel (D_x\mathcal{T}_N(n+r))v^s\parallel=\parallel (D_{\mathcal{T}_N(n)x}\mathcal{T}_N(r))\circ (D_x\mathcal{T}_N(n)) v^s\parallel\\
        &\le \tilde{C}\lambda_1^{Nr}\parallel  (D_x\mathcal{T}_N(n)) v^s\parallel\le \tilde{C}\lambda_1^{Nr}\lambda_0^n\parallel v^s\parallel\\
        &=\tilde{C}^{1-r}\lambda_0^{n+r}\parallel v^s\parallel\le\tilde{C}\lambda_0^{t}\parallel v^s\parallel,\ \forall v^s\in S(x).
    \end{align*}
We can prove \eqref{ainda1h30} analogously.
    Hence, properties \ref{1h30damanha} and \ref{ainda1h30} are valid for the semigroup $\mathcal{T}_N$. Since the property of Lipschitz Shadowing holds for $\mathcal{T}_N(1)=\mathcal{T}(1)^N$ if, and only if holds for $\mathcal{T}(1)$ (Lemmas \ref{lemaimportante} and \ref{auxiliar1}), we can work with the semigroup $\mathcal{T}_N$ instead of $\mathcal{T}$. This finishes the proof.
\end{remark}

\subsection{The subbundles $S$ and $U$}\label{subsec33}

In this subsection we  conclude the proof of the Lipschitz Shadowing property for the map $\mathcal{T}(1)|_{\mathcal{A}}:\mathcal{A}\to\mathcal{A}$. We will use the compatible subbundles $S_i$ and $U_i$, constructed in Theorem \ref{lemma2.2.9}, to obtain a new  subbundles $S$ and $U$ (independent of $i$) on ${\mathcal{A}}$. The subbundles $S$ and $U$ are the key to connect the geometrical construction of the compatible subbundles with the fixed point theorem. In fact, after we prove some properties for the subbundles $S$ and $U$ we will be in  condition  to proceed as in the finite dimensional case and apply the Banach Fixed Point Theorem \cite[Theorem 1.3.1]{Pilyugin}. Therefore, it remains to construct and prove some properties of the subbundles $S$ and $U$ to finish the proof of Lipschitz Shadowing in $\mathcal{A}$. 

 Fix neighborhoods $\mathcal{O}_i\subset \mathcal{A}$ of $x^*_i$ in  and continuous subbundles $S_i,U_i$ on $\mathcal{O}_i$ with the properties stated in Theorem \ref{lemma2.2.9}. Define $W=\bigcup\limits_{i=i}^p \mathcal{O}_i$ and take  a Birkhoff number  $T>0$  related to $W$  (see Lemma \ref{birkhoff}). Hence for each $x\in\mathcal{A}$ there exists $t\in [0,T]$ such that $\mathcal{T}(t)x\in W$. Given $x\in\mathcal{A}$ take $t=t(x)\in[0,T]$ and $i\in\{1,\dots,p\}$ such that 
\begin{equation}\label{condition}
    \mathcal{T}(t)x\in \mathcal{O}_i\ \ \text{and}\ \ \{\mathcal{T}(s)x:s\in[0,t]\}\cap\mathcal{O}_j=\emptyset, \forall  j\in\{1,\dots,p\}, j\neq i.
\end{equation}

Define $S(x)=S_i(x)$ and $U(x)=U_i(x)$. Note that $S(x)\oplus U(x)=X$ for all $x\in\mathcal{A}$. Therefore we can define, using the notation introduced in Remark \ref{defprosepa}, the projections 
\begin{equation*}
P_s(x):=P_{S(x)U(x)}\ \text{and}\     P_u(x):=I-P_s(x)=P_{U(x)S(x)},\ \forall x\in\mathcal{A}.
\end{equation*}

In order to apply the Fixed Point Theorem, we have to  bound the norms of these projections in the attractor, that is, we must show that
\begin{equation}\label{boundofPs}
    \sup\limits_{x\in\mathcal{A}}(\parallel P_s(x)\parallel_{\mathcal{L}(X)}+\parallel P_u(x)\parallel_{\mathcal{L}(X)})<+\infty.
\end{equation}

In the proof of Lipschitz Shadowing in finite dimension compact manifolds, the boundness \eqref{boundofPs} follows from the continuity of the subbundles $S_i$ and $U_i$ in $\gamma^-(\mathcal{O}_i)$. Since this is not our case (the continuity only holds for the subbundles $U_i$), we have to show  \eqref{boundofPs} without using continuity. In fact, to prove \eqref{boundofPs} we will show that the the subspaces $S_i$ and $U_i$ are not ``leaning'' into each other. Therefore, the following results will use the concept of inclination (see Lemma \ref{mudandolemma227}).

 To give a geometric intuition about our next technical results, let $X=\mathbb{R}^2$, $S=\langle(1,0)\rangle$, $U=\langle(0,1)\rangle$ and $\mathcal{V}(n)=\langle(1,n)\rangle$ for all $n\in\mathbb{N}^*$, where $\langle v\rangle$ denotes the subspace generated by $v\in X$. It is clear that $S\oplus U=X$, $\mathcal{V}(n)\cap S=\emptyset$ and $\mathcal{V}(n)\cap U=\emptyset$ for all $n\in\mathbb{N}^*$. Therefore, we can compute the inclination of $\mathcal{V}(n)$ in relation to both decompositions $S\oplus U=U\oplus S=X$. Geometrically, note that $\mathcal{V}(n)$  ``approaches'' the subspace $U$ (Figure \ref{inclinacaoinfinito}). Given $v=(c,nc)\in\mathcal{V}(n)$ non null ($c\neq0$) we have $$v=v^s+v^u=(c,0)+(0,nc),$$ where $(c,0)\in S$ and $(0,nc)\in U$. Hence,
\begin{equation*}
    \dfrac{|v^s|}{|v^u|}=\dfrac{1}{n}\ \ \text{and}\ \ \dfrac{|v^u|}{|v^s|}=n,
\end{equation*}
which implies that the inclination of $\mathcal{V}(n)$ in relation to $S\oplus U=X$ goes to $0$ (the subspaces $\mathcal{V}(n)$ are leaning into $U$) and the inclination of $\mathcal{V}(n)$ in relation to $U\oplus S=X$ goes to infinity.

\begin{figure}[H]
    \centering
    \begin{tikzpicture}
    \draw[thick,-](0,-3)--(0,4)node[above]{$U$};
    \draw[thick,-](-3,0)--(4,0)node[right]{$S$};
    \draw[blue,thick,-](-2,-2)--(3,3)node[right]{$\mathcal{V}(1)$};
    \draw[blue,thick,-](-1.2,-2.4)--(1.9,3.8)node[right]{$\mathcal{V}(2)$};
        \draw[blue,thick,-](-0.85,-2.55)--(1.35,4.05);
        \draw[blue,thick,-](-0.65,-2.6)--(1.03,4.12);
        \draw[blue,thick,-](-0.26,-2.6)--(0.42,4.2);
        \node[blue] at (1.1,4.2)[above] {$\mathcal{V}(n)$};
\end{tikzpicture}
    \caption{Inclination of $\mathcal{V}(n)$}
    \label{inclinacaoinfinito}
\end{figure}
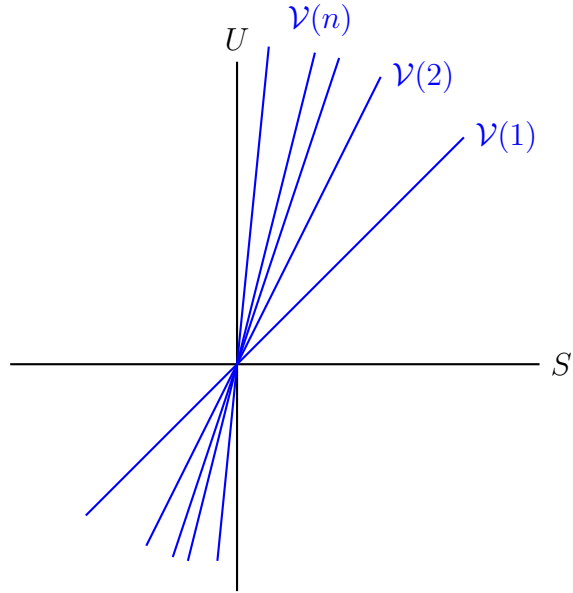
As we said before, the inclination of $\mathcal{V}(n)$ goes to infinity in relation to the decomposition $U\oplus S=X$ because $\mathcal{V}(n)$ was leaning into $U$. Note that if we have defined a continuous family $\mathcal{V}(t)=\langle (1,t)\rangle$ and let the parameter $t$ varies from $(0,+\infty)$ the same would happen. On the other hand, if the parameter $t$ varies in a compact subset $K\subset (0,+\infty)$, then the continuous subbundle $\{\mathcal{V}(t):t\in K\}$ does not approach $U$ and its inclination would not explode. This is the idea behind the next lemma.

\begin{lemma}\label{inclinacaoltda}
Let $(M,d)$ be a compact metric space, $X$ be a Hilbert space and $\{P(x)\subset X:x\in M\}$, $\{Q(x)\subset X:x\in M\}$, $\{Z(x)\subset X:x\in M\}$ be continuous subbundles such that
\begin{equation*}
    S(x)\oplus U(x)=X,\ \text{for all}\ x\in M,
\end{equation*}
and
\begin{equation*}
    S(x)\cap \mathcal{V}(x)=\{0\},\ \text{for all}\ x\in M,
\end{equation*}
where $S(x)=R(P(x))$, $U(x)=R(Q(x))$ and $\mathcal{V}(x)=R(Z(x))$. Given $v\in\mathcal{V}(x)$ there exists unique $v^s\in S(x)$ and $v^u\in U(x)$ such that
\begin{equation*}
    v=v^s+v^u
\end{equation*}
with $v^u\neq 0$. We recall  that the inclination (see Lemma \ref{mudandolemma227}) of a non-null vector $v\in\mathcal{V}(x)$ in relation to the direct sum  $S(x)\oplus U(x)=X$ is given by
\begin{equation*}
    \alpha(v)=\dfrac{\parallel v^s\parallel}{\parallel v^u\parallel}.
\end{equation*}

Now define the inclination of $\mathcal{V}(x)$ in relation to the decomposition $S(x)\oplus U(x)=X$ as
\begin{equation*}
    \alpha(\mathcal{V}(x)):=\sup\limits_{v\in\mathcal{V}(x)\atop v\neq 0}\alpha(v)=\sup\limits_{v\in\mathcal{V}(x)\atop \parallel v\parallel=1}\alpha(v).
\end{equation*}

If  $\mathcal{V}(x)$ is finite dimensional for all $x\in X$ (with same dimension for each $x\in M$), then there exists $M'>0$ such that
\begin{equation*}
   \sup\limits_{x\in M}\alpha(\mathcal{V}(x))\le M'.
\end{equation*}
\end{lemma}
\begin{proof}
Given $x\in M$, define  the continuous projection
\begin{align*}
    P_s(x):X=&S(x)\oplus U(x)\to S(x)\\
    &v=v^s+v^u\mapsto v^s
\end{align*}
Denoting $P_u(x)=I-P_s(x)$, we have 
\begin{equation*}
    \alpha(v)=\dfrac{\parallel P_s(x)v\parallel}{\parallel P_u(x)v\parallel}=\dfrac{\parallel (I-P_u(x))v\parallel}{\parallel P_u(x)v\parallel},\ \forall x\in M,\ \forall v\in \mathcal{V}(x), v\neq 0.
\end{equation*}
Note that for any $x\in M$ and $v\in\mathcal{V}(x)$ non nulle it holds
\begin{equation*}
    \alpha(v)=\dfrac{\parallel (I-P_u(x))v\parallel}{\parallel P_u(x)v\parallel}\le\dfrac{ \parallel v\parallel}{\parallel P_u(x)v\parallel}+1,
\end{equation*}
implying that
\begin{equation}\label{limitoumesmo}
    \alpha(\mathcal{V}(x))\le \sup\limits_{v\in\mathcal{V}(x)\atop \parallel v\parallel=1}\dfrac{1}{\parallel P_u(x)v\parallel}+1=\left(\inf\limits_{v\in\mathcal{V}(x)\atop \parallel v\parallel=1}\parallel P_u(x)v\parallel\right)^{-1}+1.
\end{equation}

It remains to show that there exists $\epsilon>0$ satisfying
\begin{equation*}
    \inf\limits_{v\in\mathcal{V}(x)\atop \parallel v\parallel=1}\parallel P_u(x)v\parallel>\epsilon,\ \forall x\in M.
\end{equation*}

Assume by contradiction that there exists $x_n\in M$ and $v_n\in \mathcal{V}(x_n)$, $\parallel v_n\parallel=1$ such that $P_u(x_n)v_n\to 0$ as $n\to+\infty$. We can assume, since $M$ is compact, that $x_n\xrightarrow{n\to+\infty} x\in M$ and consequently 
\begin{equation*}
    d(v_n,S(x))\to 0.
\end{equation*}

Since $\mathcal{V}(x)$  is finite dimensional we can assume without loss of generality that $v_n\to v$ as $n\to+\infty$, for some $v\in\mathcal{V}(x)$ with $\parallel v\parallel=1$. This implies that $v\in\mathcal{V}(x)\cap S(x)$ , that is, $v=0$. This contradicts $\parallel v\parallel=1$ and concludes the proof.
\end{proof} 

In order to provide a geometric intuition of our next technical result, first note that in Figure \ref{inclinacaoinfinito} we had 
 \begin{equation*}
     S\oplus U=X, U=S^{\perp}\ \text{and}\ U\oplus\mathcal{V}(n)=X\ \text{for all}\ n.
 \end{equation*}
 
Moreover, $\parallel P_{U\mathcal{V}(n)}\parallel_{\mathcal{L}(X)}\to +\infty$ and $\parallel P_{\mathcal{V}(n)U}\parallel_{\mathcal{L}(X)}\to +\infty$ as $n\to+\infty$, where we are using the notation from Remark \ref{defprosepa}. In fact, the norm of the projections grows as the subspaces are leaning into each other and  minimize when the subspaces are orthogonal (orthogonal projections have norm $1$). Therefore, our intuition leads us to the following: if we want to maintain the norm $\parallel P_{U\mathcal{V}(n)}\parallel_{\mathcal{L}(X)}$ of the projections bounded, we must control the inclination of the subspaces. T

\begin{proposition}\label{boundedsetprojections}
Let $X$ be a Hilbert space, $\Gamma$ be a non-empty set and $\{S(x):x\in\Gamma\}$ and  $\{U(x):x\in\Gamma\}$ be two families of closed subspaces of $X$ (not necessarily continuous) with
\begin{equation*}
S(x)\oplus U(x)=X,\ \forall x\in\Gamma.
\end{equation*}
Since $U^{\perp}(x)\oplus U(x) =X$ we can calculate the inclination  of $S(x)$ (see Lemma \ref{inclinacaoltda}) in relation to  $U(x)\oplus U^{\perp}(x)=X$.  Assume that there exists $M>0$ such that
\begin{equation*}
    \sup\limits_{x\in \Gamma}\alpha (S(x))\le M,
\end{equation*}
that is,
\begin{equation}\label{label}
   \sup\limits_{x\in \Gamma} \sup\limits_{v\in S(x)\atop v\neq 0}\dfrac{\parallel P_{UU^{\perp}}(x)v\parallel}{\parallel P_{U^{\perp}U}(x)v\parallel}\le M.
\end{equation}

Then, there exists $M'>0$ such that
\begin{equation*}
    \sup\limits_{x\in\Gamma}\{\parallel P_{SU}(x)\parallel_{\mathcal{L}(X)}+\parallel P_{US}(x)\parallel_{\mathcal{L}(X)}\}\le M',
\end{equation*}
where $P_{SU}(x):=P_{S(x)U(x)}$ and $P_{US}(x):=P_{U(x)S(x)}$ for all $x\in\Gamma$ (see the notation established in Remark \ref{defprosepa}).

\end{proposition}
\begin{proof}
We just need to show
\begin{equation*}
    \sup\limits_{x\in\Gamma}\{\parallel P_{SU}(x)\parallel_{\mathcal{L}(X)}\}<+\infty
\end{equation*}
since $$\parallel P_{US}(x)\parallel_{\mathcal{L}(X)}=\parallel I-P_{SU}(x)\parallel_{\mathcal{L}(X)}\le 1+\parallel P_{SU}(x)\parallel_{\mathcal{L}(X)},\ \forall x\in\Gamma.$$

First, let us prove that
\begin{equation}\label{legaltambem}
     P_{SU}(x)=(I-P_{S(x)} P_{U(x)})^{-1}\circ P_{S(x)}\circ (I-P_{S(x)}P_{U(x)}),\ \forall x\in\Gamma.
\end{equation}

Note that the map $(I-P_{S(x)} P_{U(x)})^{-1}$ is well defined (Lemma \ref{lemalegalzao}). Moreover, we  can easily verify that equality \eqref{legaltambem} holds for vectors in $S(x)$ and in $U(x)$. Since $S(x)\oplus U(x)=X$, then equality \eqref{legaltambem} holds.

Hence, 
\begin{equation*}
    \parallel P_{SU}(x)\parallel_{\mathcal{L}(X)}\le 2\parallel (I-P_{S(x)} P_{U(x)})^{-1}\circ P_{S(x)}\parallel_{\mathcal{L}(X)}, \forall x\in\Gamma.
\end{equation*}

To estimate the term on the right side, we just need to find an $\epsilon>0$ satisfying
\begin{equation*}
    \parallel (I-P_{S(x)} P_{U(x)})v\parallel\ge\epsilon\parallel v\parallel, \forall v\in S(x), \forall x\in\Gamma,
\end{equation*}
that is,
\begin{equation}\label{granfinale}
    \parallel (I-P_{S(x)} P_{U(x)})v\parallel\ge\epsilon,\ \forall v\in S(x), \parallel v\parallel=1, \forall x\in\Gamma.
\end{equation}

 From \eqref{label} we have
\begin{equation*}
    \parallel P_{U^{\perp}U}(x)v\parallel\ge \dfrac{1}{M}\parallel P_{UU^{\perp}}(x)v\parallel\ge \dfrac{1}{M}\parallel v\parallel -\dfrac{1}{M}\parallel P_{U^{\perp}U}(x)v\parallel,\ \forall v\in S(x),\ \forall x\in\Gamma.
\end{equation*}

Thus
\begin{equation}\label{minigranfinale}
    \parallel P_{U^{\perp}U}(x)v\parallel\ge\dfrac{1}{M+1}\parallel v\parallel, \forall v\in S(x),\ \forall x\in\Gamma.
\end{equation}

For $v\in S(x)$, $\parallel v\parallel=1$ we have
\begin{align*}
    \parallel P_{U^{\perp}U}(x)P_{SS^{\perp}}(x)v \parallel^2&=\langle P_{U^{\perp}U}(x)P_{SS^{\perp}}(x)v,P_{U^{\perp}U}(x)P_{SS^{\perp}}(x)v\rangle=\langle P_{SS^{\perp}}(x)v,P_{U^{\perp}U}(x)v\rangle\\
    &=\langle v,P_{SS^{\perp}}(x)P_{U^{\perp}U}(x)v\rangle\le \parallel P_{SS^{\perp}}(x)P_{U^{\perp}U}(x)v\parallel.
\end{align*}

Finally, putting the last inequality together with  \eqref{minigranfinale} we guarantee that for any $v\in S(x)$ with $\parallel v\parallel=1$ we have
\begin{align*}
    \parallel (I-P_{SS^{\perp}}(x) P_{UU^{\perp}}(x))v\parallel&=\parallel P_{SS^{\perp}}(x)(I-P_{UU^{\perp}}(x))v\parallel=\parallel P_{SS^{\perp}}(x)P_{U^{\perp}U}(x))v\parallel\\
    &\ge \parallel P_{U^{\perp}U}(x)P_{SS^{\perp}}(x)v \parallel^2\ge\dfrac{1}{(M+1)^2}.
\end{align*}

This proves \eqref{granfinale} and concludes the proof.
\end{proof}

With Lemma \ref{inclinacaoltda} and Proposition \ref{boundedsetprojections} we are ready to prove inequality \eqref{boundofPs}. Despite that, we will first prove one last technical lemma.

The following  result has a similar proof to the one found in \cite[Lemma 2.2.10]{Pilyugin}.  Despite that, we decided to show the proof since we need to adapt the result to our definition of continuity.
\begin{lemma}\label{lemma2.2.10}
Consider the subbundles $S$ and $U$ defined in the beginning of this subsection and take $x\in\mathcal{A}$. If $y=\mathcal{T}(t)x$, the following statement holds:
\begin{enumerate}
    \item If $x\in\mathcal{A}$, then $(D_x\mathcal{T}(t))S(x)\subset S(y), \ \forall t\ge 0\ \text{and}\  (D_x\mathcal{T}(t))U(x)\subset U(y),\ \forall t\le 0.$
    \item If $y_n\to y$, then there exists $n_0>0$ and isomorphisms $\Pi_n:X\to X$ for $n\ge n_0$, such that $\parallel \Pi_n-I\parallel_{\mathcal{L}(X)}\to 0$  and
    \begin{align*}
        &\Pi_n((D_x\mathcal{T}(t))S(x))\subset S(y_n), \forall t\ge T,\ \forall n\ge n_0,\\
        &\Pi_n((D_x\mathcal{T}(t))U(x))\subset U(y_n), \forall t\le -T,\ \forall n\ge n_0,
    \end{align*}
\end{enumerate}
where $T$ is the Birkhoff number of $W$ given in the construction of $S$ and $U$.
\end{lemma}
\begin{proof}
Item (1) follows straightforward from Theorem \ref{lemma2.2.9}.
We will  prove item (2) for $S$ and the proof for $U$ will follow similarly. Suppose that $y_n\to y$, where $y=\T(t)x$, with $t\ge T$. We know from the construction of $S$ that there exists $t_0\in[0,T]$ and $i\in\{1,\dots,p\}$ such that $x'=\T(t_0)x$ satisfies (\ref{condition}). Define $t'=t-t_0\ge 0$ and note that 
    \begin{equation*}
        x_n':=\T(-t')y_n\to x'.
    \end{equation*}
      
Note that $x'\in \mathcal{O}_i$ and therefore we will assume  that $x_n'\in \mathcal{O}_i$ for all $n$. From the continuity of the subbundle $S_i$ in $\mathcal{O}_i$ we know that
$$S_i(x_n')\to S_i(x'),\ \text{as}\ n\to\infty.$$ 
    
We will assume without loss of generality that $\mathcal{T}(t)x\notin \mathcal{O}_i$ (the other case is more simple). Then,  we have
\begin{align*}
    (D_x\mathcal{T}(t))S(x)&=(D_x\mathcal{T}(t))S_i(x)= S_i(\mathcal{T}(t)x)=S_i(\mathcal{T}(t')\mathcal{T}(t-t')x)\\
    &=S_i(\mathcal{T}(t')x')=(D_{x'}\mathcal{T}(t'))S_i(x')=\lim\limits_{n\to\infty}(D_{x_n'}\mathcal{T}(t'))S_i(x_n')\\
    &=\lim\limits_{n\to\infty}(D_{x_n'}\mathcal{T}(t'))S(x_n').
\end{align*}

Denoting by $Q(x)$ and $Q(x'_n)$ the orthogonal projections onto $(D_x\mathcal{T}(t))S(x)$ and $(D_{x_n'}\T(t'))S(x_n')$ respectively, define 
$$\Pi_n=I-Q(x)+Q(x'_n)Q(x).$$
Note that $\parallel \Pi_n-I\parallel_{\mathcal{L}(X)}\to 0$ as $n\to+\infty$. Moreover 
$$\Pi_n((D_x\mathcal{T}(t))S(x))=\Pi_n(R(Q(x)))= R(Q(x'_n)Q(x))\subset R(Q(x'_n))\subset S(y_n),\ \forall n\in\mathbb{N}.$$ 
This concludes the proof.
\end{proof}

The following result is a consequence of Lemma \ref{lemma2.2.10} and the compactness of $\mathcal{A}$.
\begin{corollary}\label{lemmaA.0.10}
Let $T$ be the Birkhoff number of $W$ given in the construction of $S$ and $U$. For any $\nu>0$ and $t>T$ there exists $\epsilon>0$ such that if $x,p\in\mathcal{A}$ and $y=\T(t)x$, $q=\T(t)p$  satisfies
\begin{equation*}
    d(x,p)<\epsilon\ \ \ \text{and}\ \ \ d(y,q)<\epsilon
\end{equation*}
then there exist isomorphisms $F(x,q):X\to X$ and $G(x,q):X\to X$ such that
\begin{equation*}
    \parallel F(x,q)-I\parallel_{\mathcal{L}(X)}\le\nu\ \ \text{and}\ \ \parallel G(x,q)-I\parallel_{\mathcal{L}(X)}\le\nu
\end{equation*}
and
\begin{equation*}
       F(x,q)((D_x\mathcal{T}(t))S(x))\subset S(q)\ \ \ \text{and}\ \ \ G(x,q)((D_y\T(-t))U(y))\subset U(p).
\end{equation*}
\end{corollary}

Now we prove the last properties of the subbundles $S$ and $U$, such as the boundness \eqref{boundofPs}, to conclude the proof of Lipschitz Shadowing in $\mathcal{A}$. We recall that once we have all properties of the subbundles $S$ and $U$ proved, we can proceed exactly like the finite dimensional case to apply the Banach Fixed Point Theorem and obtain the Lipschitz Shadowing property.

For each $i\in\{1,\dots,p\}$, fix an open set $\mathcal{O}_i'\subset \mathcal{O}_i$ (where $\mathcal{O}_i$ are the neighborhoods introduced in Theorem \ref{lemma2.2.9}) to satisfy the property established in Lemma \ref{sainaoentra}. Let $T'>0$ be a Birkhoff number (Lemma \ref{birkhoff}) of the set $V'=\bigcup\limits_{i=1}^n\mathcal{O}_i'$.

Remembering that we are assuming item (H1) from Theorem \ref{teointroducao}, fix $C_0,C_1>0$ such that
\begin{equation}\label{exponentialboundedness}
    \parallel D_x\mathcal{T}(t)\parallel_{\mathcal{L}(X)}\le C_1e^{C_0t},\ \forall t\in[0,T'],\ \forall x\in\mathcal{A}.
\end{equation}

We recall that condition \eqref{exponentialboundedness} is often satisfied in the infinite dimensional setting. In fact, if the semigroup is defined through the variation of constants formula \cite{carvalho2012attractors} and $\mathcal{T}(t)\in\mathcal{C}^1(X)$ for $t\ge0$, then its derivative is also written through the variation of constants formula \cite[Theorem 6.33]{carvalho2012attractors} and  we can use Gronwall's inequality to obtain \eqref{exponentialboundedness}.

\begin{lemma}\label{lemma2.2.11}
The subbundles $S,U$ satisfy:
\begin{enumerate}
    \item There exists $M>0$ such that
    $$\sup\limits_{x\in\mathcal{A}}(\parallel P_s(x)\parallel_{\mathcal{L}(X)}+\parallel P_u(x)\parallel_{\mathcal{L}(X)})\le M,$$
    where $P_s$ and $P_u$ are the projections defined at the beginning of this subsection.
    \item There exists $C>0$ such that for any $x\in\mathcal{A}$ it holds
    \begin{equation}\label{montante1}
           \parallel (D_x\mathcal{T}(t))v^s\parallel\le Ce^{-\lambda_1t}\parallel v^s\parallel,\ \forall v^s\in S(x),  \ \forall t\ge 0
    \end{equation}
    and
    \begin{equation}\label{montante2}
         \parallel (D_x\mathcal{T}(-t))v^u\parallel\le Ce^{-\lambda_1t}\parallel v^u\parallel,\ \forall v^u\in U(x),  \ \forall t\ge 0,
    \end{equation}
         where $\lambda_1$ is the constant in Theorem \ref{lemma2.2.9}.
\end{enumerate}
\end{lemma}

\begin{proof}
Let us separate the proof by items.
\begin{enumerate}
    \item  
    
    It is sufficient to estimate the norm of $P_s$ in $\mathcal{A}$, since $P_u(x)=I-P_s(x)$ for all $x\in\mathcal{A}$. From the construction of $S$ we just need to estimate the norm of 
    the operators $P_{S_i}(x):=P_{S_i(x)U_i(x)}$ (in the notation established in Remark \ref{defprosepa}) for $x$ in
    \begin{equation*}
        \gamma_{[0,T]}^-(\mathcal{O}_i):=\{\mathcal{T}(-t)y:t\in [0,T], y\in\mathcal{O}_i\}.
    \end{equation*}
    
From Lemma \ref{inclinacaoltda} we know that  the inclination  of $S_i(x)$ related to the direct sum $U_i(x)\oplus U_i^{\perp}(x)=X$ is bounded in $\overline{\mathcal{O}_i}$, that is, there exists $M>0$ such that
 \begin{equation*}
    \sup\limits_{x\in \overline{\mathcal{O}_i}}\alpha(S_i(x))=\sup\limits_{x\in \overline{\mathcal{O}_i}}\sup\limits_{v\in S_i(x)\atop\parallel v\parallel=1}\dfrac{\parallel P_{U_iU_i^{\perp}}(x)v\parallel}{\parallel P_{U_i^{\perp}U_i}(x)v\parallel}\le M.
\end{equation*}

Let us estimate the inclination of $S_i(y)$ for $y\in \gamma_{[0,T]}^-(\mathcal{O}_i)$. Fix $x\in \mathcal{O}_i$, $v\in S_i(x)\cap R(D_{\mathcal{T}(-t)x}\mathcal{T}(t))$, $t\in[0,T]$ and define $x'=\mathcal{T}(-t)x$ and $v'=(D_x\mathcal{T}(-t))v\in S_i(x')$.  Note that
\begin{align*}
    \parallel P_{U_iU_i^{\perp}}(x)v-v\parallel&= d(v,U_i(x))\\
    &\le\inf\limits_{u\in U_i(x)\cap R(D_{\mathcal{T}(-t)x}\mathcal{T}(t))}\parallel  D_{\mathcal{T}(-t)x}\mathcal{T}(t)(v'-(D_x\mathcal{T}(-t))u) \parallel\\
    &\le \parallel D_{\mathcal{T}(-t)x}\mathcal{T}(t)\parallel_{\mathcal{L}(X)} \inf\limits_{u\in U_i(x')}\parallel v'-u \parallel\\
    &=\parallel D_{\mathcal{T}(-t)x}\mathcal{T}(t)\parallel_{\mathcal{L}(X)}\parallel v'-P_{U_iU_i^{\perp}}(x')v'\parallel.
\end{align*}

Hence
\begin{equation*}
    \dfrac{1}{ \parallel P_{U_i^{\perp}U_i}(x')v'\parallel}\le  \parallel D_{\mathcal{T}(-t)x}\mathcal{T}(t)\parallel_{\mathcal{L}(X)} \dfrac{1}{\parallel P_{U_i^{\perp}U_i}(x)v\parallel}.
\end{equation*}

Therefore for $v'\in S_i(x')$, $\parallel v'\parallel =1$,  we have
\begin{equation*}
    \dfrac{\parallel P_{U_iU_i^{\perp}}(x')v'\parallel}{\parallel P_{U_i^{\perp}U_i}(x')v'\parallel}\le \parallel D_{\mathcal{T}(-t)x}\mathcal{T}(t)\parallel_{\mathcal{L}(X)} \dfrac{1}{\parallel P_{U_i^{\perp}U_i}(x')v\parallel}\le M'\dfrac{1}{\parallel P_{U_i^{\perp}U_i}(x)v\parallel},
\end{equation*}
where $M'$ is obtained through \eqref{exponentialboundedness}. 
Now we can proceed as in the proof of Lemma \ref{inclinacaoltda}  to show that 
\begin{equation*}
    \inf\limits_{x\in \overline{\mathcal{O}_i}\atop v\in S_i(x),\parallel v\parallel=1}\parallel P_{U_i^{\perp}U_i}(x)v\parallel>0.
\end{equation*}

Therefore the inclination of $S_i(x)$ in relation to the direct sum $U_i(x)\oplus U^{\perp}_i(x)=X$ is bounded for $x\in\gamma^-_{[0,T]}(\mathcal{O}_i)$. Now we can use Proposition \ref{boundedsetprojections} to conclude the proof.

\item With \eqref{exponentialboundedness} we can proceed as in the finite dimensional case \cite[Chapter 2]{Pilyugin} to show \eqref{montante1}. Inequality \eqref{montante2} follows similarly. In fact, we can use Corollary \ref{normainversaconv} together with the fact that $D_x\mathcal{T}(t)$ is an isomorphism onto its range and the subbundles $U_i$ are continuous in $\gamma({\mathcal{O}_i})$ (since they are finite dimensional) to also bound the norm of $(D_{x}\mathcal{T}(-t))|_{U(x)}$  uniformly for $t\in[0,T']$  and $x\in\mathcal{A}$. Hence, we can assume without loss of generality that $C_0$ and $C_1$ also satisfies
\begin{equation*}
    \parallel (D_x\mathcal{T}(-t))|_{U(x)}\parallel_{\mathcal{L}(U(x),X)}\le C_1e^{C_0|t|},\ \forall t\in[0,T'],\ \forall x\in\mathcal{A}
\end{equation*}
and proceed exactly as in the proof of \eqref{montante1}.
\end{enumerate}
\end{proof}

We finally announce the main result of this Section, that is item (1) of Theorem \ref{teointroducao}.

\begin{theorem}[Lipschitz Shadowing]\label{shadowing}

Let $\mathcal{T}=\{\mathcal{T}(t):t\ge0\}$ be a Morse-Smale semigroup with non-wandering set $\Omega=\{x^*_1,\dots,x^*_p\}$. Assume that $\mathcal{T}$ satisfies conditions (H1) and (H2). Then, the map $\mathcal{T}(1)|_{\mathcal{A}}:\mathcal{A}\to\mathcal{A}$ has the Lipschitz Shadowing property.
\end{theorem}
	\begin{proof}
	     With   Lemma \ref{lemma2.2.11} we can apply  the Banach Fixed Point Theorem \cite[Theorem 1.3.1]{Pilyugin}, with no obstacles in the infinite dimensional context,  and proceed as in the finite dimensional case \cite{Pilyugin,santamaria2014distance}. To preserve the completeness of the manuscript, we  provide an  sketch of the proof.

Since the proof of the theorem is very technical, we first discuss about the key ideas and intuition behind it. We want to find $d_0>0$ and $L>0$ such that for any  $d-$pseudo orbit $\{x_k\}_{k\in\mathbb{Z}}$ of $\mathcal{T}(1)|_{\mathcal{A}}$, $d\le d_0$,  there exists an orbit of $\mathcal{T}(1)$ that $Ld-$shadows $\{x_k\}_{k\in\mathbb{Z}}$. Hence, given a $d-$pseudo-orbit $\{x_k\}_{k\in\mathbb{Z}}$, $d\le d_0$, we must find a sequence of vectors $\{v_k\}_{k\in\mathbb{N}}\subset\mathcal{A}$ satisfying
\begin{equation*}
    \mathcal{T}(1)(x_k+v_k)=x_{k+1}+v_{k+1}\ \ \ \text{and}\ \ \ \parallel v_k\parallel\le Ld,\ \forall k\in\mathbb{N}.
\end{equation*}

In this case, the sequence $\{y_k:=x_k+v_k\}_{k\in\mathbb{Z}}$ is an orbit of $\mathcal{T}(1)|_{\mathcal{A}}$. The sequence $\{v_k\}_{k\in\mathbb{N}}$ will be obtained through the Banach fixed point theorem. To be more precise, consider the Banach space $(X^{\mathbb{Z}},\parallel\cdot\parallel_{X^{\mathbb{Z}}})$  given by
    \begin{equation*}
 X^{\mathbb{Z}}:=\{\{v_k\}_{k\in\mathbb{Z}}\subset X:  \sup\limits_{k\in\mathbb{Z}}\parallel v_k\parallel<+\infty\}\ \ \text{and}\ \ \ \parallel \{v_k\}_{k\in\mathbb{Z}}\parallel_{X^{\mathbb{Z}}}:=\sup\limits_{k\in\mathbb{Z}}\parallel v_k\parallel.
    \end{equation*}

We just have to show that the following map
\begin{align*}
    \Phi:&B_{X^{\mathbb{Z}}}[0,Ld]\to B_{X^{\mathbb{Z}}}[0,Ld]\\
    &\{v_k\}_{k\in\mathbb{Z}}\mapsto \{\mathcal{T}(1)(x_{k-1}+v_{k-1})-x_k\}_{k\in\mathbb{Z}},
\end{align*}
is a contraction, where $B_{X^{\mathbb{Z}}}[0,Ld]$ denotes the closed ball in $X^{\mathbb{Z}}$ centered in $0$ with radius $R$.

Building on the intuition provided above, we now present a sketch of the proof. We will use the properties of the subbundles $S$ and $U$, stated in Lemma \ref{lemma2.2.11}, to satisfy the assumptions of Theorem \ref{teoremaabstrato}  (see Appendix) and obtain an orbit of $\mathcal{T}(1)|_{\mathcal{A}}$ close to a pseudo-orbit via Banach fixed point theorem. Let us fix come constants that are  fundamental to achieve the hypothesis from Theorem \ref{teoremaabstrato}. 

Fix $\mu\in (0,1)$ arbitrary and let $T$ be the Bihrkhoff number mentioned in Corollary \ref{lemmaA.0.10}. From Lemma \ref{lemma2.2.11} we may take an integer $N$, with $N>T$, such that
\begin{equation}\label{equacaomu}
\parallel (D_x\mathcal{T}(n))\ v^s\parallel\le\mu\parallel v^s\parallel,\ \forall v^s\in S(x),\ \forall n\ge N-1,
\end{equation}
\begin{equation*}
\parallel (D_x\mathcal{T}(-n))\ v^u\parallel\le\mu\parallel v^u\parallel,\ \forall v^u\in U(x),\ \forall n\ge N-1.
\end{equation*}

Let $M>0$ be the constant given in Lemma \ref{lemma2.2.11} satisfying
\begin{equation*}
    \parallel P_s(x)\parallel_{\mathcal{L}(X)},\parallel P_u(x)\parallel_{\mathcal{L}(X)}\le M,\ \forall x\in\mathcal{A}.
    \end{equation*}

Since $\mathcal{T}(N)\in\mathcal{C}^1(X)$ and $\mathcal{A}$ is compact, we may choose $K\ge M$ such that 
\begin{equation*}
    \parallel D_x\mathcal{T}(N)\parallel_{\mathcal{L}(X)} \le K,\ \forall x\in\mathcal{A}.
\end{equation*}

Fix $\nu_0\in(0,1)$ such that $\lambda:=(1+\nu_0)\mu<1$ and define
\begin{equation*}
    N_1=M\frac{1+\lambda}{1-\lambda}.
\end{equation*}

Finally, take constants $k_1, \nu\in(0,\nu_0)$ such that 
\begin{equation}\label{defmu}
    k_1N_1<1\ \ \ \ \text{and}\ \ \ \ K(2K+1)\nu<\dfrac{k_1}{2}.
    \end{equation}
    
We will prove the Lipschitz Shadowing property for $\psi:=\mathcal{T}(N)|_{\mathcal{A}}$ and the  result for $\mathcal{T}(1)|_{\mathcal{A}}$ will follow from Lemmas \ref{lemaimportante} and \ref{auxiliar1}. Fix $d>0$ and let $\{x_k\}_{k\in\mathbb{Z}}\subset\mathcal{A}$ be a $d$-pseudo orbit of $\psi$. In order to align with the framework of Theorem \ref{teoremaabstrato}, for each $k\in\mathbb{Z}$ define the spaces $E_k=X$, the projections $P_k:=P_s(x_k)$, $Q_k:=P_u(x_k)$  and the maps
\begin{align*}
    \phi_k:&E_k\to E_{k+1}\\
    & v\mapsto \psi(x_k+v)-x_{k+1}
\end{align*}

Note that $D_0\phi_k=D_{x_k}\psi$ and
\begin{equation*}
    \phi_k(v)=(D_0\phi_k)v+h_{k+1}(v), \forall v\in E_k,
\end{equation*}
where $h_{k+1}:=\phi_k-D_0\phi_k$.

Since
$D_0h_{k+1}=0$ and $\psi\in\mathcal{C}^1(X)$, there exists $\Delta>0$ such that    
\begin{equation}\label{ajuda2}
    \parallel h_{k+1}(v)-h_{k+1}(v')\parallel\le \dfrac{k_1}{2}\parallel v-v'\parallel,\ \forall \parallel v\parallel,\parallel v'\parallel\le \Delta.
\end{equation}

Considering $\nu$ defined in \eqref{defmu}, fix $\epsilon>0$  as in Lemma \ref{lemmaA.0.10}. Since $\psi^{-1}$ is uniformly continuous in $\mathcal{A}$, we may reduce $\Delta>0$ such that $\Delta\in(0,\epsilon)$ also satisfies
\begin{equation*}
    x,y\in\mathcal{A}, \parallel x-y\parallel<\Delta\Longrightarrow\parallel \psi^{-1}x-\psi^{-1}y\parallel<\epsilon.
\end{equation*}

From now on we assume $d\in(0,\Delta)$. In particular,
\begin{equation*}
    \parallel\psi(x_k)-x_{k+1}\parallel\le \epsilon\ \ \  \text{and}\ \ \ \parallel x_k- \psi^{-1}(x_{k+1})\parallel\le \epsilon,\ \forall k\in\mathbb{Z}.
\end{equation*}

 It follows from Corollary \ref{lemmaA.0.10} that there exist isomorphisms $F(x_k,x_{k+1})$ and $G(x_k,x_{k+1})$ satisfying
\begin{equation*}
    \parallel F(x_k,x_{k+1})-I\parallel, \parallel G(x_k,x_{k+1})-I\parallel, \parallel G(x_k,x_{k+1})^{-1}-I\parallel\le\nu,
\end{equation*}
where $\nu$ is the constant in \eqref{defmu}, and
\begin{equation*}
    F(x_k,x_{k+1})(D_{x_k}\psi)P_k\subset S(x_{k+1}),\ G(x_k,x_{k+1})(D_{x_{k+1}}\psi^{-1})U(x_{k+1})\subset U(x_k).
\end{equation*}

Using the notation from Lemma \ref{lemma2.2.11} define the operators
\begin{equation*}
    A_k^s:=F(x_k,x_{k+1})(D_{x_k}\psi)P_k,\  \ \ \ A_k^u:=(D_{\psi^{-1}(x_{k+1})}\psi)G^{-1}(x_k,x_{k+1})Q_k
    \end{equation*}
and
\begin{equation*}
     B_k:=G(x_k,x_{k+1})D_{x_{k+1}}\psi^{-1},\ \ \ \ \   A_k:=A_k^s+A_k^u.
\end{equation*}
   
Note that $A_k^s$ relates the subspaces $S(x_k)$ and $S(x_{k+1})$ and  $B_k$ relates $U(x_{k+1})$ and $U(x_k)$. 

Now we can write $\phi_k$ as 
$$\phi_k=A_k+w_{k+1},$$
where
\begin{equation*}
    w_{k+1}=\phi_k-A_k=[D_{x_k}\psi-A_k]+h_{k+1}.
\end{equation*}

 One can show that the operators $P_k$, $Q_k$, $A_k, B_k, w_k$ and the constants $M, \lambda, k_1, N_1, \Delta$ satisfy all conditions from Theorem \ref{teoremaabstrato}. Therefore $\phi_k$ satisfies all conditions from Theorem \ref{teoremaabstrato} and we can finally apply it to guarantee the existence of constants $d_1, L>0$ satisfying (\ref{hipabstrato}) and (\ref{teseabstrato}).

Note that if $\{x_k\}_{k\in\mathbb{Z}}\subset\mathcal{A}$ is a $d$-pseudo orbit of $\psi$ with $0\le d\le d_0:=\min\{\Delta,d_1\}$, we have
\begin{equation*}
    \parallel \phi_k(0)\parallel\le d,
\end{equation*}
and therefore, from Theorem \ref{teoremaabstrato}, there exists sequence $\{v_k\}_{k\in\mathbb{N}}\subset X$ satisfying  
\begin{equation*}
    \phi_k(v_k)=v_{k+1}\ \ \text{and}\ \ \parallel v_k\parallel\le Ld, \ \forall k\in\mathbb{Z}.
\end{equation*}

Consequently, $v_{k+1}=\phi_k(v_k)=\mathcal{T}(N)(x_k+v_k)-x_{k+1}$ for each $k\in\mathbb{Z}$, which leads to 
\begin{equation*}
    \mathcal{T}(N)(x_k+v_k)=x_{k+1}+v_{k+1},\ \forall k\in\mathbb{Z}.
\end{equation*}

Therefore $\{y_k:=x_k+v_k\}_{k\in\mathbb{Z}}$ is an orbit of $\psi=\mathcal{T}(N)$ and
\begin{equation*}
    \parallel \psi^ky_0-x_{k}\parallel=\parallel y_k-x_k\parallel\le Ld,\ \forall k\in\mathbb{Z}.
\end{equation*}

This concludes the proof that $\mathcal{T}(N)|_{\mathcal{A}}$ has the Lipschitz Shadowing property on $\mathcal{A}$ with constants $L,d_0$.

	\end{proof}

\section{The Neighborhood of the Global Attractor}\label{sectionneighbo}

In this section we will prove item 2 from Theorem \ref{teointroducao}, that is,  for any positively invariant bounded neighborhood $\mathcal{U}\supset\mathcal{A}$ the map $\mathcal{T}(1)|_{\mathcal{A}}:\mathcal{A}\to \mathcal{A}$ has the H\"{o}lder Shadowing property. To do that, we will first show that if $\{x_n\}_{n\in\mathbb{Z}}\subset\mathcal{U}$ is a  $\delta-$pseudo orbit of $\mathcal{T}(1)$, where $\mathcal{U}$ is  a bounded positively invariant neighborhood, then the smaller the $\delta>0$ the closer $\{x_n\}_{n\in\mathbb{Z}}$ is to the global attractor $\mathcal{A}$, in a way that $\sup\limits_{n\in\mathbb{N}}d(x_n,\mathcal{A})\to 0$ as $\delta\to 0$ (Lemma  \ref{lemma-distancia}). We recall that if $\delta=0$ then $\mathcal{T}(1)x_n=x_{n+1}$ for all $n\in\mathbb{Z}$ and $\{x_n\}_{n\in\mathbb{Z}}$ is a bounded orbit of $\mathcal{T}(1)$, that is, $\{x_n\}_{n\in\mathbb{Z}}\subset\mathcal{A}$ (Proposition \ref{characatt}) and $d(x_n,\mathcal{A})=0$ for all $n\in\mathbb{Z}$.

\begin{lemma}\label{lemma-distancia}
Let $\{\T(t):t\ge0\}$ be a semigroup in a metric space $(M,d)$ with global attractor $\mathcal{A}$. Assume that there exists a positively invariant neighborhood $\U$ of $\mathcal{A}$ such that $\T:=\T(1)|_{\U}:\U\to\U$ is Lipschitz, with Lipschitz constant $L_1$, and there exists $C,\gamma>0$ such that
\begin{equation}\label{pedrogonza}
dist_H(\T^n\mathcal{U},\mathcal{A})\le Ce^{-\gamma n},\ \forall n\in\mathbb{N}.
\end{equation}
Then there exists $C_2>0$ and $\alpha\in(0,1)$ such that if
$\{x_n\}_{n\in\mathbb{Z}}$ is a  $d$-pseudo-orbit of $\T(1)$ in $\mathcal{U}$,  $0<d<1$, then
\begin{equation}\label{eq-lemma-distancia}
    d(x_n,\mathcal{A})\le C_2\cdot d^{\alpha}, \ \forall n\in\mathbb{Z}.
\end{equation}
\end{lemma}
\begin{proof} 
Without loss of generality we may assume that $L_1>1$ and also that $C>1$. 
Fix $\beta\in(0,1)$ and define 
\begin{equation*}
    K(d):=d^{\gamma \frac{1-\beta}{\log(L_1)}}, 
\end{equation*}
where $\gamma$ is the one from \eqref{pedrogonza}. It trivially satisfies $0<K(d)<1$ as long as $0<d<1$. 
Take $N\in\mathbb{N}$ such that 
\begin{equation}\label{eq-way-to-find-N-lema-distancia}
    Ce^{-\gamma N}\le K(d)\le Ce^{-\gamma (N-1)}.
\end{equation}
This is always possible since $C>1$, $K(d)\in (0,1)$ and the sequence $Ce^{-\gamma N}$ is a monotone decreasing sequence converging to $0$ as $N\to+\infty$ 


Define $z_n=x_{nN}$, for all $n\in\mathbb{Z}$. From Lemma \ref{lemaimportante} we know that $\{z_n\}_{n\in \mathbb{Z}}$ is a $C_1d$-pseudo-orbit of $\T^N$, that is
$$d(T^Nz_n,z_{n+1})\leq C_1d, \quad n\in \mathbb{Z}$$
where $$C_1=1+L_1+\cdots+L_1^{N-1}=\dfrac{L_1^N-1}{L_1-1}\leq \frac{1}{L_1-1} L_1^N.$$

Considering the second inequality on \eqref{eq-way-to-find-N-lema-distancia} and with simple computations we get
$$e^{\gamma(N-1)}\leq \frac{C}{K(d)}$$ 
and taking logarithms 
$$N\leq  1+\frac{1}{\gamma}\log(\frac{C}{K(d)})=1+\frac{\log(C)}{\gamma}-\frac{\log(K(d))}{\gamma}$$

Hence, 
\begin{equation}
    L_1^N\le  L_1^{1+\frac{\ln{C}}{\gamma}} L_1^{\frac{-\log{K(d)}}{\gamma}} 
\end{equation}
and using the expression of $K(d)$, we easily get  $ L_1^{\frac{-\log{K(d)}}{\gamma}}=d^{\beta -1}$.  Therefore, 

$$C_1d\leq \frac{1}{L_1-1}L_1^Nd\leq \frac{L_1^{1+\frac{\ln{C}}{\gamma}} }{L_1-1}d^\beta$$

Thus if we denote by $C'= \frac{L_1^{1+\frac{\ln{C}}{\gamma}} }{L_1-1}$  which is independent of $d$ we have 
$$d(T^Nz_n,z_{n+1})\le C_1 d\le C'd^{\beta}, \quad  n\in \mathbb{Z}$$
From \eqref{pedrogonza} and  \eqref{eq-way-to-find-N-lema-distancia} we conclude that
\begin{align*}
d(z_{n+1},\mathcal{A})&\le d(z_{n+1},T^Nz_n)+d(T^Nz_n,\mathcal{A})\\
&\le C'd^{\beta}+K(d)=C'd^{\beta}+ d^{\gamma \frac{1-\beta}{\log(L_1)}}= (C'+1)d^{\alpha},
\end{align*}
where $\alpha=\min\{\beta,\frac{\gamma(1-\beta)}{\ln{L_1}} \}.$ We can take $\beta=\frac{\gamma}{\gamma+\ln{L_1}}$ to maximize $\alpha$. 
Thus, taking 
\begin{equation}\label{c2}
    C_2=\frac{L_1^{1+\frac{\ln{C}}{\gamma}} }{L_1-1}+1
\end{equation}
 we obtain \eqref{eq-lemma-distancia} for the sequence $\{z_n\}_{n\in \mathbb{Z}}$.
To complete the proof apply the above arguments to the sequence $\{x_{n-r}\}_{n\in\mathbb{Z}}$ for each $r\in\{1,\dots,N-1\}$.
\end{proof}

Note that Lemma \ref{lemma-distancia} does not require the assumption of Morse-Smale, but it assumes an exponential attraction, which is a more general property. In fact, the assumption of exponential attraction given in \eqref{pedrogonza} is satisfied for Morse-Smale semigroups, as we  announce in the following lemma.

 \begin{lemma}\label{exponentialattractor}
    Let $\{\T(t):t\ge0\}$ be a Morse-Smale map in a Banach Space $X$ with global attractor $\mathcal{A}$. Then $\mathcal{A}$ attracts bounded sets exponentially, that is, for each bounded set $\mathcal{U}\subset X$  there exist $C,\gamma>0$, called constants of attraction, such that
$$dist_H(\T^n(\mathcal{U}),\mathcal{A})\le Ce^{-\gamma n}, \ \forall n\in\mathbb{N}.$$
\end{lemma}
\begin{proof}
     In \cite{carvalho2011exponential} the authors prove that any Gradient  semigroup (satisfies item (1) of Lemma \ref{lemma-Lyapunov-function})  with finite set of equilibria $\mathcal{E}=\{x_1^*,\dots,x_p^*\}$, where  $x_i^*$ is hyperbolic for each $i\in\{1,\dots,p\}$, satisfies the property of exponential attraction.
\end{proof}

Proceeding similarly to Lemma \ref{lemma-distancia}, we can obtain the following result depending of the Lipschitz constant $L_1$.
\begin{lemma}\label{teocontracao}
    Assume the same conditions of Lemma \ref{lemma-distancia}. Then, there exists $L>0$  such that if $\{x_n\}_{n\in\mathbb{Z}}\subset\U$ is a $d-$ pseudo orbit of $\mathcal{T}|_{\mathcal{U}}$, with $d\in(0,1)$, we have:
    \begin{enumerate}
    \item If $L_1<1$, then 
     \begin{equation*}
        d(x_n,\mathcal{A})\le Ld,\ \forall n\in\mathbb{Z}.
    \end{equation*} 
   
    \item If $L_1=1$, then
     \begin{equation*}
        d(x_n,\mathcal{A})\le Ld|\ln{d}|,\ \forall n\in\mathbb{Z}.
    \end{equation*}
\end{enumerate}

\end{lemma} 
\begin{proof}
    To prove item (1) we can proceed as in Lemma \ref{lemma-distancia} taking $K(d)=d$ and noting that $L_1^N<1$. In this case we obtain $L=\left(\dfrac{1}{1-L_1}\right)$. Similarly, we can prove item (2) taking $K(d)=d$ and noting that $C_1(d)=N$.
\end{proof}

\begin{remark}
    If $L_1<1$ then $\mathcal{T}(1)|_{\overline{\mathcal{U}}}$ is a contraction and has a unique equilibrium point $x^*$, which is asymptotically stable. Therefore, $\mathcal{A}=\{x^*\}$ and the dynamics in this case is trivial.
\end{remark}

We know announce our abstract result of H\"{o}lder Shadowing in a neighborhood of the global attractor. We will obtain item (2) of Theorem \ref{teointroducao} as a corollary of this result.
\begin{theorem}
\label{shadowingviz} Let $\{\T(t):t\ge0\}$ be a semigroup in a metric space $(M,d)$ with global attractor $\mathcal{A}$. Assume that $\T(1)|_{\mathcal{A}}:\mathcal{A}\to \mathcal{A}$ has the  Lipschitz Shadowing property. Moreover, assume that there exists $C>0$,  $\alpha\in(0,1)$ and a bounded positively invariant neighborhood $\U\supset \mathcal{A}$ such that for any $d\in[0,1)$ and any $d-$pseudo orbit $\{x_n\}_{n\in\mathbb{Z}}\subset\U$ there exists a sequence $\{\tilde{x}_n\}_{n\in\mathbb{Z}}\subset\mathcal{A}$ satisfying
\begin{equation*}
    d(x_n,\tilde{x}_n)\le Cd^{\alpha},\ \forall n\in\mathbb{Z}.
\end{equation*}

Then  $\mathcal{T}:=\mathcal{T}(1)|_{\U}:\U\to\U$ has the $\alpha$-H\"{o}lder Shadowing property.
\end{theorem}
\begin{proof}
From our assumptions we can fix $d_0',L>0$ such that if  $\{y_n\}_{n\in\mathbb{Z}}\subset \mathcal{A}$ satisfies
$$d(\T y_n,y_{n+1})\le d\le d_0', \forall\ n\in \mathbb{Z},$$
then there exists $x\in\mathcal{A}$ such that
$$d(\T^nx,y_n)\le Ld, \forall \ n\in\mathbb{Z}.$$

Suppose that $\{x_n\}_{n\in\mathbb{Z}}\subset \mathcal{U}$ is a  $d$-pseudo-orbit of $\T$, with $0\le d\le 1$. Then  there exists a sequence $\{\tilde{x_n}\}_{n\in\mathbb{Z}}\subset \mathcal{A}$  such that $$d(x_n,\tilde{x_n})\le Cd^{\alpha},\ \forall n\in\mathbb{Z}.$$

Let us show that $\{\tilde{x}_n\}$ is a pseudo-orbit. Note that
    \begin{align*}
     d(\T\Tilde{x_n},\Tilde{x}_{n+1})&\le  d(\T\Tilde{x_n},\T x_n)+d(\T x_n,x_{n+1})+d(x_{n+1},\Tilde{x}_{n+1})\\
    & \le L_1Cd^{\alpha}+d^{\alpha}+Cd^{\alpha}\\
    &=C_2d^{\alpha},
\end{align*}
where $C_2=(L_1C+1+C)$. Fix $d_0\in(0,1)$ satisfying  $C_2d_0^{\alpha}\le d_0'$. Then, if $d\le d_0$ the sequence $\{\tilde{x}_n\}_{n\in\mathbb{Z}}$ is a $C_2d^{\alpha}$-pseudo-orbit in $\mathcal{A}$ and since $\mathcal{T}(1)|_{\mathcal{A}}$ has the Lipschitz Shadowing property, there exists $x\in\mathcal{A}$ such that
\begin{equation*}
    d(\T^nx,\Tilde{x_n})\le LC_2d^{\alpha},\forall\ n\in\mathbb{Z},
\end{equation*}
where $L$ is the constant of the Lipschitz Shadowing.

Therefore
\begin{equation*}
\begin{aligned}
     d(\T^nx,x_n)&\le d(\T^nx,\Tilde{x_n})+d(\Tilde{x_n},x_n)\\
     &\le C_2Ld^{\alpha}+C_2d^{\alpha}\\
     &\le C_3d^{\alpha},
\end{aligned}
\end{equation*}
where $C_3=(L+1)C_2$. This proves that there exists an orbit close to the $d-$pseudo orbit $\{x_n\}_{n\in\mathbb{Z}}\subset\mathcal{U}$.
Hence, the constants $d_0, \alpha$ and $C_3$  fulfill the properties required in the H\"{o}lder Shadowing property .
\end{proof}

Finally, we obtain item (2) from Theorem \ref{teointroducao}.
\begin{corollary}\label{msholdershado}
    Let $\{\mathcal{T}(t):t\ge0\}$ be a Morse-Smale semigroup in a Hilbert space $X$ that satisfies (H1) and (H2). Then, for any positively invariant bounded neighborhood $\U\supset \mathcal{A}$, the map $\mathcal{T}(1)|_{\U}:\U\to\U$ has the H\"{o}lder Shadowing property.
\end{corollary}
\begin{proof}
    It follows straightforward from Theorem \ref{shadowing}, Lemmas \ref{lemma-distancia}, \ref{exponentialattractor} and Theorem \ref{shadowingviz}.
\end{proof}

To finish this section, we announce a general result of Shadowing in a neighborhood $\U\supset\mathcal{A}$ of the global attractor, where each type of Shadowing depends on the Lipschitz constant of the map $\mathcal{T}(1)|_{\U}:\U\to \U$.
 \begin{theorem}\label{teodasconstantesdelip}
Let $M$ be a metric space and $\{\T(t):t\ge0\}$ be a semigroup in $M$   with global attractor $\mathcal{A}$. Assume that $\mathcal{U}\subset M$ is a bounded positively invariant neighborhood of $\mathcal{A}$ such that $\T(1)|_{\mathcal{U}}:\U\to\U$ is Lipschitz, with Lipschitz constant $L_1> 0$.  Suppose that $\T(1)|_{\mathcal{A}}:\mathcal{A}\to\mathcal{A}$ has the Lipschitz Shadowing property and $\U$ is exponentially attracted by $\mathcal{A}$, as in \eqref{pedrogonza}. Then, It holds:
\begin{enumerate}
    \item If $L_1<1$ then $\T(1)|_{\mathcal{U}}$ has the Lipschitz Shadowing property;
    \item If $L_1=1$ then $\T(1)|_{\mathcal{U}}$ has the Logarithm Shadowing property;
\item If $L_1>1$ then $\T(1)|_{\mathcal{U}}$ has the H\"{o}lder Shadowing property.
\end{enumerate}
\end{theorem}
\begin{proof}
    Item (3) was already proved in Theorem \ref{shadowingviz}, where we used Lemma \ref{lemma-distancia} to estimate the distance from pseudo-orbits to the global attractor $\mathcal{A}$. To prove items (1) and (2) we reproduce the proof of Theorem \ref{shadowingviz} but using Lemma  \ref{teocontracao} to estimate the distance from the pseudo-orbits to $\mathcal{A}$, instead of Lemma \ref{lemma-distancia}.
\end{proof}

\section{Applications}\label{seccont}

In this section we apply the  results from Sections \ref{sectionnoatrator} and \ref{sectionneighbo} to obtain new information regarding the structural stability of Morse-Smale semigroups (Theorems \ref{robustezams} and \ref{gssna}) and 
continuity of global attractors (Theorem \ref{contfinal}).

Let us start with our results related to structural stability of Morse-Smale semigroups. It is known that small perturbations of a Morse-Smale  system, defined in a finite or infinite dimensional phase space, are still Morse-Smale (topological structural stability) and  there is a phase-diagram isomorphism between the global attractors (geometrical structural stability) \cite{Bortolan-Carvalho-Langa-book}. In the finite dimensional case,  where the phase space $M$ is a smooth compact (with no border) manifold, it is also know that there exists a topological conjugation between the perturbed Morse-Smale systems and the non-perturbed one \cite{Pilyugin}. Moreover, the homeomorphism (obtained from the conjugation) is uniformly close to the identity. Consequently, in the finite dimensional case we know that the orbits from the perturbed global attractor (that is the entire phase space) remain close to the orbits from the original global attractor. 

In the infinite dimensional setting,  the proof of the topological and geometrical structural stability of Morse-Smale semigroups \cite{bortolan2022nonautonomous,Bortolan-Carvalho-Langa-book} relies on a passage to the limit argument, and therefore does not provide  any information regarding the distance of the orbits from the perturbed and non-perturbed attractors. Motivated by this, the following result provides new information  concerning the proximity of the orbits associated with the perturbed and non-perturbed attractors. The concepts of continuity and collectively asymptotically compactness of a family of semigroups, see Definition \ref{conteaskc} and \cite{Bortolan-Carvalho-Langa-book,carvalho2012attractors}, will be used in the following theorem.

\begin{theorem}\label{robustezams}
    Let \( \{\mathcal{T}_\epsilon\}_{\epsilon \in [0,1]} \) be a family of semigroups defined in a Hilbert space $X$. Suppose \( \{\mathcal{T}_\epsilon\}_{\epsilon \in [0,1]} \) is continuous and collectively asymptotically compact at $ \epsilon = 0$ (see Definition \ref{conteaskc}).
Moreover, assume that:
\begin{itemize}
    \item[(a)]\ \  \( \mathcal{T}_\epsilon \) has a global attractor \( \mathcal{A}_\epsilon \) for each \( \epsilon \in [0, 1] \) and $\overline{ \bigcup\limits_{\epsilon \in (0,1]} \mathcal{A}_\epsilon }$ is compact;
    \item[(b)]\ \  \( \{\mathcal{T}_\epsilon\}_{\epsilon \in [0,1]} \subset \mathcal{C}^1(X) \) and \( \mathbb{R}^+ \times [0,1] \ni (t, \epsilon) \mapsto D_x\mathcal{T}_\epsilon(t) \in \mathcal{L}(X) \) is continuous for each \( x \in X \);
    \item[(c)]\ \ \( \mathcal{T}_0 \) is a Morse-Smale semigroup, satisfying conditions (H1) and (H2), with non-wandering set  
    \( \Omega_0=\mathcal{E}_0 = \{x_{1,0}^*, \cdots, x_{p,0}^*\} \).
\end{itemize}
Then, there exists an \( \epsilon_1 \in (0, 1] \) such that if $\epsilon\in[0,\epsilon_1)$, then \( \mathcal{T}_\epsilon \) is a Morse-Smale semigroup. Moreover, there exist $C>0$ and $\alpha\in(0,1)$  such that if $\xi_{\epsilon}$ is a bounded global orbit of $\mathcal{T}_{\epsilon}$, $\epsilon\in[0,\epsilon_1)$, then there exists a bounded global orbit $\xi_0$ of $\mathcal{T}_0$ such that
\begin{equation*}
   \parallel\xi_0(t)-\xi_{\epsilon}(t)\parallel\le C\parallel \mathcal{T}_0(1)-\mathcal{T}_{\epsilon}(1)\parallel^{\alpha}_{L^{\infty}(\mathcal{A}_{\epsilon},X)},\ \forall t\in\mathbb{R}.
\end{equation*}
\end{theorem}
\begin{proof}
As we said before, the robustness of the Morse–Smale structure under small perturbations (topological structural stability) is established in \cite{Bortolan-Carvalho-Langa-book}. It remains to show that  bounded orbits from the perturbed semigroups are close to bounded orbits from the limit semigroup.
     Let $\xi_{\epsilon}$ be a global solution of $\mathcal{T}_{\epsilon}$  and define the sequence $x_n=\xi_{\epsilon}(n)$ for all $n\in\mathbb{Z}$. Then,
\begin{equation*}
d(\mathcal{T}_0(1)x_n,x_{n+1})=d(\mathcal{T}_0(1)\xi_{\epsilon}(n),\mathcal{T}_{\epsilon}(1)\xi_{\epsilon}(n))\le \parallel \mathcal{T}_0(1)-\mathcal{T}_{\epsilon}(1)\parallel_{L^{\infty}(\mathcal{A}_{\epsilon},X)}.
\end{equation*}

Hence, if we define $\delta_{\epsilon}=\parallel \mathcal{T}_0(1)-\mathcal{T}_{\epsilon}(1)\parallel_{L^{\infty}(\mathcal{A}_{\epsilon},X)}$ then $\{x_n\}_{n\in\mathbb{Z}}$ is a $\delta_{\epsilon}-$pseudo orbit of $\mathcal{T}_0(1)$ in $\mathcal{U}$. Now the result  follows from  item (2) of Theorem \ref{teointroducao}.
\end{proof}

We can also generalize Theorem \ref{robustezams} to non-autonomous perturbations of a Morse-Smale semigroup. The following result will use  concepts from the non-autonomous theory, such as evolution processes, pullback attractors, hyperbolic global solutions and non-autonomous Morse-Smale evolution process \cite{carvalho2012attractors,Bortolan-Carvalho-Langa-book}. These concepts generalize the notion of semigroups, global attractors, hyperbolic equilibrium, etc.

\begin{theorem}\label{gssna}
Let $\left\{\mathcal{S}_\epsilon\right\}_{\epsilon \in [0,1]}$ be a family of evolution processes in a Hilbert space $X$ such that:
\begin{enumerate}
\item[(a)] $\mathcal{S}_\epsilon$ has a pullback attractor $\mathbb{A}_\epsilon=\{A_{\lambda}(t)\}_{t\in\mathbb{R}}$ and is reversible for each $\epsilon \in [0,1]$;
\item[(b)] $\overline{\bigcup_{\epsilon \in [0,1]} \bigcup_{t \in \mathbb{R}} A_\epsilon(t)}$ is compact;
\item[(c)] for each compact set $K \subset \mathbb{R}^{+} \times X$ we have
$$
\sup _{s \in \mathbb{R}} \sup _{(t, x) \in K}\left\|\mathcal{S}_\epsilon(t+s, s) x-\mathcal{S}_{0}(t+s, s) x\right\|_X \rightarrow 0 \quad \text { as } \epsilon \rightarrow 0
$$
\item[(d)] for each compact set $J \in \mathbb{R}^{+}$we have
$$
\sup _{s \in \mathbb{R}} \sup _{t \in J} \sup _{z \in A(s)}\left\|\mathcal{S}_\epsilon^{\prime}(t+s, s)(z)-\mathcal{S}_{0}^{\prime}(t+s, s)(z)\right\|_{\mathcal{L}(X)} \rightarrow 0 \quad \text { as } \epsilon \rightarrow 0
$$
\item[(e)] $\mathcal{S}_{0}$ is a Morse-Smale semigroup (autonomous) satisfying conditions (H1) and (H2) and with non-wandering set  $\mathcal{E}=\left\{e_1^*, \cdots, e_p^*\right\}$, where $e_i^*$ is a hyperbolic equilibrium point for each $i\in\{1,\dots,p\}$.
\end{enumerate}
   
    Then there exists $\epsilon_1\in(0,1]$ such that if $\epsilon\in[0,\epsilon_1)$ then $\mathcal{S}_{\epsilon}$ is a Morse-Smale evolution process. Moreover, there exist $C>0$ and $\alpha\in(0,1)$ such that if $\xi_{\epsilon}$ is a bounded global orbit of $\mathcal{S}_{\epsilon}$, with $\epsilon\in[0,\epsilon_1)$, then there exists a bounded global orbit $\xi_{0}$ of $\mathcal{S}_{0}$ such that
   \begin{equation*}
       \parallel\xi_{\epsilon}(t)-\xi_0(t)\parallel\le C\sup\limits_{t\in\mathbb{R}}\parallel \mathcal{S}_0(t+1,t)-\mathcal{S}_{\epsilon}(t+1,t)\parallel_{L^{\infty}(\mathcal{A}_{\epsilon},X)},\ \forall t\in\mathbb{R}.
   \end{equation*}
\end{theorem}

\begin{proof}
    The proof follows as in Theorem \ref{robustezams}, noting that if $\xi_{\epsilon}$ is an orbit of $\mathcal{S}_{\epsilon}$ then  $\{\xi_{\epsilon}(n)\}_{n\in\mathbb{Z}}$  is a pseudo-orbit of $\mathcal{S}_0(1,0)=\mathcal{S}_0(t+1,t)$ for all $t\in\mathbb{R}$.
\end{proof}

 We now provide our last result, which is related to continuity of global attractors. Let $(M,d)$ be a metric space and $\mathcal{T}_{\epsilon}:=\{\mathcal{T}_{\epsilon}(t):t\ge0\}_{\epsilon\in[0,1]}$ be a family of semigroups defined in $M$. Suppose that  for each $\epsilon\in[0,1]$ the semigroup $\mathcal{T}_{\epsilon}$ has a global attractor $\mathcal{A}_{\epsilon}$. In \cite{raugel2002global,carvalho2014equi} the authors show that if $\mathcal{A}_0$ attracts bounded sets of $M$ exponentially (as in \eqref{exponentialattractor}), then there exists $C>0$ and $\alpha\in (0,1)$ such that
 \begin{equation}\label{pais}
dist_H(\mathcal{A}_{\epsilon},\mathcal{A}_0)\le C\parallel \T_{\epsilon}(1)-\T_{0}(1)\parallel^{\alpha}_{L^{\infty}(\mathcal{A}_{\epsilon},X)},\ \forall \epsilon\in [0,1],
 \end{equation}
 where $dist_H$ denotes the Hausdorff semidistance (see Definition \ref{defatt}).

 In the following result we provide conditions on the semigroup $\mathcal{T}_0$ so that we can approximate the exponent $\alpha$ in \eqref{pais} to $1$. We highlight that we will not have any assumption regarding  Morse-Smale semigroups. In fact, since Lemma \ref{lemma-distancia} does not require the semigroup to be Morse-Smale,  this assumption will  not be necessary.

  \begin{theorem}\label{contfinal}
       Let $(M,d)$ be a metric space and $\{\T(t):t\ge 0\}$ be a semigroup in $M$ with global attractor $\mathcal{A}$. Let $\U\supset\mathcal{A}$ be a bounded positively invariant set that is exponentially attracted by $\mathcal{A}$ (as in Lemma \ref{lemma-distancia}) and $\mathcal{T}(1)|_{\mathcal{U}}$ is Lipschitz. Define $L_n=Lip(\T(n)|_{\mathcal{U}})$ for each $n\in\mathbb{N}$. The following properties hold:
       \begin{enumerate}
           \item  If      \begin{equation}\label{limsup}
\limsup_{n\to+\infty}\dfrac{n}{\ln{L_n}}=+\infty
       \end{equation}
       then there exists $C_3>0$ satisfying: for any $\alpha\in(0,1)$ there exists $N\in\mathbb{N}$ such that if $\{\T_{\epsilon}(t):t\ge0\}\subset\mathcal{C}(M)$ is a semigroup  with global attractor $\mathcal{A}_{\epsilon}\subset \mathcal{U}$, then
           \begin{equation*}
dist_H(\mathcal{A}_{\epsilon},\mathcal{A})\le C_3\parallel \mathcal{T}(N)-\T_{\epsilon}(N)\parallel^{\alpha}_{L^{\infty}(\mathcal{U},M)}.
           \end{equation*}
    \item If 
    \begin{equation*}
\lim_{n\to+\infty}\dfrac{n}{\ln{L_n}}=+\infty
       \end{equation*}
    then there exists $C_3>0$ satisfying: for any $\alpha\in(0,1)$ there exists $N\in\mathbb{N}$ such that if $\{\T_{\epsilon}(t):t\ge0\}\subset\mathcal{C}(M)$ is a semigroup  with global attractor $\mathcal{A}_{\epsilon}\subset \mathcal{U}$ then
    \begin{equation*}
dist_H(\mathcal{A}_{\epsilon},\mathcal{A})\le C_3\inf\limits_{n\ge N}\parallel \mathcal{T}(n)-\T_{\epsilon}(n)\parallel^{\alpha}_{L^{\infty}(\mathcal{U},M)}.
    \end{equation*}
    \item If \eqref{limsup} holds and $\mathcal{T}(1)|_{\mathcal{A}}:\mathcal{A}\to\mathcal{A}$ has the Lipschitz Shadowing property, then given $\alpha\in(0,1)$ there exists $C_3>0$ satisfying: for any semigroup $\{\T_{\epsilon}(t):t\ge0\}\subset\mathcal{M}$ with a global attractor $\mathcal{A}_{\epsilon}\subset \mathcal{U}$ we have
    \begin{equation*}
dist_H(\mathcal{A}_{\epsilon},\mathcal{A})\le C_3\parallel \mathcal{T}(1)-\mathcal{T}_{\epsilon}(1)\parallel^{\alpha}_{L^{\infty}(\mathcal{U},M)}.
    \end{equation*}
       \end{enumerate}
   \end{theorem}
   \begin{proof}
Proceeding as in the proof of Theorem \ref{robustezams} we obtain that the bounded global orbits $\xi_{\epsilon}$ of $\mathcal{T}_{\epsilon}$ are close to the bounded global orbits $\xi$ of $\mathcal{T}$. Putting this with Proposition \ref{characatt} we obtain
\begin{equation*}
dist_H(\mathcal{A}_{\epsilon},\mathcal{A}_0)\le C_2 \parallel\mathcal{T}(1)-\mathcal{T}_{\epsilon}(1)\parallel^{\alpha}_{L^{\infty}(\mathcal{U},M)},
\end{equation*}
where   $C_2$ and $\alpha$ are given in the proof of Lemma \ref{lemma-distancia} 
       \begin{equation}\label{lasttime}
    C_2=L_1^{\frac{\ln{C}}{\gamma}}\dfrac{L_1}{L_1-1}+1\ \ \text{and}\ \ \alpha=\min\left\{(1-\beta),\frac{\beta\gamma}{\ln{L_1}}\right\}.
       \end{equation}
       
  Given $n\in\mathbb{N}$ we can define the new semigroup $\mathcal{T}_n=\{\T(nt):t\ge0\}$. Note that $Lip(\mathcal{T}_n(1))=L_n$ and the global attractor of $\mathcal{T}_n$ is $\mathcal{A}$, with constants of attraction (see Lemma \ref{exponentialattractor})  $C$ and $n\gamma$. 

  Since
  \begin{equation*}
     \sup_{n\in\mathbb{N}} \left\{L_n^{\frac{\ln{C}}{n\gamma}}\dfrac{L_n}{L_n-1}\right\}\le L_1^{\frac{\ln{C}}{\gamma}} \sup_{n\in\mathbb{N}}\left\{\dfrac{L_1^n}{L_1^n-1}\right\}<+\infty
  \end{equation*}
  then  there exists $C_3>0$ such that
 \begin{equation*}
      \sup_{n\in\mathbb{N}} \left\{L_n^{\frac{\ln{C}}{n\gamma}}\dfrac{L_n}{L_n-1}+1\right\}\le C_3.
 \end{equation*}
 
 Given $\beta\in(0,1)$ take $N\in\mathbb{N}$ such that
 \begin{equation*}
     \dfrac{\beta N\gamma}{\ln{L_N}}>1-\beta.
 \end{equation*}
 
 It follows from \eqref{lasttime} that if $\{x_n\}_{n\in\mathbb{Z}}\subset \mathcal{U}$ is a $d-$pseudo orbit of $\mathcal{T}_N$, then 
 \begin{equation*}
d(x_n,\mathcal{A})\le C_3d^{1-\beta},\ \forall n\in\mathbb{Z}.
 \end{equation*}
 
Since $\beta\in(0,1)$ was arbitrary, we proved items (1) and (2). To prove item (3) we can proceed as in the proof of items (1) and (2) to estimate the distance between pseudo orbits of $\mathcal{T}(n)|_{\mathcal{U}}$ and the global attractor $\mathcal{A}$, and then use Lemmas \ref{lemaimportante} and \ref{auxiliar1} to estimate the distance between  pseudo orbits of $\mathcal{T}(1)|_{\mathcal{U}}$ and $\mathcal{A}$.
   \end{proof}

Note that Theorem \ref{contfinal} requires that the semigroup  attract bounded sets at an exponential rate but the Lipschitz constant of the semigroup does not grow exponentially on a neighborhood of the global attractor. Hence, the global attractor must have an exponential attraction without having any unstable hyperbolic equilibria ($x^*$ is hyperbolic and $W^{u}(x^*)\neq x^*$), otherwise  the Lipschitz constant of the semigroup would grow exponentially. In \cite{carvalho2011exponential} the authors provide results that allow us to obtain semigroups with  global attractors that attracts bounded sets at an exponential rate, but have no unstable hyperbolic equilibria. In fact,  to obtain exponential attraction it is sufficient to have hyperbolicity for the  "stable part" of the equilibria. Let us provide an example based in \cite{carvalho2011exponential}. In the same reference the reader can find  more examples, many of them originated from bifurcations.

Consider the ODE in $\mathbb{R}^n$
     \begin{equation}\label{origiode}
     \dot{u}=-(1-\parallel u\parallel^2)^3(2-\parallel u\parallel)u.
    \end{equation}

Multiplying both sides by $u$ we have
\begin{equation*}
    \dfrac12\dfrac{d}{dt}\parallel u\parallel^2=-(1-\parallel u\parallel^2)^3(2-\parallel u\parallel)\parallel u\parallel^2
\end{equation*}
and denoting $\rho=\parallel u\parallel$ we have the scalar ODE
\begin{equation}\label{scalarode}
    \dfrac12\dfrac{d}{dt}\rho^2=-(1-\rho^2)^3(2-\rho)\rho^2\Longrightarrow\dot{\rho}=-(1-\rho^2)^3(2-\rho)\rho.
\end{equation}

The solutions of equation \eqref{scalarode}  generate a Gradient semigroup $\mathcal{T}_{\rho}$ (satisfies item (1) from Lemma \ref{lemma-Lyapunov-function})  with global attractor $\mathcal{A}_{\rho}=[0,2]$. Moreover, the solutions of \eqref{origiode} generate a Gradient semigroup $\mathcal{T}$ with global attractor $\mathcal{A}=B_{\mathbb{R}^n}[0,2]=\{x\in\mathbb{R}^n:\parallel x\parallel\le2\}$. From a straightforward computation we obtain that the equilibria $\{0,2\}$ of $\mathcal{T}_{\rho}$ are stable  and the equilibrium point $\{1\}$ is unstable and not hyperbolic.

Note that if $\rho>2$, then
\begin{equation*}
    \dfrac{d}{dt}(\rho-2)\le -54(\rho-2)
\end{equation*}
and consequently
\begin{equation*}
    d(u(t),\mathcal{A})\le \max\{0,\parallel u(0)\parallel-2\}e^{-54t},\ \forall t\ge0.
\end{equation*}

Therefore $\mathcal{A}$ attracts points exponentially and from \cite{carvalho2011exponential}  $\mathcal{A}$ attracts bounded sets of $\mathbb{R}^n$ exponentially. Moreover, since $\rho^*=1$ is not a hyperbolic equilibrium point of $\mathcal{T}_{\rho}$, the Lipschitz constant $\tilde{L}_n$ of $\mathcal{T}(n)|_{\mathcal{A}}$  does not grow exponentially. Therefore, for any positively invariant bounded neighborhood $\mathcal{U}\supset\mathcal{A}$, the sequence $L_n=Lip(\mathcal{T}(n)|_{\mathcal{U}})$ does also not grows exponentially. Hence, the semigroup $\T$ satisfies the assumptions from Theorem \ref{contfinal}.

\section{Appendix}\label{appendix}
We include in this appendix a collection of concepts and results that have been used in the proofs of the main results of the paper.
 
\begin{definition}\label{homoclinic}
Let $(M,d)$ be a metric space and  $\mathcal{T}=\{\mathcal{T}(t):t\ge0\}$ be a  semigroup in  $M$ with global attractor $   \mathcal{A}$ and set of equilibria $\mathcal{E}$. A \textbf{homoclinic structure} in $\mathcal{A}$ is a subset $\{x_{l_1}^*,\dots,x^*_{l_k}\}\subset\mathcal{E}$ and a set of non constant global solution $\{\xi_i:\mathbb{R}\to M\}_{i=1,\dots,k}$, satisfying
\begin{equation*}
 \lim\limits_{t\to-\infty}\xi_i(t)=x_{l_i}^*\ \ \ \text{and}\ \ \ \lim\limits_{t\to+\infty}\xi_i(t)=x_{l_{i+1}}^*,\ \ i=1,\dots,k
\end{equation*}
where we define $x_{l_{k+1}}^*=x_{l_{1}}^*$.
\end{definition}
\begin{lemma}\label{birkhoff}
Let $(M,d)$ be a metric space and $\mathcal{T}=\{\mathcal{T}(t):t\ge0\}$ be a semigroup in $M$ with global attractor $\mathcal{A}$ and set of equilibria $\mathcal{E}=\{x^*_1,\dots,x^*_p\}$. Suppose that $\mathcal{T}$ satisfies item (2) of Lemma \ref{lemma-Lyapunov-function}. Then for each $\epsilon>0$ and bounded set $B\subset M$ there exists $t_0>0$, called Birkhoff number of $B$, such that
\begin{equation*}
    \{\T(t)x:t\in[0,t_0]\}\cap\left(\bigcup_{j=1}^p \mathcal{O}_{\epsilon}(x^*_j)\right)\neq\emptyset,\ \forall x\in B.
\end{equation*}
\end{lemma}
\begin{proof}
    See \cite[Lemma 3.13]{Bortolan-Carvalho-Langa-book}.
\end{proof}
\begin{lemma}\label{sainaoentra}
Let $(M,d)$ be a metric space and $\mathcal{T}=\{\mathcal{T}(t) : t \ge0\}$ be a semigroup in $M$ with global attractor $\mathcal{A}$ and set of equilibria $\mathcal{E}=\{x_1^*,\dots,x_p^*\}$. Assume that $\mathcal{T}$ is Dynamically Gradient \cite{Bortolan-Carvalho-Langa-book}, i.e. $\mathcal{T}$ satisfies items (2) and (3) from Lemma \ref{lemma-Lyapunov-function}. Then for any $i\in\{1,\dots,p\}$ and  $\delta>0$ there exists $0<\delta'<\delta$ such that if $x\in B(x^*_i,\delta')$ and $\mathcal{T}(t_0)x\notin B(x^*_i,\delta)$ for some $t_0>0$, then
\begin{equation*}
    \mathcal{T}(t)x\notin B(x^*_i,\delta'),\ \forall t\ge t_0.
\end{equation*}
\end{lemma}
\begin{proof}
    See \cite[Lemma 3.14]{Bortolan-Carvalho-Langa-book}.
\end{proof}
\begin{proposition}\label{subsequenciasolucoes}
Let $(M,d)$ be a metric space and $\mathcal{T}=\{\mathcal{T}(t) : t \ge0\}$ be a semigroup in $M$  with global attractor $\mathcal{A}$. Let $\{u_k\}_{k \in \mathbb{N}}$ be a bounded sequence in $M$ and $\{\sigma_k\}_{k \in \mathbb{N}}$ be a sequence of positive numbers such that $\sigma_k \to \infty$ as $k\to+\infty$.  Define $J_k = \{s \in \mathbb{R} : -\sigma_k \leq s < \infty\}$ and consider the orbits $\xi_k : J_k \to M$, where $\xi_k(s) = \mathcal{T}(s + \sigma_k)u_k$, $s \in J_k$. Then, there exists a bounded global orbit $\xi : \mathbb{R} \to \mathcal{A}$ of $\mathcal{T}$ and a subsequence of $\{\xi_k\}_{k \in \mathbb{N}}$ (which we again denote by $\{\xi_k\}_{k \in \mathbb{N}}$) such that
\begin{equation}\label{eqdconv}
     \xi_k(s) \xrightarrow{k\to+\infty} \xi(s), \quad \forall s \in \mathbb{R}.
\end{equation}
\end{proposition}
\begin{proof}
    See \cite{Bortolan-Carvalho-Langa-book}.
\end{proof}
\begin{lemma}\label{projfechada}
	    Let $X$ be a normed vector space and $P,Q\in\mathcal{L}(X)$ be continuous projections such that $\parallel P-Q\parallel_{\mathcal{L}(X)} <1$. Then $R(P)$ is isomorphic to $R(Q)$.
	\end{lemma}
\begin{proof}
    
\end{proof}


\begin{lemma}\label{gj}
Let $X$ be a Hilbert space and $P,Q\in\mathcal{L}(X)$ be continuous orthogonal projections. Then
$\parallel PQ\parallel_{\mathcal{L}(X)}=\parallel QP\parallel_{\mathcal{L}(X)}$.
\end{lemma}
\begin{proof}
    Just note that $P$ and $Q$ are self-adjoint.
\end{proof}

\begin{lemma}\label{iminv}
    Let $X$ be a normed vector space and $V,W$ subspaces of $X$. Let $T:X\to X$ be an injective linear map such that $W\subset T(X)$. Then
    \begin{equation*}
        T^{-1}(V+W)=T^{-1}(V)+T^{-1}(W),
    \end{equation*}
    where $T^{-1}$ denotes the inverse image of $T$.
\end{lemma}
\begin{proof}
    The inclusion $T^{-1}(V)+T^{-1}(W)\subset T^{-1}(V+W)$ is obvious. If $x\in T^{-1}(V+W)$ then $Tx=v+w$, with $v\in V$ and $w\in W$. Since $W\subset T(E)$ and $T$ is injective we conclude that $v,w\in T(E)$ and therefore $x=T^{-1}v+T^{-1}w$.
\end{proof}

\begin{lemma}\label{lemalegalzao}
Let $X$ be a Hilbert space and $S,U$ be closed subspaces of $X$ such that  
$S\oplus U=X$. Then, using the notation defined on Remark \ref{defprosepa}, we have
\begin{equation*}
    \parallel P_{SS^{\perp}}\circ
P_{UU^{\perp}}\parallel_{\mathcal{L}(X)}<1.
\end{equation*}
\end{lemma}
\begin{proof}
If $S=\{0\}$ or $U=\{0\}$ it is obvious, so we will not consider this cases. Fix $v\in U$ with $\parallel v\parallel=1$ and note that
\begin{equation*}
    1=\parallel v\parallel=\parallel P_{US}(v)-P_{US}P_{SS^{\perp}}(v)\parallel\le \parallel P_{US}\parallel_{\mathcal{L}(X)} \parallel v-P_{SS^{\perp}}v\parallel.
\end{equation*}

Consequently,
\begin{align*}
    \dfrac{1}{ \parallel P_{US}\parallel_{\mathcal{L}(X)}^2}&\le \langle v-P_{SS^{\perp}}v,v-P_{SS^{\perp}}v\rangle\\
    &=\langle (I-P_{SS^{\perp}})v,v-P_{SS^{\perp}}v \rangle\\
    & =\langle (I-P_{SS^{\perp}})v,v\rangle\\
    &=\parallel v\parallel^2-\langle P_{SS^{\perp}}v,(I-P_{SS^{\perp}})v+P_{SS^{\perp}}v)\rangle\\
    &=1 -\parallel P_{SS^{\perp}}v\parallel^2\\
\end{align*}

Therefore $\parallel P_{US}\parallel_{\mathcal{L}(X)} \ge 1$ and 
\begin{equation*}
   \parallel P_{SS^{\perp}}v\parallel\le \sqrt{1-\dfrac{1}{ \parallel P_{US}\parallel_{\mathcal{L}(X)}^2}},\ \forall \parallel v\parallel =1, v \in U.
\end{equation*}

Since $P_{US}\in\mathcal{L}(X)$ is a non-null projection, then $\parallel P_{US}\parallel_{\mathcal{L}(X)}\ge 1$, which concludes the proof.
\end{proof}

\begin{theorem}\label{teoremaabstrato}
Let $\{E_k\}_{k\in \mathbb{Z}}$  be a sequence of Banach spaces and $\{\phi_k\}_{k\in\mathbb{Z}}$ be a sequence of maps  $\phi_k:E_k\to E_{k+1}$ such that for each $k\in\mathbb{Z}$ we have
\begin{equation*}
    \phi_k=A_k+w_{k+1},
\end{equation*}
where $A_k:E_k\to E_{k+1}$ is linear. Consider the following conditions:
\begin{enumerate}
    \item For each $k\in\mathbb{Z}$ there exist continuous projections $P_k,Q_k:E_k\to E_k$ and constants $\lambda\in(0,1), M>0$ independent of $k$ such that
    \begin{equation*}
        P_k+Q_k=Id,\ \ \ \ \parallel P_k\parallel_{\mathcal{L}(E_k)},\parallel Q_k\parallel_{\mathcal{L}(E_k)}\le M,\ \forall k\in\mathbb{Z}
    \end{equation*}
    and
    \begin{equation*}
        \parallel A_kP_k\parallel_{\mathcal{L}(E_k,E_{k+1})}\le\lambda,\ \ \ A_kP_k(E_k)\subset P_{k+1}E_{k+1},\ \forall k\in\mathbb{Z}.
    \end{equation*}
    \item There exist linear maps $B_k:Q_{k+1}E_{k+1}\to E_k$ satisfying
    \begin{equation*}
      B_kQ_{k+1}E_{k+1}\subset Q_kE_k,\ \parallel B_kQ_{k+1}\parallel_{\mathcal{L}(E_{k+1},E_k)}\le\lambda,\ A_kB_k|_{Q_{k+1}E_{k+1}}=I|_{Q_{k+1}E_{k+1}} ,\ \forall k\in\mathbb{Z}. 
    \end{equation*}
\item There exist $\Delta,k_1>0$ with the property
\begin{equation*}
    \parallel w_{k+1}v-w_{k+1}v'\parallel_{E_{k+1}}\le k_1\parallel v-v'\parallel_{E_k},\ \forall\parallel v\parallel,\parallel v'\parallel\le\Delta,\ \forall k\in\mathbb{Z}.
\end{equation*}
\item\label{item44} $k_1N_1<1$, where
\begin{equation*}
    N_1=M\left(\dfrac{1+\lambda}{1-\lambda}\right).
\end{equation*}

If items (1)-(4) are fulfilled, then there exist $d_1,L>0$ with the property: if
\begin{equation}\label{hipabstrato}
    \parallel \phi_k(0)\parallel_{E_{k+1}}\le d\le d_1,\ \forall k\in\mathbb{Z}
\end{equation}
then there exist $v_k\in E_k$ such that
\begin{equation}\label{teseabstrato}
    \parallel v_k\parallel_{E_{k}}\le Ld\ \ \text{and}\ \ \phi_k(v_k)=v_{k+1},\ \forall k\in\mathbb{Z}.
\end{equation}
\end{enumerate}
\end{theorem}
\begin{proof}
    See  \cite[Lemma 1.3.1]{Pilyugin} and \cite[Lemma A.0.8]{santamaria2014distance}.
\end{proof}

\begin{definition}\label{conteaskc}
    Let $(M,d)$ be a metric space and for each $\epsilon\in[0,1]$ let $\mathcal{T}_{\epsilon}=\{\mathcal{T}_{\epsilon}(t):t\ge0\}$ be a semigroup in $M$. We say that the family $\{\mathcal{T}_{\epsilon}\}_{\epsilon\in[0,1]}$ is:
    \begin{enumerate}
        \item \textbf{continuous} at $\epsilon=0$ if for each compact set $K\subset M$ and positive number $R>0$ it holds
        \begin{equation*}
            \sup\limits_{t\in[0,R]}\sup\limits_{x\in M}d(\mathcal{T}_{\epsilon}(t)x,\mathcal{T}_{0}(t)x)\to 0\ \text{as}\ \epsilon\to0.
        \end{equation*}
        \item \textbf{collectively asymptotically compact} at $\epsilon=0$ if given sequences of non-negative numbers $\{t_k\}_{k\in\mathbb{N}}$, with $t_k\to+\infty$ as $k\to+\infty$, and $\{\epsilon_k\}_{k\in\mathbb{N}}\subset[0,1]$ with $\epsilon_k\to 0$ as $k\to+\infty$, and a bounded sequence $\{x_k\}_{k\in\mathbb{N}}\subset M$, the sequence $\{\mathcal{T}(t_k)x_k\}_{k\in\mathbb{N}}$ has a convergent subsequence in $M$.
    \end{enumerate}
\end{definition}

\subsection*{Acknowledgments} 
JMA was partially supported by grants PID2022-137074NB-I00 and CEX2023-001347-
S ``Severo Ochoa Programme for Centres of Excellence in R\&D'' MCIN/AEI/10.13039/501100011033, the two of them from Ministerio de Ciencia e Innovaci\'on, Spain. Also by ``Grupo de Investigaci\'on UCM 920894 - CADEDIF'',

ANC was partially supported by FAPESP grant 2020/14075-6 y CNPq 308902/2023-8

CRTJ was partially supported by CAPES-PROEX-11169228/D and by FAPESP \# 2020/14353-6 and \#2022/02172-2, Brazil.


\begin{thebibliography}{9}
\bibitem{aragao2013non}Aragao-Costa, E., Caraballo, T., Carvalho, A. \& Langa, J. Non-autonomous Morse-decomposition and Lyapunov functions for gradient-like processes. {\em Transactions of the American Mathematical Society}. \textbf{365}, 5277-5312 (2013).




\bibitem{arrieta2000abstract}Arrieta, J. \& Carvalho, A. Abstract parabolic problems with critical nonlinearities and applications to Navier-Stokes and heat equations. {\em Transactions of the American Mathematical Society}. \textbf{352}, 285-310 (2000).


\bibitem{arrieta1992damped}Arrieta, J., Carvalho, A. \& Hale, J. A damped hyerbolic equation with critical exponent. {\em Communications in Partial Differential Equations}. \textbf{17}, 841-866 (1992).

\bibitem{arrieta2017distance}Arrieta, J. \& Santamar\'{ı}a, E. Distance of attractors of reaction-diffusion equations in thin domains. {\em Journal of Differential Equations}. \textbf{263}, 5459-5506 (2017).

\bibitem{Arrieta-Esperanza-2014}Arrieta, J. \& Santamaría, E. Estimates on the distance of inertial manifolds. {\em Discrete and Continuous Dynamical Systems}. \textbf{34},  3921-3944 (2014).

\bibitem{ArrietaEsperanza}Arrieta, J. \& Santamaría, E. $C^{1,\theta}$-Estimates on the distance of inertial manifolds. {\em Collectanea mathematica}. \textbf{69(3)}, 315-336 (2018).

\bibitem{bortolan2020lipschitz}Bortolan, M., Cardoso, C., Carvalho, A. \& Pires, L. Lipschitz perturbations of Morse-Smale semigroups. {\em Journal of Differential Equations}. \textbf{269}, 1904-1943 (2020).

\bibitem{Bortolan-Carvalho-Langa-book}Bortolan, M., Carvalho, A. \& Langa, J. Attractors Under Autonomous and Non-autonomous Perturbations. Vol. 246. {\em American Mathematical Society} (2020).

\bibitem{bortolan2014structure}Bortolan, M., Carvalho, A. \& Langa, J. Structure of attractors for skew product semiflows. {\em Journal of Differential Equations}. \textbf{257}, 490-522 (2014).

\bibitem{bortolan2022nonautonomous}Bortolan, M., Carvalho, A., Langa, J. \& Raugel, G. Nonautonomous Perturbations of Morse–Smale Semigroups: Stability of the Phase Diagram. {\em Journal of Dynamics and Differential Equations}.  1-67 (2022).


\bibitem{brunovsky2003genericity}Brunovsky, P. \& Raugel, G. Genericity of the Morse-Smale property for damped wave equations. {\em Journal of Dynamics and Differential Equations}. \textbf{15}, 571-658 (2003).


\bibitem{buckholtz2000hilbert}Buckholtz, D. Hilbert space idempotents and involutions. {\em Proceedings of the American Mathematical Society}. \textbf{128}, 1415-1418 (2000).


\bibitem{carvalho2008regularity}Carvalho, A. \& Cholewa, J. Regularity of solutions on the global attractor for a semilinear damped wave equation. {\em Journal of Mathematical Analysis and Applications}. \textbf{337}, 932-948 (2008).



\bibitem{carvalho2011exponential}Carvalho, A. \& Cholewa, J. Exponential global attractors for semigroups in metric spaces with applications to differential equations. {\em Ergodic Theory and Dynamical Systems}. \textbf{31}, 1641-1667 (2011).

\bibitem{carvalho2009local}Carvalho, A. \& Cholewa, J. Local well posedness, asymptotic behavior and asymptotic bootstrapping for a class of semilinear evolution equations of the second order in time. {\em Transactions of the American Mathematical Society}. \textbf{361}, 2567-2586 (2009).
\bibitem{carvalho2014equi}Carvalho, A., Cholewa, J. \& Dłotko, T. Equi-exponential attraction and rate of convergence of attractors with application to a perturbed damped wave equation. {\em Proceedings of the Royal Society of Edinburgh Section A: Mathematics}. \textbf{144}, 13-51 (2014).
\bibitem{carvalho2012attractors}Carvalho, A., Langa, J. \& Robinson, J. Attractors for infinite-dimensional non-autonomous dynamical systems. Vol. 182. {\em Springer Science \& Business Media} (2012).
\bibitem{carvalho1998boundary}Carvalho, A. \& Primo, M. Boundary synchronization in parabolic problems with nonlinear boundary conditions. {\em Dynamics of Continuous, Discrete and Impulsive Systems Series B: Application and Algorithm} \textbf{7}, 541-560 (2000).
\bibitem{chow1989shadowing}Chow, S., Lin, X. \& Palmer, K. A shadowing lemma with applications to semilinear parabolic equations. {\em SIAM Journal on Mathematical Analysis}. \textbf{20}, 547-557 (1989).

\bibitem{da2010dirichlet}Da Prato, G. \& Lunardi, A. On the Dirichlet semigroup for Ornstein–Uhlenbeck operators in subsets of Hilbert spaces. {\em Journal of Functional Analysis}. \textbf{259}, 2642-2672 (2010).
\bibitem{desheng2004equi}Desheng, L. \& Kloeden, P. Equi-attraction and the continuous dependence of attractors on parameters. {\em Glasgow Mathematical Journal}. \textbf{46}, 131-141 (2004).
\bibitem{foias1988inertial}Foias, C., Sell, G. \& Temam, R. Inertial manifolds for nonlinear evolutionary equations. {\em Journal of Differential Equations}. \textbf{73}, 309-353 (1988).

\bibitem{gan2002generalized}Gan, S. A generalized shadowing lemma. {\em Discrete and Continuous Dynamical Systems}. \textbf{8}, 627-632 (2002).

\bibitem{hale2010asymptotic}Hale, J. Asymptotic behavior of dissipative systems. No.25. \textit{American Mathematical Society} (2010).
\bibitem{hale2009ordinary}Hale, J. Ordinary differential equations. \textit{Courier Corporation} (2009).

\bibitem{hale2006dynamics}Hale, J., Magalhães, L. \& Oliva, W. Dynamics in infinite dimensions. Vol 47. \textit{Springer Science \& Business Media} (2006).


\bibitem{hale1989lower}Hale, J. \& Raugel, G. Lower semicontinuity of attractors of gradient systems and applications. {\em Annali di Matematica Pura ed Applicata}. \textbf{154}, 281-326 (1989).

\bibitem{hale1990lower}Hale, J. \& Raugel, G. Lower semicontinuity of the attractor for a singularly perturbed hyperbolic equation. {\em Journal of Dynamics And Differential Equations}. \textbf{2}, 19-67 (1990).

\bibitem{henry1994exponential}Henry, D. Exponential dichotomies, the shadowing lemma and homoclinic orbits in Banach spaces. {\em Resenhas do Instituto de Matemática e Estat\'{i}stica da Universidade de São Paulo}. \textbf{1}, 381-401 (1994).

\bibitem{henry2006geometric}Henry, D. Geometric theory of semilinear parabolic equations. Vol 840. \textit{Springer} (2006).

\bibitem{hoang2015continuity}Hoang, L., Olson, E. \& Robinson, J. On the continuity of global attractors. {\em Proceedings of the American Mathematical Society}. \textbf{143}, 4389-4395 (2015).

\bibitem{jolly1989explicit}Jolly, M. Explicit construction of an inertial manifold for a reaction diffusion equation. {\em Journal of Differential Equations}. \textbf{78}, 220-261 (1989).



\bibitem{komura1967nonlinear}Komura, Y. Nonlinear semi-groups in Hilbert space. {\em Journal of The Mathematical Society Of Japan}. \textbf{19}, 493-507 (1967).


\bibitem{kupka1963contributiona}Kupka, I. Contribution a la théorie des champs génériques. {\em Contributions to Differential Equations}. \textbf{2}, 457-484 (1963).

\bibitem{lee2021diffeomorphisms}Lee, M., Oh, J. \& Wen, X. Diffeomorphisms with a generalized Lipschitz shadowing property. {\em Discrete \& Continuous Dynamical Systems: Series A}. \textbf{41} (2021).
\bibitem{li2003chaos}Li, Y. Chaos and shadowing lemma for autonomous systems of infinite dimensions. {\em Journal of Dynamics and Differential Equations}. \textbf{15}, 699-730 (2003).

\bibitem{palis1969morse}Palis, J. On morse-smale dynamical systems. \textit{Topology}. \textbf{8}, 385-404 (1969).

\bibitem{palis2012geometric}Palis, J. \& De Melo, W. Geometric theory of dynamical systems: an introduction. \textit{Springer Science \& Business Media} (2012).
\bibitem{palmer2012lipschitz}Palmer, K., Pilyugin, S. \& Tikhomirov, S. Lipschitz shadowing and structural stability of flows. {\em Journal of Differential Equations}. \textbf{252}, 1723-1747 (2012).
\bibitem{pazy2012semigroups}Pazy, A. Semigroups of linear operators and applications to partial differential equations. Vol. 44. \textit{Springer Science \& Business Media} (2012).
\bibitem{PEIXOTO1967214}Peixoto, M. On an approximation theorem of Kupka and Smale. {\em Journal of Differential Equations}. \textbf{3}, 214-227 (1967).


\bibitem{Pilyugin}Pilyugin, S.  Shadowing in Dynamical Systems. \textit{Springer-Verlag} (1999).

\bibitem{pilyugin2017shadowing}Pilyugin, S. \& Sakai, K. Shadowing and hyperbolicity. \textit{Springer} (2017).
\bibitem{pilyugin2010lipschitz}Pilyugin, S. \& Tikhomirov, S. Lipschitz shadowing implies structural stability. {\em Nonlinearity}. \textbf{23}, 2509 (2010).
\bibitem{raugel1989continuity}Raugel, G. Continuity of attractors. {\em ESAIM: Mathematical Modelling and Numerical Analysis}. \textbf{23}, 519-533 (1989).
\bibitem{raugel2002global}Raugel, G. Global attractors in partial differential equations. {\em Handbook of Dynamical Systems}. \textbf{2}, 885-982 (2002).

\bibitem{robbin1971structural}Robbin, J. A structural stability theorem. {\em Annals of Mathematics}. \textbf{94}, 447-493 (1971).
\bibitem{robinson1974structural}Robinson, C. Structural stability of vector fields. {\em Annals of Mathematics}. \textbf{99}, 154-175 (1974).

\bibitem{robinson2002computing}Robinson, J. Computing inertial manifolds. {\em Discrete and Continuous Dynamical Systems}. \textbf{8}, 815-834 (2002).
\bibitem{robinson2001infinite}Robinson, J. Infinite-dimensional dynamical systems: an introduction to dissipative parabolic PDEs and the theory of global attractors. \textit{Cambridge University Press} (2001).
\bibitem{rosa1998global}Rosa, R. The global attractor for the 2D Navier-Stokes flow on some unbounded domains. {\em Nonlinear Analysis}. \textbf{32}, 71-86 (1998).
\bibitem{rozendaal2019optimal}Rozendaal, J., Seifert, D. \& Stahn, R. Optimal rates of decay for operator semigroups on Hilbert spaces. {\em Advances in Mathematics}. \textbf{346}, 359-388 (2019).



\bibitem{santamaria2014distance}Santamaria, E. Distance of attractors of evolutionary equations. \textit{Ph. D. thesis, Universidad Complutense de Madrid} (2014).


\bibitem{sell2013dynamics}Sell, G. R., \& You, Y. Dynamics of evolutionary equations. Vol. 143. \textit{Springer Science \& Business Media} (2013).


\bibitem{shub2013global}Shub, M. Global stability of dynamical systems.     \textit {Springer Science \& Business Media} (2013).

\bibitem{smale1963stable}Smale, S. Stable manifolds for differential equations and diffeomorphisms. {\em Annali della Scuola Normale Superiore di Pisa-Scienze Fisiche e Matematiche}. \textbf{17}, 97-116 (1963).


\bibitem{tan1993asymptotic}Tan, K. \& Xu, H. Asymptotic behavior of almost-orbits of nonlinear semigroups of non-Lipschitzian mappings in Hilbert spaces. {\em Proceedings of the American Mathematical Society}. \textbf{117}, 385-393 (1993).





\bibitem{temam2012infinite}Temam, R. Infinite-dimensional dynamical systems in mechanics and physics. Vol 68. \textit{Springer Science \& Business Media} (2012).

\bibitem{tikhomirov2015holder}Tikhomirov, S. Hölder shadowing on finite intervals. \textit{Ergodic Theory and Dynamical Systems}. \textbf{35(6)}, 2000-2016 (2015).

\bibitem{viana2016foundations}Viana, M. Foundations of Ergodic Theory. No. 151. \textit{Cambridge University Press} (2016).
\end{thebibliography}
\end{document}